\documentclass[reqno]{amsart}
\usepackage{graphicx,amsmath,amsfonts,amsthm,amssymb,paralist}

\newtheorem{theorem}{Theorem}[section]
\newtheorem{corollary}[theorem]{Corollary}
\newtheorem{lemma}[theorem]{Lemma}
\newtheorem{proposition}[theorem]{Proposition}
\newtheorem{assumption}{Assumption}

\theoremstyle{definition} \newtheorem{definition}[theorem]{Definition}
\theoremstyle{remark} \newtheorem{remark}[theorem]{Remark}
\numberwithin{equation}{section}

\newcommand{\g}{\geqslant} \newcommand{\RR}{\mathbb{R}}
\newcommand{\ZZ}{\mathbb{Z}} \newcommand{\CC}{\mathbb{C}}
\newcommand{\s}{\mathcal{S}} 
\newcommand{\sph}{\mathbb{S}} \newcommand{\NN}{\mathbb{N}}
\newcommand{\p}{\partial} \newcommand{\q}{\varphi}
\newcommand{\les}{\leqslant} \newcommand{\lesa}{\lesssim}
\newcommand{\mc}[1]{\mathcal{#1}} \newcommand{\mb}[1]{\mathbf{#1}}
 \newcommand{\lr}[1]{ \langle #1
  \rangle} \newcommand{\ind}{\mathbbold{1}}

\DeclareSymbolFont{bbold}{U}{bbold}{m}{n}
\DeclareSymbolFontAlphabet{\mathbbold}{bbold}

\setdefaultenum{(i)}{(a)}{1}{A}

 \DeclareMathOperator*{\supp}{supp}
\DeclareMathOperator*{\diam}{diam} \DeclareMathOperator*{\diag}{diag}
\DeclareMathOperator*{\dist}{dist}

\usepackage{fullpage}

\begin{document}

\title[Transference of Bilinear Restriction Estimates and the
DKG-System]{Transference of Bilinear Restriction Estimates to
  Quadratic Variation Norms and the Dirac-Klein-Gordon
  System}%

\author[T.~Candy]{Timothy Candy}%
\address[T.~Candy]{Universit\"at Bielefeld, Fakult\"at f\"ur Mathematik,
  Postfach 100131, 33501 Bielefeld, Germany}%
\email{tcandy@math.uni-bielefeld.de}%

\author[S.~Herr]{Sebastian Herr}%
\address[S.~Herr]{Universit\"at Bielefeld, Fakult\"at f\"ur Mathematik,
  Postfach 100131, 33501 Bielefeld, Germany}%
\email{herr@math.uni-bielefeld.de}%
\thanks{The authors acknowledge support from the German Research
  Foundation via Collaborative Research Center 701.}


\begin{abstract}
  Firstly, bilinear Fourier Restriction estimates --which are
  well-known for free waves-- are extended to adapted spaces of
  functions of bounded quadratic variation, under quantitative
  assumptions on the phase functions. This has applications to
  nonlinear dispersive equations, in particular in the presence of
  resonances. Secondly, critical global well-posedness and
  scattering results for massive Dirac-Klein-Gordon systems in
  dimension three are obtained, in resonant as well as in
  non-resonant regimes. The results apply to small initial data in
  scale-invariant Sobolev spaces exhibiting a small amount of
  angular regularity.
\end{abstract}

\maketitle

\section{Introduction}\label{sect:intro}

The Fourier restriction conjecture was shaped in the 1970s by work of
Stein, among others, and has generated significant advances
in the field of harmonic analysis and dispersive partial differential
equations since then, see e.g. \cite{Stein1993,Tao2004} for a survey
and references.

As an example, let $n\g 2$ and $C$ be a compact subset of the cone,
say $C=\{(|\xi|,\xi) \mid \frac{1}{2}\les |\xi|\les 2\}\subset \RR^{n+1}$, and $g$ be
a Schwartz function on $\RR^{n+1}$.  Equivalently to the Fourier
restriction operator $\mathcal{R}: g \mapsto \widehat{g}|_C$, consider
its adjoint, the Fourier extension operator
\[
\mathcal{E} f (t,x)=\int_{\RR^n}e^{-i(t,x)\cdot( |\xi|,\xi)}f
(\xi)d\xi,
\]
for smooth $f$ with $\supp(f)$ contained in the unit annulus. The function
$\mathcal{E}f$ can be viewed as the inverse Fourier transform of a
surface-measure supported on the cone $C$, and  defines a
function on $\RR^{n+1}$ which solves the wave equation. The Fourier
restriction conjecture for the cone is equivalent to establishing the
corresponding Fourier extension estimate
$$ \| \mathcal{E}f \|_{L^p_{t, x}(\RR^{n+1})}\lesa\| f\|_{L^q(\RR^n)}$$
within the optimal range of $p,q$. In the special case $q=2$ this holds
iff $p\geq \frac{2n+2}{n-1}$, and in the literature on dispersive
equations this is stated as
$$ \| e^{-it|\nabla|}f \|_{L^p_{t, x}(\RR^{n+1})}\lesa\| f\|_{L^2_x}$$
and called a Strichartz estimate for the wave equation
\cite{Strichartz1977}, see also \cite{Keel1998} for more information.

In the course of proving Fourier extension estimates for the cone, it became apparent that a key role was played by bilinear estimates. Indeed, a major breakthrough was achieved by Wolff
\cite{Wolff2001}, when he proved that for every $p>\frac{n+3}{n+1}$, $n\g 2$, we
have
$$ \big\| e^{ - it |\nabla|} f e^{ - i t |\nabla|} g \big\|_{L^p_{t, x}(\RR^{n+1})} \lesa \| f \|_{L^2_x} \| g \|_{L^2_x},$$
provided the supports of $\widehat{f}$ and $\widehat{g}$ are angularly
separated and contained in the unit annulus. As a
result Wolff was able to prove the linear restriction conjecture for $C$
in dimensions $n=3$. It is important to note that, in the presence of angular
separation, a larger set of $p$ can be covered in the bilinear
estimate than would follow from a simple application of H\"older's
inequality together with the linear estimates.

In parallel to these developments, bilinear estimates proved useful in
the context of nonlinear dispersive equations, see
e.g. \cite{Klainerman1993,Bourgain1998,Foschi2000}. The perturbative
approach to dispersive equations is based on constructing adapted
function spaces in which nonlinear terms can be
effectively estimated. Bilinear estimates for solutions to the
homogeneous equation, which go beyond simple almost orthogonality
considerations, give precise control over dynamic interactions of products of
linear solutions. However, to apply these homogeneous estimates to the nonlinear problem, necessitates
the transfer of such genuinely bilinear estimates to adapted function
spaces.

Such a \emph{Transference Principle} was implemented first in
$X^{s,b}$ spaces, see \cite[Lemma 2.3]{Ginibre1997} and
\cite[Proposition 3.7]{Klainerman2002}. Let us briefly illustrate it
by looking at the following example. Suppose that $u,v \in L^\infty_t
L^2_x$ are superpositions of modulated solutions of the homogeneous
equation, i.e.
\[
u(t)=\int_{\RR} e^{it\lambda} e^{it|\nabla|} F_\lambda d\lambda,\quad
v(t)=\int_{\RR} e^{it\lambda'} e^{it|\nabla|} G_{\lambda'} d\lambda'.
\]
which is true for $u,v\in X^{0,b}$ if $b>\frac12$. Suppose in
addition, that the spatial Fourier supports of $u,v$ are angularly
separated. Then, for any $p>\frac{n+3}{n+1}$, Wolff's estimate
transfers to
\[
\big\| u v \big\|_{L^p_{t, x}(\RR^{n+1})} \leq \int_{\RR}\int_{\RR}
\|e^{it|\nabla|} F_\lambda e^{it|\nabla|}G_{\lambda'}\|_{L^p_{t,
    x}(\RR^{n+1})} d\lambda d\lambda' \lesa \Big(\int_\RR
\|F_\lambda\|_{L^2_x}d\lambda\Big)\Big(\int_\RR
\|G_{\lambda'}\|_{L^2_x}d\lambda'\Big)
\]
which is equivalent to the bilinear estimate holding for functions in $X^{0, b}$.
Another strategy involves certain atomic function spaces introduced in
\cite{Koch2005}. Suppose that
\[
u(t)=\sum_{J\in \mc{I}} \ind_J(t) e^{it|\nabla|} f_J, \quad
v(t)=\sum_{J'\in \mc{I'}} \ind_{J'}(t) e^{it|\nabla|} g_{J'}.
\]
for finite partitions $\mc{I},\mc{I'}$ of $\RR$ and $f_J,g_{J'}\in
L^2_x$. Then, under the above angular separation assumption, Wolff's
bound implies
\[
\big\| u v \big\|_{L^p_{t, x}(\RR^{n+1})} \leq \Big(\sum_{J\in
  \mc{I}}\sum_{J'\in \mc{I'}} \|e^{it|\nabla|} f_J
e^{it|\nabla|}g_{J'}\|_{L^p_{t, x}(\RR^{n+1})}^p \Big)^{\frac1p}\lesa
\Big(\sum_{J\in \mc{I}}
\|f_J\|_{L^2_x}^p\Big)^{\frac1p}\Big(\sum_{J'\in \mc{I'}}
\|g_{J'}\|_{L^2_x}^p\Big)^{\frac1p}.
\]
As a consequence, we deduce that Wolff's bilinear estimate holds for angularly separated  functions in the atomic space $U^p$, see Definition \ref{def:up} below. This is one instance of the transference principle in $U^p$, which has been formalised in
\cite[Proposition 2.19]{Hadac2009}.

For many applications, the above superposition requirements are too
strong, partly due to the duality theory for the spaces $X^{0,b}$ for
$b>\frac12$ and $U^p$ for $p\leq 2$.  Nevertheless, variations of the
above strategies have been successfully employed in numerous
applications to nonlinear global-in-time problems in the case $p\geq
2$. In the case $p<2$, the only result we are aware of is \cite[Lemma
5.7 and its proof]{Sterbenz2010}, where this approach is used in conjunction
with an interpolation argument to give a partial result only, see
Remark \ref{rem:inter} for further details.

It turned out that one of the most powerful function spaces in the
context of adapted function spaces, is the space of functions of bounded quadratic variation
$V^2$, which is slightly bigger than $U^2$. Our first main result
of this paper is the corresponding transference principle in $V^2$ for a quite
general class of surfaces in Theorem \ref{thm - bilinear Lp estimate}
below.

We start with some definitions. Define $\mc{Z} = \{ (t_j)_{j \in \ZZ}
\mid t_j \in \RR \text{ and } t_j<t_{j+1} \}$ to be the set of
increasing sequences of real numbers and $1\leq p<\infty$. Given a
function $\rho: \RR \rightarrow L^2_x$, we define the $p$-variation of
$\rho$ to be
 $$ | \rho |_{V^p} = \sup_{(t_j) \in \mc{Z}}
 \Big( \sum_{j \in \ZZ} \| \rho(t_j) - \rho(t_{j-1}) \|_{L^2_x}^p
 \Big)^\frac{1}{p}.$$ The Banach space $V^p$ is then defined to be all
 right continuous functions $\rho: \RR \rightarrow L^2_x$ such that
 $$ \| \rho \|_{V^p} = \| \rho \|_{L^\infty_t L^2_x} + | \rho |_{V^p} < \infty.$$
 Given a phase $\Phi: \RR^n \rightarrow \RR$ we let $V^p_{\Phi}$
 denote the space of all functions $u$ such that $e^{ - i t \Phi(- i
   \nabla)} u \in V^p$ equipped with the
 obvious norm $\| u \|_{V^p_\Phi} = \big\| e^{ - i t \Phi(- i \nabla)}
 u \big\|_{V^p}$. In other words, the space $V^p_{\Phi}$ contains all functions $u\in
 L^\infty_t L^2_x$ such that the pull-back along the linear flow has
 bounded $p$-variation, in particular we have
 \[
 \| e^{ i t \Phi(- i \nabla)} f \|_{V^p_\Phi}=\|f\|_{L^2_x}.
 \]

 Before stating Theorem \ref{thm - bilinear Lp estimate}, we need to
 introduce the assumptions that we impose on our phases, which are
 motivated by \cite{Lee2010,Bejenaru2016}. Examples will be discussed
 in Section \ref{sect:assu}.  Let $\Phi_j: \RR^n \rightarrow \RR$ and
 $\Lambda_j$ be a convex subset of $\{\frac{1}{16} \les |\xi| \les 16
 \}$. Given $\mathfrak{h} = (a, h) \in \RR^{1+n}$ and $\{j, k\} = \{1,
 2\}$ we define the hypersurfaces
 \[\Sigma_j(\mathfrak{h}) = \{ \xi \in \Lambda_j \cap (
 \Lambda_k + h) \mid \Phi_j(\xi) = \Phi_k(\xi-h) + a \}.\] With this
 notation, we are ready to state the main assumption, cp. \cite{Bejenaru2016,Lee2010}.
 \begin{assumption}[Transversality/Curvature/Regularity] \label{assump on phase} There exist $\mb{D}_1, \mb{D}_2>0$ and $N \in
   \NN$ such that for $\Phi_1,\Phi_2:\RR^n \to \RR$ the following
   holds true:
   \begin{enumerate}
   \item\label{it:ass1} for every $\{j, k\} = \{1, 2\}$, $\mathfrak{h} \in
     \RR^{1+n}$, $\xi, \xi' \in \Sigma_j(\mathfrak{h})$, and $\eta \in
     \Lambda_k$ we have the estimate $$ \big|\big(\nabla \Phi_j(\xi) -
     \nabla \Phi_j(\xi')\big)\wedge \big( \nabla \Phi_j(\xi) - \nabla
     \Phi_k(\eta)\big)| \g \mb{D}_1 |\xi - \xi'|,$$
   \item\label{it:ass2} we have $\Phi_j\in C^N(\Lambda_j)$ with the derivative
     bound $$ \sup_{1\les |\kappa| \les N} \| \p^\kappa \Phi_j
     \|_{L^\infty(\Lambda_j)} \les \mb{D}_2.$$
   \end{enumerate}
 \end{assumption}

The condition \eqref{it:ass1} in Assumption \ref{assump on phase} is somewhat
difficult to interpret, but one immediate consequence is the bound
\begin{equation}\label{eqn - seperation of velocities in Sigma}
  |\nabla \Phi_j(\xi) - \nabla \Phi_j(\xi') | \g \frac{ \mb{D}_1|\xi -
    \xi'|}{\| \nabla \Phi_1\|_{L^\infty} + \| \nabla
    \Phi_2\|_{L^\infty}} .
\end{equation}
which holds for every $\xi, \xi' \in \Sigma_j(\mathfrak{h})$. To some
extent, this is a \emph{curvature} condition, as it shows that the
normal direction varies on $\Sigma_j(\mathfrak{h})$. Another consequence of \eqref{it:ass1} is that for every $\xi\in \Lambda_1$, $\eta \in \Lambda_2$ we have the \emph{transversality} bound
\begin{equation}\label{eq:trans-main}
  |\nabla \Phi_1(\xi) - \nabla \Phi_2(\eta) | \g \frac{ \mb{D}_1}{\min\{\| \nabla^2 \Phi_1\|_{L^\infty}, \| \nabla^2
    \Phi_2\|_{L^\infty}\}}. 
\end{equation}
This follows by simply observing that for every $\xi \in \Lambda_1$ there is $\mathfrak{h}\in \RR^{1+n}$ such that $\xi \in \Sigma_1(\mathfrak{h})$.
 Our first main result can now be stated as follows.
 \begin{theorem}\label{thm - bilinear Lp estimate}
   Let $n\g 2$, $p>\frac{n+3}{n+1}$, and $\mb{D}_1,
   \mb{D}_2, \mb{R}_0 >0$. For $j=1, 2$, let $\Lambda_j, \Lambda_j^*
   \subset \{ \frac{1}{16} \les |\xi| \les 16\}$ with $\Lambda_j$
   convex and $\Lambda^*_j + \frac{1}{\mb{R}_0} \subset
   \Lambda_j$. There exists $N \in \NN$ and a constant $C>0$ such
   that, for any phases $\Phi_1$ and $\Phi_2$ satisfying Assumption
   \ref{assump on phase}, and any $u \in V^2_{\Phi_1}$, $v \in
   V^2_{\Phi_2}$ with $\supp \widehat{u}(t) \subset \Lambda^*_1$,
   $\supp \widehat{v}(t) \subset \Lambda^*_2$, we have
   \[\|u v \|_{L^p_{t, x}(\RR^{1+n})} \les C \| u \|_{V^2_{\Phi_1}} \|
   v \|_{V^2_{\Phi_2}}.\]
 \end{theorem}

 Note that the constants $N$ and $C$ depend on the parameters
 $p>\frac{n+3}{n+1}$, $n\g 2$, and $ \mb{D}_1, \mb{D}_2,
 \mb{R}_0 >0$, but are otherwise independent of the phase $\Phi_j$,
 the sets $\Lambda_j$, $\Lambda_j^*$, and the functions $u$ and
 $v$. Moreover, as the conditions in Assumption \ref{assump on phase}
 are invariant under translations, the condition that $\Lambda_j
 \subset \{ \frac{1}{16}\les |\xi| \les 16\}$ can be replaced with the
 condition that the sets $\Lambda_j$ are simply contained in balls of
 radius $16$. In other words, the \emph{location} of the sets
 $\Lambda_j$ plays no role. We refer the reader to Corollary
 \ref{cor:mixed} for a generalisation of Theorem \ref{thm -
   bilinear Lp estimate} to mixed norms. Further, we refer to Corollary \ref{cor - small scale
   bilinear estimate} for a generalisation to more general frequency scales in the case
 of hyperboloids, which is also shown to be sharp.

 Let us summarize the developments for solutions to the homogeneous
 equation, i.e.
 \[u=e^{it\Phi_1(-i\nabla)}f, \quad v=e^{it\Phi_2(-i\nabla)}g.\] First
 estimates of this type for nontrivial $p<2$ are due to Bourgain
 \cite{Bourgain1991,Bourgain1995} in the case of the cone,
 i.e. $\Phi_1(\xi)=\Phi_2(\xi)=|\xi|$. Subsequently, these have been
 improved by Tao--Vargas--Vega \cite{Tao1998}, Moyua--Vargas--Vega
 \cite{Moyua1999}, Tao--Vargas \cite{Tao2000a}, before finally Tao
 \cite{Tao2001} proved the endpoint case $p=\frac{n+3}{n+1}$, see also
 Remark \ref{rem:proof-endpoint}. Actually, we observe that the vector-valued inequality in \cite{Tao2001}  is strong enough to deduce the estimate in $U^2$ in the case of the wave equation, see Remark \ref{rem:vec}.
 Related estimates for null-forms
 have been proved by Tao--Vargas \cite{Tao2000b},
 Klainerman--Rodnianski--Tao \cite{Klainerman2002b}, Lee-Vargas \cite{Lee2008}, and Lee-Rogers-Vargas \cite{Lee2008b}.
 In the case of
 the paraboloid, i.e. $\Phi_1(\xi)=\Phi_2(\xi)=|\xi|^2$, the result
 for homogeneous solutions is due to Tao \cite{Tao2003a}, with
 generalisations by Lee \cite{Lee2006,Lee2006a}, Lee--Vargas
 \cite{Lee2010}, and Bejenaru \cite{Bejenaru2016} under more general
 curvature and transversality conditions, as well as by
 Buschenhenke--M\"uller--Vargas \cite{Buschen2015} for surfaces of
 finite type. For our approach, the most important references are \cite{Tao2003a}
 concerning notation and general line of proof and
 \cite{Lee2010,Bejenaru2016}, concerning the assumptions on the phases
 and its consequences. Throughout the paper, we shall point out
 similarities and differences in more detail.

 We would like to highlight the fact that we explicitely
 track the dependence of the constants on the phases in Theorem \ref{thm - bilinear
   Lp estimate} based on the global, quantitative Assumption \ref{assump on phase},
 in particular we avoid abstract localisation arguments.
 This is helpful for applications to dispersive
 equations, as we will see below. The main novelty of this result, however, lies in
 the fact that it holds for $V^2_{\Phi_j}$-functions in the range $p\leq 2$.

 Now, we turn to the application of Theorem \ref{thm - bilinear Lp
   estimate} to nonlinear dispersive equations with a quadratic
 nonlinearity which exhibit resonances. Roughly speaking, by a
 resonance we mean the scenario that a product of two solutions to the
 homogeneous equations creates another solution of the homogeneous
 equation, see Section 8 for details. This leads to the lack of
 oscillations in the Duhamel integral and hence to strong nonlinear
 effects. In many instances, one finds that the Fourier supports
 intersect transversally in the resonant sets. As an example, we
 mention the local well-posedness theory for the Zakharov system
 \cite{Bejenaru2009,Bejenaru2011}, where this is exploited in terms of
 a nonlinear Loomis-Whitney inequality
 \cite{Bennett2005,Bejenaru2010,Bennett2010,Koch2015}. This is a
 special case of the multilinear restriction theory
 \cite{Bennett2006,Bennett2010}. Here, we will exploit transversality
 in resonant sets via Theorem \ref{thm - bilinear Lp estimate} and
 prove global-in-time estimates which go beyond the range of linear
 Strichartz estimates.

 With this approach, we address the Dirac-Klein-Gordon system
 \begin{equation}\label{eq:dkg}
   \begin{split}
     -i\gamma^\mu \partial_\mu \psi+M\psi=&\phi \psi\\
     \Box\phi +m^2\phi=&\psi^\dagger \gamma^0\psi
   \end{split}
 \end{equation}
 Here, $\psi:\RR^{1+3}\to \CC^4$ is a spinor field,
 $\psi^\dagger=\overline{\psi}^t$, $\phi: \RR^{1+3}\to \RR$ is a
 scalar field, $\Box:=\partial_t^2-\Delta_x$ is the d'Alembertian
 operator, and $M,m\g 0$. We use the summation convention with respect
 to $\mu=0,\ldots, 4$ and the Dirac matrices $\gamma^\mu\in
 \CC^{4\times 4}$ are given by
 \[
 \gamma^0=\diag(1,1,-1,-1), \; \gamma^j=\begin{pmatrix} 0 & \sigma^j\\
   -\sigma^j & 0
 \end{pmatrix},
 \]
 with the Pauli matrices
 \[
 \sigma^1=\begin{pmatrix}0&1\\1&0
 \end{pmatrix}, \; \sigma^2=\begin{pmatrix}0&-i\\i&0
 \end{pmatrix}, \; \sigma^3=\begin{pmatrix}1&0\\0&-1
 \end{pmatrix}.
 \]
 We are interested in the system \eqref{eq:dkg} with the initial
 condition
 \begin{equation}\label{eq:init}
 \psi(0)=\psi_0:\RR^3\to \CC^4 \text{ and } (\phi(0),\partial_t
 \phi(0))=(\phi_0,\phi_1):\RR^3\to \RR\times \RR.
\end{equation}
In the massless case \eqref{eq:dkg} can be rescaled and
 the scale invariant Sobolev space for $(\psi_0,\phi_0,\phi_1)$ is
 \[L^2(\RR^3;\CC^4)\times \dot{H}^{\frac12}(\RR^3;\RR)\times
 \dot{H}^{-\frac12}(\RR^3;\RR).\]
 Let $\langle \Omega\rangle^\sigma$ denote $\sigma$ angular derivatives, see
 Subsection \ref{subsect:anasphere} for precise definitions. Our second main
 result is the following.
 \begin{theorem}\label{thm:dkg}
 Suppose that either $2M\g m>0$ and $\sigma>0$, or that $m>2M>0$ and
 $\sigma >\frac{7}{30}$. Then, for initial data satisfying
\[
\big\|\langle\Omega\rangle^{\sigma}\psi_0\big\|_{L^2(\RR^3)}+\big\|\langle
   \Omega\rangle^{\sigma}\phi_0\big\|_{H^{\frac12}(\RR^3)}+\big\|\langle
   \Omega\rangle^{\sigma}\phi_1\|_{H^{-\frac12}(\RR^3)}\ll 1,
\]
the system \eqref{eq:dkg}--\eqref{eq:init} is globally well-posed and
solutions $(\psi,\phi)$
   scatter to free solutions as $t\to \pm \infty$.
 \end{theorem}
 As the proof relies on contraction arguments in adapted function
 spaces, the notion of global well-posedness in Theorem \ref{thm:dkg}
 includes persistence of regularity and the local Lipschitz continuity
 of the flow map and it provides a certain uniqueness class.
 Note that the angular regularity does not affect the
 scaling of the spaces. In summary, Theorem \ref{thm:dkg}
 establishes global well-posedness and scattering in the critical Sobolev space
 for small initial data with a bit of angular regularity.

 In the case $2M>m>0$, which we call \emph{non-resonant regime} due to
 Lemma \ref{lem - modulation bound}, this theorem improves Wang's
 result \cite{Wang2015} by both relaxing the angular regularity
 hypothesis and replacing Besov spaces by Sobolev spaces. We also
 mention the previous subcritical result \cite{Bejenaru2015} without
 additional angular regularity, where the possibility of a Besov endpoint result with an $\epsilon>0$ of angular regularity was discussed \cite[Remark 4.2]{Bejenaru2015}. In the case $m>2M>0$, which we call
 the \emph{resonant regime} due to Lemma \ref{lem - modulation bound},
 this appears to be the first global well-posedness and scattering
 result in critical spaces for \eqref{eq:dkg}. A similar comment applies to the case $2M=m>0$, which
 we call the \emph{weakly resonant regime}.  It is the resonant regime
 where we employ Theorem \ref{thm - bilinear Lp
   estimate}, see also Remark \ref{rmk:trivial-int}. Concerning further comments on the number of angular derivatives required in
 the resonant case, we refer to Remark
 \ref{rmk:ang-reg}.

 We shall only mention a few selected results
 on this well-studied system \eqref{eq:dkg}.  We refer the reader to
 \cite{D'Ancona2007b} for previous local results and to
 \cite{Chadam1974,Bachelot1988,Bejenaru2015,Wang2015} for previous
 global results on this system, also to the references therein.
 Concerning its relevance in physics we refer the reader to
 \cite{Bjorken1964}.

 The organisation of the paper is as follows: In Section
 \ref{sect:assu}, we discuss a sufficient condition on the phases,
 verify Assumption \ref{assump on phase} in the case of the Schr\"{o}dinger, the wave,
 and the Klein-Gordon equation, and derive important consequences, in
 particular the dispersive inequality, and a bilinear estimate for
 homogeneous solutions in $L^2_{t, x}$. In Section
 \ref{sect:wave-packets-atomic-tubes}, we study wave packets, atomic
 spaces and tubes. In Section \ref{sect:loc}, we state and prove a crucial
 localised version of Theorem \ref{thm - bilinear Lp estimate}. The proof proceeds by performing an induction on scales argument, and reducing the problem to obtaining a crucial $L^2$-bound which in turn follows from a combinatorial estimate. Section \ref{sect:glob} is devoted to the
 globalisation lemma, which removes the localisation assumption used in Section \ref{sect:loc}, and hence concludes the proof of Theorem \ref{thm - bilinear Lp estimate}.
 In Section \ref{sect:small}, we generalise
 Theorem \ref{thm - bilinear Lp estimate} to mixed norms and, in the case of hyperboloids,
 give an extension to general scales and discuss counterexamples. In Section
 \ref{sect:prep} we prepare the analysis of the Dirac-Klein-Gordon
 System and prove Theorem \ref{thm:dkg} under the hypothesis that
 certain bilinear estimates hold true. In Section \ref{sect:mult} we
 discuss some auxilliary estimates and finally provide proofs of the
 bilinear estimates used in Section \ref{sect:prep}.

 \section{On Assumption \ref{assump on phase}: Examples and Consequences}\label{sect:assu}
 In this section we discuss examples, and consider in detail a number of key consequences of
 Assumption \ref{assump on phase}.  All of this is known to
 experts at least in the specific cases we are interested in. The main
 objective is to verify that Assumption \ref{assump on phase} allows
 for a unified treatment which allows to track the dependence of
 constants on the phases.

\subsection{A Sufficient Condition}
Let $\diam(\Lambda_j)= \sup_{\xi, \xi' \in \Lambda_j} |\xi - \xi'|$.
The condition \eqref{it:ass1} in Assumption \ref{assump on phase} is somewhat
difficult to check (essentially since we insist on a \emph{global}
condition rather than just a local condition using the Hessian of
$\Phi_j$). In practise it is easier to check the following marginally
stronger conditions.
\begin{lemma}\label{lem - simplification of cond ii} Assume that the following three conditions hold:
\begin{enumerate}
\item For all $\xi \in \Lambda_1$ and $\eta \in \Lambda_2$
\begin{equation}
|\nabla \Phi_1(\xi) - \nabla \Phi_2(\eta) | \g \mb{A}_1.
\label{eq:trans}
\end{equation}
\item For
  $j=1,2$, and every $\mathfrak{h} \in \RR^{1+n}$ and $\xi, \xi' \in
  \Sigma_j(\mathfrak{h})$
\begin{equation}
\label{eqn - cond on phase II}
    \Big| \big(\nabla \Phi_j(\xi) - \nabla \Phi_j(\xi') \big) \cdot \frac{\xi - \xi'}{|\xi - \xi'|} \Big|  \g \mb{A}_2 |\xi - \xi'|.
  \end{equation}
\item The sets $\Lambda_1$ and $\Lambda_2$ satisfy
\begin{equation}\label{eqn - cond on set II}
\diam(\Lambda_1) + \diam(\Lambda_2)   \les \frac{ \mb{A}_1 \mb{A}_2}{ 2 \big( \| \nabla^2 \Phi_1 \|_{L^\infty(\Lambda_1)} + \| \nabla^2 \Phi_2 \|_{L^\infty(\Lambda_2)}\big)^2}.
\end{equation}
\end{enumerate}
Then, condition \eqref{it:ass1} in Assumption \ref{assump on phase} holds
  with $\mb{D}_1 = \frac{1}{2}\mb{A}_1 \mb{A}_2 $.
\end{lemma}

\begin{proof}
  The first step is to observe that for vectors $x, y \in \RR^n$, and
  $\omega \in \sph^{n-1}$ we have
  \begin{equation}\label{eqn - wedge product identity}
    |x\wedge y| \g |y||x \cdot \omega| - |x| |y \cdot \omega| .
  \end{equation}
  Indeed, this follows from $|x\wedge y|^2 = |x|^2|y|^2 - (x \cdot
  y)^2 = |y|^2 \big| x - \frac{x\cdot y}{|y|^2} y\big|^2$, which
  implies
  $$ |x \wedge y| = |y| \left| x - \frac{x \cdot y}{|y|^2} y
  \right| \g |y| \left| x \cdot \omega - \frac{x \cdot y}{|y|^2} y
    \cdot \omega \right| \g |y| \left( |x\cdot \omega| -
    \frac{|x|}{|y|} |y \cdot \omega|\right).$$ In particular, if we
  let $x = \nabla \Phi_j(\xi) - \nabla \Phi_j(\xi') $, $y = \nabla
  \Phi_j(\xi) - \nabla \Phi_k(\eta)$, and $\omega = \frac{\xi -
    \xi'}{|\xi - \xi'|}$, then since $|x|\les \| \nabla^2
  \Phi_j\|_{L^\infty(\Lambda_j)} |\xi - \xi'|$ (using the convexity of
  $\Lambda_j$) the lower bound \eqref{it:ass1} in Assumption \ref{assump on
    phase} would follow from \eqref{eqn - cond on phase II},
  \eqref{eqn - wedge product identity}, and the transversality
  condition \eqref{eq:trans}, provided that
  \begin{equation}\label{eqn - dot product small} \left|
      \big(\nabla
      \Phi_j(\xi) - \nabla \Phi_k(\eta)\big) \cdot \frac{\xi -
        \xi'}{|\xi - \xi'|} \right| \les \frac{ \mb{A}_1 \mb{A}_2}{2 \|
      \nabla^2 \Phi_j\|_{L^\infty(\Lambda_j)}}.
  \end{equation}
  The proof of \eqref{eqn - dot product small} requires the condition
  $\xi, \xi' \in \Sigma_j(\mathfrak{h})$ together with the assumption
  \eqref{eqn - cond on set II} on the size of the sets
  $\Lambda_j$. Let
$$ \sigma_j(x, z)= \Phi_j(x) - \Phi_j(z) - \nabla \Phi_j(z) \cdot (x-z).$$
A computation gives
\begin{align*}  \nabla\Phi_j(z) \cdot (x - y) &= \big( \Phi_j(x) - \sigma_j(x, z) - \Phi_j(z) - \nabla \Phi_j(z) \cdot z\big) - \big( \Phi_j(y) - \sigma_j(y, z) - \Phi_j(z) - \nabla \Phi_j(z) \cdot z\big) \\
  &= \Phi_j(x) - \Phi_j(y) + \sigma_j(y, z) - \sigma_j(x, z)
\end{align*}
and hence, using the assumption $\xi, \xi' \in
\Sigma_j(\mathfrak{h})$, we see that
\begin{align*}
  \big( \nabla \Phi_j(\xi) - \nabla &\Phi_k(\eta) \big) \cdot ( \xi - \xi') \\
  &= \Phi_j(\xi) - \Phi_j(\xi') + \sigma_j(\xi', \xi)  - \Big( \Phi_j(\xi - h) - \Phi_k(\xi' - h) + \sigma_k(\xi' - h, \eta) - \sigma_k(\xi - h, \eta) \Big) \\
  &= \sigma_j(\xi', \xi) + \sigma_k(\xi-h, \eta) - \sigma_k(\xi'-h,
  \eta).
\end{align*}
If we now observe that
$$\sigma_j(x, z) - \sigma_j(y,z) = \int_0^1 \big[ \nabla \Phi_j\big( y + t(x-y)\big)  - \nabla \Phi_j(z) \big]\cdot (x-y) dt \les \| \nabla^2 \Phi_j \|_{L^\infty(\Lambda_j)} \text{ diam}(\Lambda_j) |x-y|$$ we then deduce the bound
$$\left| \big(\nabla \Phi_j(\xi) - \nabla \Phi_k(\eta)\big) \cdot \frac{\xi - \xi'}{|\xi - \xi'|} \right| \les \text{ diam}(\Lambda_1) \| \nabla^2 \Phi_1 \|_{L^\infty(\Lambda_1)} + \text{ diam}(\Lambda_2) \| \nabla^2 \Phi_2 \|_{L^\infty(\Lambda_2)}. $$
Consequently \eqref{eqn - dot product small} follows from \eqref{eqn -
  cond on set II}.
\end{proof}

\subsection{The Schr\"{o}dinger, the Wave and the Klein-Gordon Equation}\label{subsect:swkg} We now consider some examples of phases satisfying Assumption \ref{assump on phase}. It is enough to check the conditions in Lemma \ref{lem - simplification of cond ii}. In particular, by making the sets $\Lambda_j$ slightly smaller if necessary, it suffices to ensure that the transversality condition (\ref{eq:trans}) and curvature condition (\ref{eqn - cond on phase II}) hold.

Firstly, consider the Schr\"{o}dinger case
$$ \Phi_j(\xi) = \frac{1}{2} |\xi|^2.$$
Then the condition \eqref{eq:trans} in Lemma \ref{lem - simplification of cond ii} becomes
$$ |\nabla \Phi_1(\xi) - \nabla\Phi_2(\eta)| = |\xi - \eta|,$$
thus we simply require that the sets $\Lambda_j$ have some
separation. Assuming that the diameters of the sets $\Lambda_j$ are
sufficiently small, we just need to ensure that \eqref{eqn - cond on
  phase II} holds. However \eqref{eqn - cond on phase II} is just
$$ \Big| \big(\nabla \Phi_j(\xi) - \nabla \Phi_j(\xi') \big) \cdot \frac{\xi - \xi'}{|\xi - \xi'|} \Big| = |\xi - \xi'|$$
and so \eqref{eqn - cond on phase II} clearly holds (with constant
$\mb{A}_2=1$).

Secondly, consider the case
$$ \Phi_j(\xi) = \lr{\xi}_{m_j} =  \big( m_j^2 + |\xi|^2 \big)^\frac{1}{2}$$
where the mass satisfies $m_j \g 0$.  To simplify notation, we assume
that for $\xi \in \Lambda_j$ we there is a constant $A>0$ such that
$$ \frac{1}{A} \les \lr{\xi}_{m_j} \les A. $$
To check the transversality condition \eqref{eq:trans} we note that
\begin{align} \big| \nabla\Phi_1(\xi) - \nabla \Phi_2(\eta) \big|^2 &= \left| \frac{\xi}{\lr{\xi}_{m_1}} - \frac{\eta}{\lr{\eta}_{m_2}} \right|^2 \notag \\
  &= \left( \frac{|\xi|}{\lr{\xi}_{m_1}} - \frac{|\eta|}{\lr{\eta}_{m_2}} \right)^2 + \frac{ 2 |\xi| |\eta|}{\lr{\xi}_{m_1} \lr{\eta}_{m_2}} \left( 1 - \frac{\xi \cdot \eta}{|\xi| |\eta| } \right) \notag\\
  &= \left( \frac{ (m_2 |\xi| + m_1 |\eta|)(m_2 |\xi| - m_1 |\eta|)}{\lr{\xi}_{m_1}
      \lr{\eta}_{m_2} ( |\xi| \lr{\eta}_{m_2} + |\eta|
      \lr{\xi}_{m_1})} \right)^2 + \frac{ 2
    |\xi| |\eta|}{\lr{\xi}_{m_1} \lr{\eta}_{m_2}} \left( 1 - \frac{\xi
      \cdot \eta}{|\xi| |\eta| } \right) \label{eqn:KG transverse}
\end{align}
(in particular, we \emph{always} have transversality if $|\xi| \approx
|\eta| \approx 1$ and $m_1 \ll m_2$).

On the other hand, to check the condition \eqref{eqn - cond on phase II},  we use the
following elementary bound.

\begin{lemma}\label{lem - simplifying curvature computation}
  Let $\ell \g 2$ and $(a, h) \in \RR^{1+\ell}$. If $x, y \in \{ z\in \RR^\ell
  \mid |z|=|z-h| +a\}$ we have the inequality
$$\Big| \frac{x}{|x|} - \frac{y}{|y|} \Big|^2 \g  |x-y|^2 \Big| \frac{x}{|x|} - \frac{x-h}{|x-h|}\Big|^4 \frac{ |x-h|^2}{16|x| |y| |x-h|^2 + 4 (|x-h|+|x|)^2 |y|^2}. $$
\end{lemma}
\begin{proof}
  The condition $x \in \{ z\in \RR^\ell \mid |z|=|z-h| +a\}$ implies that
  $|x-h|^2 = ( |x| -a)^2$ and hence $\frac{x}{|x|} \cdot h =
  \frac{|h|^2 -a^2}{2|x|} +a$. Therefore
 $$ \Big| \frac{x}{|x|} - \frac{y}{|y|} \Big| \g \frac{|h|^2 - a^2}{2|h|} \Big| \frac{1}{|x|} - \frac{1}{|y|} \Big| = \frac{ |x-h|}{2|h| |y| } \Big| \frac{x}{|x|} - \frac{x-h}{|x-h|} \Big|^2 \big| |x|-|y| \big|$$
 where we used the identities $h=x - (x-h)$ and $a= |x| -
 |x-h|$. Lemma now follows by noting that $|x-y|^2 = |x| |y|\big|
 \frac{x}{|x|} - \frac{y}{|y|} \big|^2 + \big| |x|-|y| \big|^2$.
\end{proof}

We now show that \eqref{eqn - cond on phase II} holds. A computation
gives
\begin{align*}
  \big|\big(\nabla \Phi_j(\xi) - \nabla \Phi_j(\xi') \big) \cdot (\xi - \xi') \big| &= \left| \frac{ |\xi|^2}{\lr{\xi}_{m_j}} + \frac{|\xi'|^2}{\lr{\xi'}_{m_j}}  - \frac{ \xi \cdot \xi'}{\lr{\xi}_{m_j}}  - \frac{\xi \cdot \xi'}{\lr{\xi}_{m_j}} \right| \\
  &= \left| \lr{\xi}_{m_j} + \lr{\xi'}_{m_j} - \frac{\xi \cdot \xi' + m_j^2}{\lr{\xi}_{m_j}} - \frac{\xi \cdot \xi' + m_j^2}{\lr{\xi'}_{m_j}} \right| \\
  &= \frac{\lr{\xi}_{m_j} + \lr{\xi'}_{m_j}}{2} \Big| \frac{x}{|x|} -
  \frac{y}{|y|} \Big|^2
\end{align*}
were we let $x = (m_j, \xi)$ and $y=(m_j, \xi')$. If we now note that
the surface $\Phi_j(\xi) = \Phi_k(\xi - h) +a$ can be written as $|x|
= |y-h'| + a$ with $h'=(m_k-m_j, h)$, then an application of Lemma
\ref{lem - simplifying curvature computation} gives
$$ \big|\big(\nabla \Phi_j(\xi) - \nabla \Phi_j(\xi') \big) \cdot (\xi - \xi') \big| \g \frac{\mb{A}_1^4}{32 A^6} |\xi - \xi'|^2.$$
Therefore, by Lemma \ref{lem - simplification of cond ii}, we see that
\eqref{it:ass1} in Assumption \ref{assump on phase} holds with $\mb{D}_1 = \frac{
  \mb{A}_1^5}{64 A^6}$. Note that the above argument also applies in the case of the wave equation $m_1=m_2=0$.

\subsection{The Dispersive Inequality}
To simplify the statements to follow, we fix constants $\mb{R}_0\g 1$,
$ \mb{D}_1, \mb{D}_2>0$ and $N > n+1$, and assume that we have
phases $\Phi_1$, $\Phi_2$ satisfying Assumption \ref{assump on phase}
and sets $\Lambda_j$, $\Lambda_j^*$ with $\Lambda_j$ convex and
$\Lambda_j^* + \frac{1}{\mb{R}_0} \subset \Lambda_j \subset \{
\frac{1}{16}\les |\xi| \les 16\}$.

As a consequence of the curvature type bound \eqref{eqn - seperation of velocities in Sigma} relative to the ($n-1$)-dimensional surface  $\Sigma_j(\mathfrak{h})$, we
expect that we should have the dispersive inequality
\begin{equation}\label{eqn - dispersive est}
  \|e^{it \Phi_j(-i\nabla)} f \|_{L^\infty_x} \lesa t^{-\frac{n-1}{2}} \| f \|_{L^1_x}
\end{equation}
for $f \in L^1$ with $\supp \widehat{f} \subset \Lambda_j$. To prove
this decay in practise, the standard approach would involve a
stationary phase argument. However, as we only have curvature
information on the surfaces $\Sigma_j(\mathfrak{h})$, and these
surfaces are somewhat involved to work with, the standard approach via
stationary phase arguments, keeping track of the constants, seems
difficult to implement. Consequently, we instead present a different
argument, using an approach via wave packets.  Roughly speaking,
fixing some large time $t \approx R$, the idea is to cover $\Lambda_j$
with balls of size $R^{-\frac{1}{2}}$ and decompose
        $$ e^{it \Phi_j(-i\nabla)} f = \sum_{\xi_0 \in R^{-\frac{1}{2} } \ZZ^n \cap \supp \widehat{f}} K_{\xi_0} * f$$
        for some smooth kernels $K_{\xi_0}(t, x)$ with $\|
        K_{\xi_0}(t) \|_{L^\infty_x} \les R^{-\frac{n}{2}}$. Then
        since $\Sigma_j(\mathfrak{h})$ is a hypersurface, by
        restricting to points close to $\Sigma_j(\mathfrak{h})$ we
        should have
        \begin{align*}
          \| e^{ i t \Phi_j(-i\nabla)} f\|_{L^\infty_x}& \les \| f \|_{L^1_x} \bigg\| \sum_{\xi_0 \in R^{-\frac{1}{2} } \ZZ^n \cap \supp \widehat{f}} K_{\xi_0}(t, x) \bigg\|_{L^\infty_x} \\
          &\lesa \| f \|_{L^1_x} R^\frac{1}{2} \sup_{\mathfrak{h}}
          \bigg\| \sum_{\xi_0 \in R^{-\frac{1}{2} } \ZZ^n \cap
            (\Sigma_j(\mathfrak{h}) + R^{-\frac{1}{2}})} K_{\xi_0}(t,
          x) \bigg\|_{L^\infty_x}.
        \end{align*}
        The condition \eqref{it:ass1} in Assumption \ref{assump on phase} then
        shows that, for times $t \approx R^{-\frac{1}{2}}$, the
        spatial supports of the kernels $K_{\xi_0}(t, x)$ are
        essentially disjoint, and hence
    $$ \bigg\| \sum_{\xi_0 \in R^{-\frac{1}{2} } \ZZ^n \cap (\Sigma_j(\mathfrak{h}) + R^{-\frac{1}{2}})} K_{\xi_0}(t, x) \bigg\|_{L^\infty_x} \approx \sup_{\xi_0 \in R^{-\frac{1}{2} } \ZZ^n \cap (\Sigma_j(\mathfrak{h}) + R^{-\frac{1}{2}})} \| K_{\xi_0}(t) \|_{L^\infty_x} \lesa R^{-\frac{n}{2}} \approx t^{-\frac{n}{2}}$$
    which would then give the desired dispersive estimate \eqref{eqn -
      dispersive est}.

    In the remainder of this subsection, we fill in the details of the
    argument sketched above.  We first require a technical lemma
    involving the surfaces $\Sigma_j(\mathfrak{h})$.

\begin{lemma}\label{lem - properties of surface Sigma}
  Let $\{j, k\} = \{ 1, 2\}$, $\mathfrak{h} = (a, h) \in \RR^{1+n}$,
  and $r \g 2\frac{\mb{D}_2}{\mb{D}_1} \mb{R}_0$. Assume $\xi_0 \in
  (\Lambda_j^* + \frac{1}{2 \mb{R}_0}) \cap (\Lambda_k^* + h +
  \frac{1}{2 \mb{R}_0})$ and
        $$ \big| \Phi_j(\xi_0) - \Phi_k(\xi_0-h) - a \big| \les \frac{1}{r}.$$
        Then $|\xi_0 - \Sigma_j(\mathfrak{h})| \les \frac{\mb{D}_2}{\mb{D}_1 r} $.
      \end{lemma}
      \begin{proof}
        Define $F(\xi) = \Phi_1(\xi) - \Phi_2(\xi - h) -a$, by replacing $F$ with $-F$ if necessary, we may assume that $F(\xi_0)\g 0$. We need to
        show there exists $|\xi - \xi_0| \les \frac{\mb{D}_2}{\mb{D}_1 r}$ such that $F(\xi)=0$. To this end,
        let $\xi(s)$ be the solution to
        \begin{align*}
          \p_s \xi(s) &= - \frac{\nabla F\big( \xi(s) \big)}{|\nabla F\big( \xi(s) \big)|}\\
          \xi(0) &= \xi_0.
        \end{align*}
        Note that, for times $s \in [0, \frac{\mb{D_2}}{r \mb{D}_1}]$, we
        have $|\xi(s) - \xi_0| \les  s $. On the other hand, since
        $|F(\xi_0)| \les \frac{1}{r}$ by assumption, the
        transversality property \eqref{eq:trans-main} implies
	$$ F\big(\xi(s) \big) = F(\xi_0) - \int_0^s |\nabla F\big(\xi(s')\big) | ds' \les \frac{1}{r} - s \frac{\mb{D}_1}{\mb{D}_2}. $$
        Consequently $F(\xi(s))$ must be zero for some $s \in [0,
        \frac{\mb{D_2}}{r \mb{D}_1}]$ and hence result follows.
      \end{proof}

      We now come to the proof of the dispersive inequality.

      \begin{lemma}[Dispersion] \label{lem - dispersion} Let
        $j=1,2$. For any $f \in L^1_x$ with $\supp \widehat{f} \subset
        \Lambda_j^* + \frac{1}{2\mb{R}_0}$ and any $t \g1 $ we have
	$$ \big\| e^{it \Phi_j(-i\nabla)} f \big\|_{L^\infty_x} \lesa t^{-\frac{n-1}{2}}  \| f \|_{L^1_x}$$
        where the implied constant depends only $\mb{R}_0,
        \mb{D}_1, \mb{D}_2$, and $n\geq 2$.
      \end{lemma}

\begin{proof}
  It is enough to consider the case $j=1$ and $R\les t \les 2R$ with
  $R \g (10 \mb{R}_0)^2$. Since $\Lambda_2^* + \frac{1}{2 \mb{R}_0}$
  contains a ball of size $(2\mb{R}_0)^{-1}$, we can find a finite set
  $H \subset \RR^n$ such that $\# H \lesa \mb{R}_0^n$ and $ \Lambda_1
  = \cup_{h \in H} \Lambda_1 \cap (\Lambda_2^* + \frac{1}{2\mb{R}_0}
  h)$. In particular, by decomposing $\supp \widehat{f}$ into
  $\mc{O}(\mb{R}_0^n)$ sets, is is enough to consider the case $\supp
  \widehat{f} \subset (\Lambda^*_1 + \frac{1}{2\mb{R}_0}) \cap
  (\Lambda_2^* + \frac{1}{2\mb{R}_0} + h)$. Let $\rho \in
  C^\infty_0(|\xi| \les 1)$ such that
		$$ \sum_{k \in \ZZ^n} \rho(\xi - k) = 1.$$
                The support assumption on $\widehat{f}$, together with
                the fact that $R\g (10\mb{R}_0)^2$, implies that
	$$ \big(e^{it \Phi_1(-i\nabla)} f\big)(x) = \sum_{\xi_0 \in R^{-\frac{1}{2} }\ZZ^n \cap (\supp \widehat{f} + \frac{1}{10\mb{R}_0})} K_{\xi_0}(t) * f(x)$$
        where $K_{\xi_0}(t,x) = \int_{\RR^n} \rho( R^{\frac{1}{2}}(\xi
        - \xi_0)) e^{ i t \Phi_1(\xi)} e^{ i x \cdot \xi} d\xi$. Since
        $R\les t \les 2R$, our goal is to show that
	$$\Big\| \sum_{\xi_0 \in R^{-\frac{1}{2} }\ZZ^n \cap (\supp \widehat{f} + \frac{1}{10\mb{R}_0})} | K_{\xi_0}(t,x)|\Big\|_{L^\infty_x} \lesa R^{-\frac{n-1}{2}}.$$
        We would like to write this sum in a way which involves the
        hypersurfaces $\Sigma_1(\mathfrak{h})$. Fix $0<\delta \ll
        \frac{\mb{D}_1}{\mb{D}_1 + \mb{D}_2}$ and let $\delta^* =
        \frac{\mb{D}_1}{\mb{D}_2}\delta$.  Given $\xi_0 \in
        R^{-\frac{1}{2}} \ZZ \cap (\supp \widehat{f} + \frac{1}{10
          \mb{R}_0})$, we can find $a \in \delta^* R^{-\frac{1}{2}}
        \ZZ$ with $|a|\les 2 \mb{D}_2$ such that
        $$ | \Phi_1(\xi_0) - \Phi_2(\xi_0 - h) - a | \les \delta^* R^{-\frac{1}{2}}. $$
        Therefore, an application of Lemma \ref{lem - properties of
          surface Sigma} with $r= R^{\frac{1}{2}}/\delta^*$, implies that $\xi_0 \in \Sigma_1(a, h) +
        \delta R^{-\frac{1}{2}}$ and hence we have
	\begin{align*} \sum_{\xi_0 \in R^{-\frac{1}{2} }\ZZ^n \cap (\supp \widehat{f} + \frac{1}{10\mb{R}_0})} | K_{\xi_0}(t, x)| &\les \sum_{\substack{ a \in \delta^* R^{-\frac{1}{2}} \ZZ \\ |a| \les 2 \mb{D}_2}}\,\, \sum_{\xi_0 \in R^{-\frac{1}{2} }\ZZ^n \cap (\Sigma_1(a, h) + \delta R^{-\frac{1}{2}})} | K_{\xi_0}(t, x)| \\
          &\lesa R^\frac{1}{2} \sup_{\mathfrak{h}} \sum_{\xi_0 \in
            R^{-\frac{1}{2} }\ZZ^n \cap (\Sigma_1(\mathfrak{h}) +
            \delta R^{-\frac{1}{2}})} | K_{\xi_0}(t, x)|.
	\end{align*}
        We now exploit the localisation of the kernel, together with
        the partial curvature condition \eqref{eqn - seperation of velocities in Sigma}. Write
	\begin{align*} K_{\xi_0}(t,x) &= R^{-\frac{n}{2}} \int_{\RR^n}
          \rho(\xi) e^{ it [ \Phi_1( R^{-\frac{1}{2}} \xi + \xi_0) -
            R^{-\frac{1}{2}} \nabla \Phi_1(\xi_0) \cdot \xi ]} \,\,e^{
            i R^{-\frac{1}{2}} (x + t \nabla\Phi_1(\xi_0))\cdot \xi}
          d\xi.
	\end{align*}
        Integrating by parts $n+1$ times gives
        \begin{equation}\label{eqn - lem dispersion - decay of kernel}
          |K_{\xi_0}(t, x)| \lesa R^{ - \frac{n}{2}} \Big( 1 +
          R^{-\frac{1}{2}} \big| x + t \nabla\Phi_1(\xi_0)\big|
          \Big)^{-n-1}.
        \end{equation}
        Let $\xi_0' \in R^{-\frac{1}{2} }\ZZ^n \cap (\Sigma_1(a, h) +
        R^{-\frac{1}{2}})$ denote the minimum of $|x +t\nabla
        \Phi_1(\xi_0)|$. We claim that for every $\xi_0 \in
        R^{-\frac{1}{2} }\ZZ^n \cap (\Sigma_1(a, h) +
        R^{-\frac{1}{2}})$ we have
        \begin{equation}\label{eqn - lem dispersion - lower bound on
            localised terms}
          |x +t\nabla \Phi_1(\xi_0)| \g \frac{ \mb{D}_1}{4} R |\xi_0 - \xi_0'|.
        \end{equation}
        Assuming this holds for the moment, we would then obtain
	\begin{align*} \sum_{\xi_0 \in R^{-\frac{1}{2} }\ZZ^n \cap (\supp \widehat{f} + \frac{1}{10\mb{R}_0})} |K_{\xi_0}(t,x)| &\lesa R^\frac{1}{2} \sup_{\mathfrak{h}} \sum_{\xi_0 \in R^{-\frac{1}{2} }\ZZ^n \cap (\Sigma_1(\mathfrak{h}) + R^{-\frac{1}{2}})} | K_{\xi_0}(t, x)|  \\
          &\lesa R^{-\frac{n-1}{2}} \sum_{\xi_0 \in R^{-\frac{1}{2}} \ZZ^n} ( 1 + R^\frac{1}{2} |\xi_0 - \xi_0'|)^{-n-1} \\
          &\lesa R^{-\frac{n-1}{2}}
	\end{align*}
        as required. Thus it only remains to verify \eqref{eqn - lem
          dispersion - lower bound on localised terms}. This is immediate if $ R \mb{D}_1 |\xi_0 - \xi_0'| \les 2 |x
        + t \nabla \Phi_1(\xi_0')|$. Thus we may assume that $R
        \mb{D}_1 |\xi_0 - \xi_0'| \g 2 |x + t \nabla
        \Phi_1(\xi_0')|$. Note that this implies that $|\xi - \xi_0|\g
        R^{-\frac{1}{2}}$. By construction, there exists $\xi, \xi'
        \in \Sigma_1(\mathfrak{h})$ such that $|\xi - \xi_0| \les
        \delta R^{-\frac{1}{2}}$, $|\xi' - \xi_0'|\les \delta
        R^{-\frac{1}{2}}$. Therefore, applying the lower bound
        \eqref{eqn - seperation of velocities in Sigma}, we deduce
        that
	\begin{align*}
          |x + t \nabla \Phi_1(\xi_0)| &\g t |\nabla\Phi(\xi) - \nabla \Phi(\xi')| - |x + t \nabla \Phi_1(\xi_0')| - t |\nabla \Phi_1(\xi_0) - \nabla \Phi_1(\xi)| -  t |\nabla \Phi_1(\xi_0') - \nabla \Phi_1(\xi')| \\
          &\g R \mb{D}_1 |\xi - \xi'|  - |x + t \nabla \Phi_1(\xi_0')| - 4 \mb{D}_2 \delta R^\frac{1}{2} \\
          &\g R \frac{\mb{D}_1}{2}  |\xi_0 - \xi'_0|  - 4(\mb{D}_1 + \mb{D}_2) \delta R^\frac{1}{2} \\
          &\g R \frac{\mb{D}_1}{4} |\xi_0 - \xi'_0|
	\end{align*}
        provided that we choose $\delta \ll \frac{\mb{D}_1}{\mb{D}_1 +
          \mb{D}_2}$. Hence we obtain \eqref{eqn - lem dispersion -
          lower bound on localised terms} and thus result follows.
      \end{proof}
      \begin{remark}\label{rem:strichartz}
        By the standard $TT^*$-argument, this implies the linear
        Strichartz type estimates for wave admissible pairs. We omit
        the details and refer to \cite{Keel1998}.
      \end{remark}

      \subsection{Classical Bilinear Estimate in $L^2_{t, x}$} The
      main use of the transversality property \eqref{eq:trans-main} contained in Assumption
      \ref{assump on phase} is to deduce the following well-known
      bilinear estimate, which dates back at least to Bourgain
      \cite[Lemma 111]{Bourgain1998} in the case of the Schr\"odinger
      equation and $n=2$.

\begin{lemma}\label{lem - key bilinear L2 bound I}
  Let $0<r<1$ and $f, g \in L^2_x$. Assume that the supports of
  $\widehat{f}$ and $\widehat{g}$ are contained in balls of radius $r$
  intersected with $\Lambda_1$ and $\Lambda_2$ respectively,
and for all $\xi \in\Lambda_1$ and $\eta \in\Lambda_2$
\begin{equation}\label{eq:trans-c0}
|\nabla\Phi_1(\xi)-\nabla\Phi_2(\eta)|\g \mb{C_0}.
\end{equation}
Then,
        $$ \| e^{ it \Phi_1(-i\nabla)} f e^{ i t \Phi_2(-i\nabla)} g \|_{L^2_{t, x}(\RR^{1+n})} \lesa \Big(\frac{ r^{n-1}}{\mb{C}_0} \Big)^{\frac{1}{2}} \| f \|_{L^2_x} \| g \|_{L^2_x}.$$
      \end{lemma}
      \begin{proof}
        For $m=1,\ldots,n$ let $\Omega_m = \big\{ (\xi, \eta) \in
        \Lambda_1 \times \Lambda_2 \,\,\big| \,\,| \p_m \Phi_1(\xi) -
        \p_m \Phi_2(\eta)| \g \tfrac{\mb{C}_0}{2n} \big\}.$ Condition \eqref{eq:trans-c0} and the
        support assumptions on $\widehat{f}$ and $\widehat{g}$ implies
        that we can decompose
$$ \widehat{\big( e^{ it \Phi_1(-i \nabla)} f e^{ it \Phi_2(-i\nabla)}g\big)}(\xi) = \sum_{m=1}^n \int_{\RR^n} \widehat{f}(\xi-\eta) \widehat{g}(\eta) \ind_{\Omega_m}(\xi-\eta, \eta) \,e^{ it (\Phi_1(\xi-\eta) + \Phi_2(\eta))} d\eta.$$
Consider the $m=1$ term and write $\eta=(\eta_1, \eta') \in \RR \times
\RR^{n-1}$. The change of variables $(\eta_1, \eta') \mapsto (\tau,
\eta')$ where $\tau = \Phi_1(\xi-\eta) + \Phi_2(\eta)$ gives
  \begin{align*}
    &\int_{\RR^n} \widehat{f}(\xi-\eta) \widehat{g}(\eta) \ind_{\Omega_1}(\xi-\eta, \eta) \,e^{ it (\Phi_1(\xi-\eta) + \Phi_2(\eta))} d\eta \\
={}& \int_{\RR} \int_{\RR^{n-1}} \frac{\widehat{f}\big(\xi-\eta^*\big)\widehat{g}\big(\eta^*\big)}{\p_1 \Phi_1(\xi-\eta^*) - \p_1 \Phi_2(\eta^*)}  \ind_{\Omega_1}\big(\xi-\eta^*, \eta^*\big) \,d\eta' e^{ it\tau} d\tau,
  \end{align*}
  where $\eta^* = \big(\eta_1[\tau, \xi, \eta'], \eta')$. Thus an
  application of Plancherel, followed by H\"older in $\eta'$, shows
  that
  \begin{align*}
   & \bigg\| \int_{\RR^n} \widehat{f}(\xi-\eta) \widehat{g}(\eta) \ind_{\Omega_m}(\xi-\eta, \eta) \,e^{ it (\Phi_1(\xi-\eta) + \Phi_2(\eta))} d\eta \bigg\|_{L^2_{t, \xi}} \\
={}& \bigg\| \int_{\RR^{n-1}} \frac{\widehat{f}\big(\xi-\eta\big)\widehat{g}\big(\eta\big)}{\p_1 \Phi_1(\xi-\eta) - \p_1 \Phi_2(\eta)}  \ind_{\Omega_1}\big(\xi-\eta, \eta\big) d\eta' \bigg\|_{L^2_{\tau, \xi}} \\
    \les{} & (2r)^{\frac{n-1}{2}} \frac{2n}{\mb{C}_0}   \bigg\| \widehat{f}\big(\xi-\eta^*\big)\widehat{g}\big(\eta^*\big) \bigg\|_{L^2_{\tau, \xi, \eta'}} \\
    \les{} & 2n \Big(\frac{(2r)^{n-1}}{\mb{C}_0} \Big)^{\frac{1}{2}} \|
    f\|_{L^2_x} \| g \|_{L^2_x}
  \end{align*}
  where the last inequality follows by undoing the change of
  variables. Since the terms with $1<m\les n$ are identical the lemma
  follows.
\end{proof}

\subsection{Geometric Consequences} The last step in the proof of
Theorem \ref{thm - bilinear Lp estimate} requires a combinatorial
Kakeya type bound. This bound relies on the fact that certain tubes
intersect transversally, and is the main reason for introducing the
condition \eqref{it:ass1} in Assumption \ref{assump on phase}.  The following is
motivated by \cite{Lee2010,Bejenaru2016}, see also Section 9 of \cite{Tao2003a}.

 Let $\mathfrak{h} \in
\RR^{1+n}$ and define the conic hypersurface
    $$ \mc{C}_j(\mathfrak{h}) = \{ (r, - r \nabla \Phi_j(\xi) ) \mid   r \in \RR, \,\, \xi \in \Sigma_j(\mathfrak{h}) \}. $$
A computation shows that the tangent plane to
    $\mc{C}_j(\mathfrak{h})$ is spanned by the vectors
		$$ \big(1, - \nabla \Phi_j(\xi) \big) \text{ and }  H\Phi_j(\xi)v \text{ for }v \in T_\xi \Sigma_j(\mathfrak{h}).$$
On the other hand, as we will see in the proof the Lemma \ref{lem - surface C transverse} below, the condition \eqref{it:ass1} in Assumption \ref{assump on phase} implies that
	$$ \big| \big( 1 , - \nabla \Phi_j(\xi) \big) \wedge \big( 1, - \nabla \Phi_k(\eta) \big) \wedge \big( 0, \nabla\Phi_j(\xi) - \nabla\Phi_j(\xi') \big) \big|  \gtrsim |\xi - \xi'|$$
        for every $\xi, \xi' \in \Sigma_j(\mathfrak{h})$.
Hence, letting $\xi' \rightarrow \xi$ in $\Sigma_j(\mathfrak{h})$, we can interpret \eqref{it:ass1} in Assumption \ref{assump on phase} as saying that, for every $v \in T_\xi \Sigma_j(\mathfrak{h})$ we have
		$$\big| \big( 1 , - \nabla \Phi_j(\xi) \big) \wedge \big( 1, - \nabla \Phi_k(\eta) \big) \wedge \big( 0, H\Phi(\xi) v \big) \big| \gtrsim |v| $$
where $H\Phi_j(\xi)$ denotes the Hessian of $\Phi_j$
at $\xi$. In particular, the vector $(1, - \nabla \Phi_k(\eta) )$ must be transversal to the surface $\mc{C}_j(\mathfrak{h})$ for every $\eta \in \Lambda_k$. A more quantitative version of this statement -- and the one we make use of in
                practise -- is given by the following.

\begin{lemma}\label{lem - surface C transverse}
  Let $\mathfrak{h} \in \RR^{1+n}$ and $\{j, k\} = \{1, 2\}$. For
  every $\eta \in \Lambda_j$ and $p, q \in \mc{C}_k(\mathfrak{h})$ we have
	$$  \big| (p-q) \wedge \big( 1 , - \nabla \Phi_j(\eta)\big) \big|\ \g  \frac{\mb{D}_1 |p-q|}{ (1 + \| \nabla \Phi_k\|_{L^\infty(\Lambda_k)}) \| \nabla^2 \Phi_k \|_{L^\infty(\Lambda_k)}} .$$
      \end{lemma}
      \begin{proof}
        Let $w, w', w'' \in \RR^n$. The identity $ |x\wedge y \wedge
        z| = \inf_{ v \in \text{span}\{x, y\}} \frac{ |v \wedge
          z|}{|v|} |x \wedge y|$ implies that
        \begin{align*}
          \big| \big( 1, w''\big) \wedge \big( 1 , w \big) \wedge
          \big( 0, w - w' \big) \big|
          &= \big| \big( 1, w''\big) \wedge \big(0 , w-w''\big) \wedge \big( 0,w- w'\big) \big| \\
          &= \inf_{v \in W} \frac{ \big|v \wedge  \big( 1, w''\big)\big|}{|v|} \big| \big(0 , w - w''\big) \wedge \big( 0, w - w' \big) \big| \\
          &\g |( w - w'') \wedge (w - w') |,
        \end{align*}
        where $W=\text{span}\{(0, w - w''), (0, w - w') \}$.
        Consequently, applying the wedge product identity once more,
        we deduce that for every $ {v \in \text{span}\{ (1, w), (0,
          w-w') \}}$
	\begin{equation}\label{eqn - lem surface C transverse - wedge
            product identity}
          |v \wedge (1, w'')| \g \frac{ |(w - w'')\wedge (w - w')|}{ (1 +|w|) |w-w'|} |v|.
        \end{equation}	
        Fix $\eta \in \Lambda_j$ and $p, q \in
        \mc{C}_k(\mathfrak{h})$. By definition, this implies that we
        have $\xi, \xi' \in \Sigma_j(\mathfrak{h})$ and $r, r' >0$
        such that $p = (r, - r\nabla \Phi_k(\xi))$ and $q= (r',
        -r'\nabla \Phi_k(\xi'))$.  Clearly, due to the convexity of
        $\Lambda_k$ we have $|\nabla \Phi_k(\xi) - \nabla
        \Phi_k(\xi')| \les \| \nabla^2 \Phi_k \|_{L^\infty(\Lambda_k)}
        |\xi - \xi'|$.  If we now let $w = - \nabla \Phi_k(\xi)$, $w'
        = - \nabla \Phi_k(\xi')$, and $w'' =- \nabla \Phi_j(\eta)$ in
        (\ref{eqn - lem surface C transverse - wedge product
          identity}), then we deduce from \eqref{it:ass1} in Assumption
        \ref{assump on phase} that
	$$ \big| v \wedge \big( 1, - \nabla \Phi_j(\eta) \big) \big| \g  \frac{\mb{D}_1 |v|}{ (1 + \| \nabla \Phi_k\|_{L^\infty(\Lambda_k)}) \| \nabla^2 \Phi_k \|_{L^\infty(\Lambda_k)}} $$
        for every $v \in \text{span}\{ (1, - \nabla \Phi_k(\xi)), (0,
        \nabla \Phi_k(\xi) - \nabla \Phi_k(\xi'))\}$. Taking $v=p-q$
        and observing that we can write
	$$(p - q) = (r-r') (1, - \nabla \Phi_k(\xi)) + r' (0, \nabla \Phi_k(\xi) - \nabla \Phi_k(\xi') ),$$
        the required bound now follows.
      \end{proof}

      \section{Wave Packets, Atomic Spaces, and
        Tubes}\label{sect:wave-packets-atomic-tubes}
      In this section we discuss the  wave packet decomposition. To
      some extent, we follow the arguments in \cite{Tao2003a}, but use
      a slightly different notation by using projections labelled by
      phase space points as in \cite{Lee2010}.  Again, this helps us
      to carefully track constants. In addition, we consider
      certain atomic decompositions.  Concerning the phases $\Phi_j$,
      it turns out that the only property we require in the
      construction of wave packets below, is \eqref{it:ass2} in Assumption
      \ref{assump on phase}. Consequently, throughout this section, we
      fix constants $\mb{R}_0\g 1$, $\mb{D}_2>0$ and $N > n+1$, and
      assume that for $j=1, 2$ we have sets $\Lambda_j$, $\Lambda_j^*$
      with $\Lambda_j$ convex and $\Lambda_j^* + \frac{1}{\mb{R}_0}
      \subset \Lambda_j \subset \{ \frac{1}{16}\les |\xi| \les 16\}$,
      and phases $\Phi_j: \Lambda_j \rightarrow \RR$ such that
        $$ \sup_{1\les |\kappa|\les N} \| \p^\kappa \Phi_j \|_{L^\infty(\Lambda_j)} \les \mb{D}_2.$$

\subsection{Wave Packets}\label{subsect:wave-packets}  Let $R \g 1$ and  define the cylinder
	$$ Q_R = \big\{ (t, x) \in \RR^{1+n}\,\big|\,  \tfrac{R}{2} < t < R, \, |x|< R \},$$
 and $\mc{X} = R^{\frac{1}{2}} \ZZ^n \times R^{-\frac{1}{2}} \ZZ^n$.
Define
        $$\mc{X}_j = \{ (x_0, \xi_0) \in \mc{X} \mid \xi_0 \in \Lambda^*_j + 3R^{-\frac{1}{2}} \}$$
        to be the set of phase points which are within
        $3R^{-\frac{1}{2}}$ of $\Lambda^*_j$. Note that provided $R \g
        (3 R_0)^2$, if $\gamma = (x_0, \xi_0) \in \mc{X}_j$, then
        $\xi_0 \in \Lambda_j$. Given a point $\gamma = (x_0, \xi_0)
        \in \mc{X}$ in phase-space, we let $x(\gamma) = x_0$ and
        $\xi(\gamma)= \xi_0$ denote the projections onto the first and
        second components respectively. Fix $\eta, \rho \in \s(\RR^n)$
        such that $\supp \widehat{\eta} \subset \{ |\xi| \les 1\}$,
        $\supp \rho \subset \{ |\xi| \les 1\}$, and for all $x, \xi
        \in \RR^n$
    $$ \sum_{k \in \ZZ^n} \eta( x - k) = \sum_{k \in \ZZ^n} \rho(\xi- k ) = 1.  $$
    Given $\gamma \in \mc{X}$ and $f \in L^2_x(\RR^n)$, define the
    phase-space localisation operator
	$$ \big(L_\gamma f\big)(x) = \eta\Big( \frac{x - x(\gamma)}{R^\frac{1}{2}} \Big) \Big[\rho\Big( \frac{  -i \nabla - \xi(\gamma)}{R^{-\frac{1}{2}}} \Big) f\Big](x). $$
        Note that by definition we have
	$$ f = \sum_{\gamma \in \mc{X}} L_\gamma f, \qquad \supp \widehat{L_\gamma f} \subset \{ \xi\in \RR^n \mid |\xi - \xi(\gamma)| \les  2R^{-\frac{1}{2}} \}.$$
        Moreover, letting $w_\gamma(x) = ( 1 + \frac{ |x -
          x(\gamma)|}{R^{\frac{1}{2}}} )^{N-1 + \frac{n+1}{2}}$, for any
        $\Gamma \subset \mc{X}$ we have the orthogonality bounds
	\begin{equation}\label{eqn -orthog of phase space loc
            operator}
          \Big\| \sum_{\gamma \in \Gamma} L_\gamma f \Big\|_{L^2_x} \lesa \bigg( \sum_{\gamma \in \Gamma} \| w_\gamma(x) L_\gamma f(x) \|_{L^2_x}^2 \bigg)^\frac{1}{2} \lesa \| f\|_{L^2_x}.
	\end{equation}
        To simplify notation slightly, we define the slightly larger
        phase-space localisation operators $L^\sharp_\gamma =
        \omega_\gamma(x) L_\gamma$. It is worth noting that
        $L^\sharp_\gamma f $ no longer has compact Fourier support,
        this does not pose any problems in the arguments to follow, as
        the only properties that we require are the trivial bound $ \|
        L_\gamma f \|_{L^2_x} \les \| L_\gamma^\sharp f\|_{L^2_x}$ and
        the orthogonality bound in (\ref{eqn -orthog of phase space
          loc operator}).

        To define wave packets, we conjugate the phase-space
        localisation operator $L_\gamma$ with the flow $e^{ i t
          \Phi_j(- i \nabla)}$.

        \begin{definition}[Wave Packets]\label{def:wave-packets} Let
          $j=1, 2$, $R\g (3R_0)^2$, and $u \in L^\infty_t
          L^2_x(\RR^{1+n})$. Given a point $\gamma_j \in \mc{X}_j$, we
          define
	$$ \big(\mc{P}_{\gamma_j} u\big)(t) = e^{i t \Phi_j(-i \nabla)} L_{\gamma_j} \Big( e^{  - i t \Phi_j(-i\nabla)} u(t) \Big).$$
        Similarly, we define
        $$ \big(\mc{P}_{\gamma_j}^\sharp u\big)(t) = e^{i t \Phi_j(-i \nabla)} L_{\gamma_j}^\sharp \Big( e^{  - i t \Phi_j(-i\nabla)} u(t) \Big).$$
      \end{definition}

      We also require the associated tubes $T_\gamma$.

\begin{definition}[Tubes]\label{def:tubes}
  Let $j=1, 2$ and $\gamma_j \in \mc{X}_j$. Then we define the tube
  $T_{\gamma_j} \subset \RR^{1+n}$ as
        $$ T_{\gamma_j} = \left\{ (t,x) \in \RR^{1+n} \middle| \tfrac{R}{2} \les t \les R, \,\, \big| x -  x(\gamma) + t \nabla\Phi_j\big(\xi(\gamma)\big)\big|\les R^{\frac{1}{2}} \right\}.$$
      \end{definition}

      The most important properties of the wave packets $P_{\gamma_j}
      u$ are summarised in the following.

\begin{proposition}[Properties of Wave Packets]\label{prop - properties of wave packets}
  Let $j=1, 2$. For any $R \g (3\mb{R}_0)^2$, $f \in L^2_x$ with
  $\supp \widehat{f} \subset \Lambda^*_j$, and $u = e^{ it
    \Phi_j(-i\nabla)}f$, we have $u =\sum_{\gamma_j \in \mc{X}_j}
  \mc{P}_{\gamma_j} u$, $\supp \widehat{\mc{P}_{\gamma_j}u} \subset \{
  |\xi - \xi(\gamma)| \les 2 R^{-\frac{1}{2}}\}$, and given any
  $\Gamma_j \subset \mc{X}_j$ we have the orthogonality bound
  \begin{equation}\label{eqn -orthog of wave packets}
    \Big\| \sum_{\gamma_j \in \Gamma_j} \mc{P}_{\gamma_j} u  \Big\|_{L^\infty_t L^2_x} \lesa \bigg( \sum_{\gamma_j \in \Gamma_j} \|  L_{\gamma_j}^\sharp f \|_{L^2_x}^2 \bigg)^\frac{1}{2} \lesa \| f \|_{L^2_x}.
  \end{equation}
  Moreover, the wave packets $\mc{P}_{\gamma_j}u$ are concentrated on
  the tubes $T_{\gamma_j}$ in the sense that for every $r \g R^\frac{1}{2}$, and
  any ball $B \subset \RR^{1+n}$, we have the bound
  \begin{equation}\label{eqn - conc of wave packets} \bigg\|
    \sum_{\substack{\gamma_j \in \Gamma_j \\ \dist(T_{\gamma_j},
        B)>r}} \mc{P}_{\gamma_j} u \bigg\|_{L^\infty_{t, x}(B \cap
      Q_R)} \lesa \Big( \frac{r}{R^\frac{1}{2}}\Big)^{\frac{n+3}{2} -
      N} \bigg( \sum_{\gamma_j \in \Gamma_j}  \| L^\sharp_{\gamma_j} f \|_{L^2_x}^2\bigg)^\frac{1}{2}
  \end{equation}
  Here, the implied constants depend only on $\mb{R}_0, \mb{D}_2,
  N$ and $n\g 2$.
\end{proposition}

      \begin{proof}
    This result is somewhat standard, see for instance  \cite[Lemma
        4.1]{Tao2003a} and \cite[Lemma 2.2]{Lee2006} for related estimates. We  only prove the localisation property \eqref{eqn - conc of wave
          packets}, as the remaining properties follow directly from
        the definition of $\mc{P}_{\gamma}$, together with the
        analogous properties of the phase-space localisation operator
        $L_{\gamma}$. Let $\gamma_j = (x_0, \xi_0)$ and write
	\begin{align*}
          \mc{P}_{\gamma_j} u(t, x) &= \int_{\RR^n} \widehat{(L_{\gamma_j} f)}(\xi) e^{ i t \Phi_j(\xi)} e^{ i x \cdot \xi} d\xi \\
          &= \int_{\RR^n} K_{\xi_0}(t, x-y) (L_{\gamma_j} f)(y) dy
	\end{align*}
        where, as in the proof of Lemma \ref{lem - dispersion}, the
        kernel is given by $K_{\xi_0}(t, x) = \int_{\RR^n} \rho(
        R^{\frac{1}{2}}(\xi - \xi_0)) e^{ i t \Phi_j(\xi)} e^{ i x
          \cdot \xi} d\xi$. Note that, as in \eqref{eqn - lem
          dispersion - decay of kernel}, integrating by parts $N-1$
        times, and using the fact that $|t| \les R$, $R \gg1$, we
        deduce that
	$$K_{\xi_0}(t, x) \lesa  R^{-\frac{n}{2}} \Big( 1 + \frac{| x + t \nabla\Phi_j(\xi_0) |}{R^\frac{1}{2}} \Big)^{1-N}. $$
        Plugging this bound into the identity for $\mc{P}_{\gamma_j}
        u(t, x)$, we deduce that
	\begin{align*} | \mc{P}_{\gamma_j} u(t, x) | &\lesa  R^{-\frac{n}{2}} \Big( 1 + \frac{| x -x_0 + t \nabla\Phi_j(\xi_0) |}{R^\frac{1}{2}} \Big)^{1-N} \int_{\RR^n}  \Big( 1 + \frac{| y-x_0 |}{R^\frac{1}{2}} \Big)^{N-1} | L_{\gamma_j} f(y)|  dy\\
          &\lesa R^{-\frac{n}{4}}  \Big( 1 + \frac{| x -x_0 + t \nabla\Phi_j(\xi_0)
            |}{R^\frac{1}{2}} \Big)^{1-N} \| L^\sharp_{\gamma_j} f \|_{L^2_x}
        \end{align*}
Since there are $\mc{O}(R^{\frac{n}{2}})$ choices of $\xi_0$, and
    $$|x - x_0 + t \nabla \Phi_j(\xi_0)| = \big|(t, x) - \big( t, x_0 - t\nabla \Phi_j(\xi_0)\big)| \g \text{dist}\big((t,x), T_{\gamma_j}\big), $$
an application of H\"older's inequality gives for any $(t, x) \in B$
    \begin{align*}
      \sum_{\substack{\gamma_j \in \Gamma_j \\ \text{dist}(T_{\gamma_j}, B) \g r }}  |\mc{P}_{\gamma_j} u(t, x)| &\lesa R^{-\frac{n}{4}} \bigg( \sum_{\substack{\gamma_j \in \mc{X}_j \\ \text{dist}(T_{\gamma_j}, B) \g r }}  \Big( 1 + \frac{| x -x_0 + t \nabla\Phi_j(\xi_0) |}{R^\frac{1}{2}} \Big)^{2-2N}\bigg)^\frac{1}{2} \bigg( \sum_{\gamma_j \in \Gamma_j }   \big\| L^\sharp_{\gamma_j} f \big\|_{L^2_x}^2 \bigg)^\frac{1}{2} \\
      &\lesa  \Big( \frac{r}{R^\frac{1}{2}} \Big)^{\frac{n+3}{2} - N}  \sup_{\xi_0}  \bigg( \sum_{x_0 \in R^\frac{1}{2} \ZZ^n} \Big( 1 + \frac{| x -x_0 + t \nabla\Phi_j(\xi_0) |}{R^\frac{1}{2}} \Big)^{-n-1} \bigg)^\frac{1}{2} \bigg( \sum_{\gamma_j \in \Gamma_j }   \big\| L^\sharp_{\gamma_j} f \big\|_{L^2_x}^2 \bigg)^\frac{1}{2}\\
      &\lesa \Big( \frac{r}{R^\frac{1}{2}} \Big)^{\frac{n+3}{2} - N}\bigg( \sum_{\gamma_j \in \Gamma_j }   \big\| L^\sharp_{\gamma_j} f \big\|_{L^2_x}^2 \bigg)^\frac{1}{2}
    \end{align*}
    as required.
  \end{proof}

\subsection{Atomic Spaces and Wave Packets}\label{subsect:atomic-spaces}
Closely related to the $V^p$ spaces, are the slightly smaller $U^p$
spaces, see \cite{Koch2005,Hadac2009,Koch2014}.

\begin{definition}\label{def:up} Let $1\leq p<\infty$.
  A function $\rho: \RR \rightarrow L^2_x$ is called a $U^p$
  \emph{atom} if there exists a decomposition $ \rho = \sum_{J \in
    \mc{I}} \ind_J(t) f_J $ subordinate to a finite partition $\mc{I}
  = \{ (-\infty, t_1), [t_2, t_3),\ldots, [t_N, \infty) \}$ of $\RR$,
  such that
        $$ \|f_J\|_{\ell^p_JL^2_x}:=\Big(\sum_{J \in \mc{I}} \| f_J \|_{L^2_x}^p \Big)^{\frac1p}\les 1. $$
        The atomic Banach space $U^p$ is then defined as
    $$ U^p = \Big\{ \sum_{j} c_j \rho_j \, \Big| \, (c_j) \in \ell^1(\NN), \, \text{$\rho_j$ a $U^p$ atom} \Big\}$$
    with the induced norm
    $$ \| \rho \|_{U^p} = \inf_{ \substack{ \rho = \sum_k c_k \phi_k \\ \phi_k \text{ $U^p$ atom}}  } \sum_k |c_k|. $$
    The space $U^p_\Phi$ is the set of all $u: \RR \rightarrow L^2_x$
    such that $e^{ - i t \Phi(-i\nabla)} u \in U^p$ with the obvious
    norm.
  \end{definition}

Let $u = \sum_J \ind_J(t) e^{ i t \Phi_j(-i\nabla)} f_J$ be a
$U^2_{\Phi_j}$ atom. Since $\ind_J(t)$ commutes with spatial Fourier
multipliers, we have
    $$ \mc{P}_{\gamma_j} u = \sum_J \ind_J(t) e^{ i t \Phi_j(-i\nabla)} L_{\gamma_j} f_J,\text{ and } \mc{P}^\sharp_{\gamma_j} u = \sum_J \ind_J(t) e^{ i t \Phi_j(-i\nabla)} L_{\gamma_j}^\sharp f_J.$$
    Proposition \ref{prop - properties of wave packets} gives the
    following.

\begin{corollary}[Wave Packets for $U^2_{\Phi_j}$ atoms]\label{cor - wave packets on U2}
  Let $j=1, 2$. For any $R \g (3\mb{R}_0)^2$ and $U^2_{\Phi_j}$
  atom $u = \sum_J \ind_J(t) e^{ it \Phi_j(-i\nabla)} f_J$ with $\supp
  \widehat{u} \subset \Lambda^*_j$, we have $u =\sum_{\gamma_j \in
    \mc{X}_j} \mc{P}_{\gamma_j} u$, $\supp
  \widehat{\mc{P}_{\gamma_j}u} \subset \{ |\xi - \xi(\gamma)| \les
  2R^{-\frac{1}{2}}\}$, and given any $\Gamma_j \subset \mc{X}_j$ we
  have the orthogonality bound
  \begin{equation}\label{eqn -orthog of wave packets atoms}
    \Big\| \sum_{\gamma_j \in \Gamma_j} \mc{P}_{\gamma_j} u  \Big\|_{L^\infty_t L^2_x} \lesa \bigg( \sum_{\gamma_j \in \Gamma_j} \|  L_{\gamma_j}^\sharp f_J \|_{\ell^2_J L^2_x}^2 \bigg)^\frac{1}{2} \lesa \| f_J \|_{\ell^2_JL^2_x}.
  \end{equation}
  Moreover, the wave packets $\mc{P}_{\gamma_j}u$ are concentrated on
  the tubes $T_{\gamma_j}$ in the sense that for every $r \g R^\frac{1}{2}$, and
  any ball $B \subset \RR^{1+n}$, we have the bound
  \begin{equation}\label{eqn - conc of wave packets atoms} \bigg\|
    \sum_{\substack{\gamma_j \in \Gamma_j \\ \dist(T_{\gamma_j},
        B)>r}} \mc{P}_{\gamma_j} u \bigg\|_{L^\infty_{t, x}(B \cap
      Q_R)} \lesa \Big( \frac{r}{R^\frac{1}{2}}\Big)^{\frac{n+3}{2} -
      N}\bigg( \sum_{\gamma_j \in \Gamma_j} \|  L_{\gamma_j}^\sharp f_J \|_{\ell^2_J L^2_x}^2 \bigg)^\frac{1}{2}.
  \end{equation}
  Here, the implied constants depend only on $\mb{R}_0, \mb{D}_2,
  N$ and $n\g 2$.
\end{corollary}

\subsection{Sets and Relations of Tubes}\label{subsec - tube
  definitions}

We repeat the definitions and notation used by Tao \cite{Tao2003a},
but as above we adopt the point of view that the basic objects are the
phase space elements $\gamma \in \mc{X}_j$, rather than the associated
tubes $T_{\gamma_j}$.

For $\delta>0$, let $\mc{B}$ be a collection of (space-time) balls of
radius $R^{1-\delta}$ which form a finitely overlaping cover of
$Q_R$. Similarly let $\mb{q}$ denote a collection of finitely
overlapping cubes $q$ of radius $R^\frac{1}{2}$ which cover the
cylinder $Q_R$. Let $R^\delta q$ denote a cube of radius $R^{\delta +
  \frac{1}{2}}$ with the same centre as $q$. Given a collection
$\Gamma_j \subset \mc{X}_j$, and a cube $q\in \mb{q}$, we define
		$$ \Gamma_j(q) = \{ \gamma_j \in \Gamma_j \mid T_{\gamma_j} \cap R^\delta q \ne \varnothing \}$$
                so $\Gamma_j(q)$ is the subcollection of our
                phase-space decomposition, such that the associated
                tube $T_{\gamma_j}$ intersects a slight enlargement of
                the cube $q \in \mb{q}$. In the remainder of this
                subsection, the implied constants may depend on $n\g
                2$ only. Given $1 \les \mu_1,\mu_2 \lesa R^{100n}$,
                define
	$$ \mathbf{q}(\mu_1, \mu_2) = \{ q \in \mathbf{q} \mid \mu_j \les \# \Gamma_j(q) < 2 \mu_j, \, j=1, 2 \}.$$
        Thus, roughly, $\mb{q}(\mu_1, \mu_2)$ restricts to those
        elements of $\mb{q}$ which are intersected by $\mu_j$ tubes
        $T_{\gamma_j}$, $\gamma_j \in \Gamma_j$. Given $\gamma_j \in
        \Gamma_j$, we let
	$$\lambda(\gamma_j, \mu_1, \mu_2) = \#\{ q \in \mb{q}(\mu_1, \mu_2) \mid T_{\gamma_j} \cap R^\delta q \not = \varnothing \}$$
        and for every $1 \les \lambda_j \lesa R^{100n}$ we define
	$$ \Gamma_j[\lambda_j, \mu_1, \mu_2] = \{ \gamma_j \in \Gamma_j \mid \lambda_j \les \lambda(\gamma_j, \mu_1, \mu_2) < 2 \lambda_j \}.$$
        So $\Gamma_j[\lambda_j, \mu_1, \mu_2]$ essentially restricts
        to $\gamma_j \in \Gamma_j$, such that the associated tubes
        $T_{\gamma_j}$ intersect $\lambda_j$ cubes in $\mb{q}(\mu_1,
        \mu_2)$. Clearly
	$$ \bigcup_{1\les \lambda_j, \mu_1, \mu_2 \lesa R^{100n}} \Gamma_j(\lambda_j, \mu_1, \mu_2) = \Gamma_j.$$
        The following relation $\sim$ between balls in $\mc{B}$ and
        $\gamma_j \in \Gamma_j$ plays a key role in the arguments to
        follow.

\begin{definition}\label{defn - tubes and balls relation}
  Given $\gamma_j \in \Gamma_j[\lambda_j, \mu_1, \mu_2]$, we let
  $B(\gamma_j, \lambda_j, \mu_1, \mu_2) \in \mc{B}$ denote a ball
  which maximises
	$$ \# \{  q \in \mb{q}(\mu_1, \mu_2) \mid T_{\gamma_j} \cap R^\delta q \not =  \varnothing, \,\, q\cap B(\gamma_j, \lambda_j, \mu_1, \mu_2) \not= \varnothing \}.$$
        If $B \in \mc{B}$, and $\gamma_j \in \Gamma_j[\lambda_j,
        \mu_1, \mu_2]$, we then define $\gamma_j \sim_{\lambda_j,
          \mu_1, \mu_2} B$, if $B \subset 10 B(\gamma_j, \lambda_j,
        \mu_1, \mu_2)$. To extend this definition to general points
        $\gamma_j \in \Gamma_j$, we simply say that $\gamma_j \sim B$
        if there exists some $1 \les \lambda_j, \mu_1, \mu_2 \lesa
        R^{100n}$ such that $\gamma_j \sim_{\lambda_j, \mu_1, \mu_2}
        B$.
      \end{definition}

\begin{remark}\label{rem - consequences of tube balls relation}
  This definition has the following important consequences.
  \begin{enumerate}
  \item Let $\gamma_j \in \Gamma_j$ and consider the set $\{ B \in
    \mc{B} \mid \gamma_j \sim B\}$. Since there are at most $\mc{O}(
    R^\epsilon)$ dyadic $1\les \lambda_j, \mu_1, \mu_2 \les R^{100n}$
    such that $\gamma_j \in \Gamma_j[\lambda_j, \mu_1, \mu_2]$, and
    only $\mc{O}(1)$ balls $B$ such that $\gamma_j \sim_{\lambda_j,
      \mu_1, \mu_2} B$, we have
		$$   \# \{ B \in \mc{B} \mid \gamma_j \sim B \} \les \sum_{\substack{1 \les \lambda_j, \mu_1, \mu_2\les R^{100n} \\ \gamma_j \in \Gamma_j[\lambda_j, \mu_1, \mu_2] }} \# \{ B \in \mc{B} \mid  \gamma_j \sim_{\lambda_j, \mu_1, \mu_2} B \} \lesa \sum_{1 \les \lambda_j, \mu_1, \mu_2 \les R^{100n}} 1 \lesa R^\epsilon. $$
			
              \item Fix $1\les \lambda_1, \mu_1, \mu_2 \lesa R^{100n}$
                and let $\gamma_j \in \Gamma_j[\lambda_j, \mu_1,
                \mu_2]$. By definition, we have
		\begin{align*}
                  \lambda_j &\les \# \{ q \in \mb{q}(\mu_1, \mu_2) \mid T_{\gamma_j} \cap R^\delta q \not = \varnothing \}\\
                  &\les \sum_{B \in \mc{B}} \# \{ q \in \mb{q}(\mu_1, \mu_2) \mid T_{\gamma_j} \cap R^\delta q \not = \varnothing, \,\, q \cap B \not = \varnothing \}\\
                  &\les \#\mc{B}\, \# \{ q \in \mb{q}(\mu_1, \mu_2)
                  \mid T_{\gamma_j} \cap R^\delta q \not =
                  \varnothing, \,\, q \cap B(\gamma_j, \lambda_1,
                  \mu_1, \mu_2) \not = \varnothing \}
		\end{align*}
                where we used the maximal property of the ball
                $B(\gamma_j, \lambda_j, \mu_1, \mu_2)$. Therefore, as
                $\# \mc{B} \lesa R^{(n+1)\delta}$, we deduce the lower
                bound
	$$  \# \{ q \in \mb{q}(\mu_1, \mu_2) \mid T_{\gamma_j} \cap R^\delta q \not = \varnothing, \,\, q \cap B(\gamma_j, \lambda_j, \mu_1, \mu_2) \not = \varnothing \} \gtrsim R^{-(n+1)\delta} \lambda_j.$$
      \end{enumerate}
    \end{remark}

    \section{A Local Bilinear Restriction Estimate}\label{sect:loc}

    The main step in the proof of Theorem \ref{thm - bilinear Lp
      estimate} is prove the following spatially localised version in $U^2_\Phi$.

\begin{theorem}\label{thm - loc bilinear Lp estimate}
  Let $n\g 2$ and $\alpha>0$. Let $\mb{R}_0 \g 1$ and $
  \mb{D}_1, \mb{D}_2 >0$. For $j=1, 2$, let $\Lambda_j, \Lambda_j^*
  \subset \{ \frac{1}{16} \les |\xi| \les 16\}$ with $\Lambda_j$
  convex and $\Lambda^*_j + \frac{1}{\mb{R}_0} \subset
  \Lambda_j$. There exists $N \in \NN$ and a constant $C>0$ such that,
  for any phases $\Phi_1$ and $\Phi_2$ satisfying Assumption
  \ref{assump on phase}, any $u \in U^2_{\Phi_1}$, $v \in
  U^2_{\Phi_2}$ with $\supp \widehat{u}(t) \subset \Lambda^*_1$,
  $\supp \widehat{v}(t) \subset \Lambda^*_2$, and any $R \g 1$, we
  have
		$$  \|u v \|_{L^\frac{n+3}{n+1}_{t, x}(Q_R)} \les C R^\alpha  \| u \|_{U^2_{\Phi_1}} \| v \|_{U^2_{\Phi_2}}$$
              \end{theorem}
In the remainder of this section we give the proof of Theorem \ref{thm - loc bilinear Lp estimate}. The proof is broken up into three key steps. The first step is use an induction on scales argument to reduce to proving an $L^2_{t, x}$ bound. We then use the localisation properties of the wave packet decomposition to show that the $L^2_{t, x}$ bound follows from a combinatorial Kakeya type bound. The final step is prove the combinatorial estimate using a ``bush'' argument.

\subsection{Induction on Scales}
Let $\alpha>0$ and fix the constants $\mb{R}_0 \g 1$, $
\mb{D}_1, \mb{D}_2 >0$.  Fix $N = \frac{\alpha + 1}{\alpha} (100n)^2$. For $j=1, 2$, let $\Lambda_j, \Lambda_j^* \subset \{
\frac{1}{16} \les |\xi| \les 16\}$ with $\Lambda_j$ convex and
$\Lambda^*_j + \frac{1}{\mb{R}_0} \subset \Lambda_j$. It is enough to
show that there exists a constant $C>0$ such that, for any phases $\Phi_1$
and $\Phi_2$ satisfying Assumption \ref{assump on phase}, any $R \g (3\mb{R}_0)^2$, and any $U^2_{\Phi_j}$ atoms $u = \sum_J
\ind_J(t) e^{ i t \Phi_1(-i\nabla)} f_J$,  $v = \sum_{J'}
\ind_{J'}(t) e^{ it \Phi_2(-i \nabla)} g_{J'}$ with $\supp
\widehat{f} \subset \Lambda^*_1$, $\supp \widehat{g}_{J'} \subset
\Lambda^*_2$, we have
\begin{equation}\label{eqn - sec induc on scales - loc bilinear est atom}
  \|u v \|_{L^\frac{n+3}{n+1}_{t, x}(Q_R)} \les C R^{2\alpha}. \end{equation}
To simplify the notation to follow, we now work under the assumption that any implicit constants may now depend on $\alpha$, $n\g 2$, and the constants $\mb{R}_0,\mb{D}_1, \mb{D}_2$, but will be independent of $R$ and the particular choice of phases $\Phi_j$ satisfying Assumption \ref{assump on phase}.

The proof of \eqref{eqn - sec induc on scales - loc bilinear est atom}
proceeds along the same lines as Tao's argument for the paraboloid
\cite{Tao2003a}. Namely, we use an induction on scales argument to
deduce the estimate at scale $R$, by applying a weaker estimate at a
smaller scale $R^{1-\delta}$.  We start by observing that it suffices to show that, for every $\Gamma_j \subset \mc{X}_j$ such that $\#\Gamma_j \les R^{10n}$, and any $\beta \g \alpha$, we have
\begin{equation}\label{eqn - sec induc on scales - induction assump} \begin{split}
         \bigg\| \sum_{\gamma_j \in \Gamma_j} \mc{P}_{\gamma_1} u
         \mc{P}_{\gamma_2} v \bigg\|_{L^{\frac{n+3}{n+1}}_{t,
             x}(Q_R)} \lesa R^\beta \big( \# \Gamma_1 \#
         \Gamma_2 \big)^\frac{1}{2} \sup_{\gamma_j \in \Gamma_j} \|
         L^\sharp_{\gamma_1} f_J \|_{\ell^2_J L^2_x} \|
         L^\sharp_{\gamma_2} g_{J'} \|_{\ell^2_{J'}
           L^2_x}. \end{split}
     \end{equation}
To deduce (\ref{eqn - sec induc on scales - loc bilinear est atom}) from (\ref{eqn - sec induc on scales - induction assump}), we let
	$$ \mc{X}_1(\nu_1) = \big\{ \gamma_1 \in \mc{X}_1 \mid \nu_1 \les \| L^\sharp_{\gamma_1} f_J \|_{\ell^2_J L^2_x} \les 2 \nu_1, \,\, T_{\gamma_1} \cap 2Q_R \not = \varnothing \big\}$$
        and $\mc{X}_2(\nu_2)$ similarly where $\nu_j \in 2^\ZZ$. An
        application of Corollary \ref{cor - wave packets on U2} gives the decomposition $u = \sum_{\gamma_j \in \mc{X}_j} \mc{P}_j u$ as well as the bounds
	$$ \bigg\| \sum_{\substack{\gamma_1 \in \mc{X}_1 \\  T_{\gamma_1} \cap 2Q_R = \varnothing} }  \mc{P}_{\gamma_1} u \bigg\|_{L^\infty_{t, x}(Q_R)} \lesa R^{-99n} $$
        and
 $$ \bigg( \sum_{\gamma_j \in \mc{X}_j} \| \mc{P}_{\gamma_1} u \|_{L^\infty_t L^2_x}^2 \bigg)^\frac{1}{2}\lesa \bigg( \sum_{\gamma_j \in \mc{X}_j} \| L^\sharp_{\gamma_1} f_J \|_{\ell^2_J L^2_x}^2 \bigg)^\frac{1}{2} \lesa 1 .$$
 The analogous bounds hold for $v$.  Moreover $\# \{ \gamma_j \in
 \mc{X}_j \mid T_{\gamma_j} \cap 2Q_R \not = \varnothing \} \lesa
 R^{n+1}$. Collecting these properties together, we deduce that
 $\mc{X}_1(\nu_1) = \varnothing$ for $\nu_1 \gg 1$ and
	$$ \bigg\| u - \sum_{R^{-100n} \les \nu_1 \lesa 1} \sum_{\gamma_1 \in \mc{X}_1(\nu_1)} \mc{P}_{\gamma_1} u \bigg\|_{L^\infty_{t, x}(Q_R)} \lesa R^{-90n}.$$
        A similar argument shows that
		$$ \bigg\| v - \sum_{R^{-100n} \les \nu_2 \lesa 1} \sum_{\gamma_2 \in \mc{X}_2(\nu_2)} \mc{P}_{\gamma_2} v \bigg\|_{L^\infty_{t, x}(Q_R)} \lesa R^{-90n}.$$
                Therefore, applying the bound (\ref{eqn - sec induc on
                  scales - induction assump}) with $\Gamma_j =
                \mc{X}_j(\nu_j)$ and $\beta = \alpha$,  we obtain
                \begin{align*}
                  \| u v \|_{L^{\frac{n+3}{n+1}}_{t, x}(Q_R)} &\les \bigg\|uv - \sum_{R^{-100n} \les \nu_j\lesa 1} \sum_{\gamma_j \in \mc{X}_j(\nu_j)} \mc{P}_{\gamma_1} u \mc{P}_{\gamma_2} v \bigg\|_{L^{\frac{n+3}{n+1}}_{t, x}(Q_R)} \\
                  &\qquad \qquad \qquad \qquad+ \sum_{R^{-100n} \les \nu_j \lesa 1} \bigg\| \sum_{\gamma_j \in \mc{X}_j(\nu_j)} \mc{P}_{\gamma_1} u \mc{P}_{\gamma_2} v \bigg\|_{L^{\frac{n+3}{n+1}}_{t, x}(Q_R)} \\
                  &\lesa  1   +  \log(R) R^\alpha \sup_{\nu_j} \Big(  \big( \# \mc{X}_1(\nu_1) \# \mc{X}_2(\nu_2) \big)^\frac{1}{2} \sup_{\gamma_j \in \mc{X}_j(\nu_j)} \| L^\sharp_{\gamma_1} f_J \|_{\ell^2_J L^2_x} \| L^\sharp_{\gamma_2} g_{J'} \|_{\ell^2_{J'} L^2_x}\Big) \\
                  &\lesa R^{2\alpha}
                \end{align*}
where the last line follows from the orthogonality properties of the phase space localisation operators (\ref{eqn -orthog of phase space loc
            operator}). Hence \eqref{eqn - sec induc on scales - loc bilinear est atom} follows.

The proof of (\ref{eqn - sec induc on scales - induction assump}) proceeds via an induction on scales argument. The first step is to note that we already have (\ref{eqn - sec induc on scales - induction assump})
provided we take $\beta>0$ sufficiently large. Indeed, a crude
argument by H\"older and Bernstein inequalities implies the bound with
$\beta=\frac{n+1}{n+3}$ (which could be improved by using linear
Strichartz estimates as indicated in Remark \ref{rem:strichartz}). Suppose we could show that, if \eqref{eqn - sec induc on scales - induction assump} holds for some $\beta> \alpha$, then for every
$\epsilon>0$ we have
\begin{equation}\label{eqn - sec induc on scales - loc induc
    conclusion}  \bigg\| \sum_{\gamma_j \in \Gamma_j} \mc{P}_{\gamma_1} u
         \mc{P}_{\gamma_2} v \bigg\|_{L^{\frac{n+3}{n+1}}_{t,
             x}(Q_R)} \lesa R^{2\epsilon}\big( R^{(1-\delta) \beta} + R^{D \delta}\big) \big( \# \Gamma_1 \#
         \Gamma_2 \big)^\frac{1}{2} \sup_{\gamma_j \in \Gamma_j} \|
         L^\sharp_{\gamma_1} f_J \|_{\ell^2_J L^2_x} \|
         L^\sharp_{\gamma_2} g_{J'} \|_{\ell^2_{J'}
           L^2_x}.\end{equation}
where $\delta =  \frac{\alpha}{D + \alpha}$ and $D\g 0$ is some constant which depends only on the dimension $n$. Then, since $D\delta<\alpha$, by taking $\epsilon>0$ sufficiently small, we deduce that we must have (\ref{eqn - sec induc on scales - induction assump}) for some $\beta'<\beta$. Iterating this argument then gives (\ref{eqn - sec induc on scales - induction assump}) for $\beta = \alpha$. Consequently, our aim is to prove \eqref{eqn - sec induc on scales - loc induc conclusion}, under the assumption that we already have (\ref{eqn - sec induc on scales - induction assump}) for some $\beta >\alpha$.

We now fix $\Gamma_j \subset \mc{X}_j$ such that $\#\Gamma_j \les R^{10n}$, and $\beta > \alpha$. Let $\mc{B}$ denote a collection of balls $B$ of radius $R^{1-\delta}$ which form a finitely overlapping cover of $Q_R$. Let $\sim$ denote the relation between points $\gamma_j \in \Gamma_j$ and balls $B \in
\mc{B}$ given by Definition \ref{defn - tubes and balls relation}. It is important to note that the relation $\sim$ depends only on the fixed sets $\Gamma_j$, and \emph{not} on $u$ and $v$. Decompose
	$$ \bigg\| \sum_{\gamma_j \in \Gamma_j} \mc{P}_{\gamma_1} u \mc{P}_{\gamma_2} v \bigg\|_{L^{\frac{n+3}{n+1}}_{t, x}(Q_R)}
        \les \sum_{B\in \mc{B}} \bigg\| \sum_{\substack{\gamma_j \in
            \Gamma_j \\ \gamma_j \sim B}} \mc{P}_{\gamma_1} u
        \mc{P}_{\gamma_2} v \bigg\|_{L^{\frac{n+3}{n+1}}_{t, x}(B)} +
        \sum_{B \in \mc{B}} \bigg\| \sum_{\substack{\gamma_j \in
            \Gamma_j\\ \gamma_1 \not \sim B \text{ or } \gamma_2 \not
            \sim B}} \mc{P}_{\gamma_1} u \mc{P}_{\gamma_2} v
        \bigg\|_{L^{\frac{n+3}{n+1}}_{t, x}(B)}. $$
For the first term, which contains the tubes which are concentrated on $B$,  we apply the induction assumption at scale $R^{1-\delta}$ to deduce that
	\begin{align*}
          \sum_{B\in \mc{B}}& \bigg\| \sum_{\substack{\gamma_j \in \Gamma_j \\ \gamma_j \sim B}} \mc{P}_{\gamma_1} u \mc{P}_{\gamma_2} v \bigg\|_{L^{\frac{n+3}{n+1}}_{t, x}(B)} \notag \\
          &\lesa R^{(1-\delta) \beta} \sum_{B\in \mc{B}}  \big( \# \{ \gamma_1 \in \Gamma_1\mid \gamma_1 \sim B\} \#\{ \gamma_2 \in \Gamma_2\mid \gamma_2 \sim B\} \big)^\frac{1}{2} \sup_{\gamma_j \in \Gamma_j}  \|
         L^\sharp_{\gamma_1} f_J \|_{\ell^2_J L^2_x} \|
         L^\sharp_{\gamma_2} g_{J'} \|_{\ell^2_{J'}
           L^2_x} \notag \\
           &\lesa R^{\epsilon} R^{ ( 1- \delta) \beta} (\#\Gamma_1 \#\Gamma_2)^\frac{1}{2} \sup_{\gamma_j \in \Gamma_j}  \|
         L^\sharp_{\gamma_1} f_J \|_{\ell^2_J L^2_x} \|
         L^\sharp_{\gamma_2} g_{J'} \|_{\ell^2_{J'}
           L^2_x}
        \end{align*}
   where the last line followed from (i) in Remark \ref{rem - consequences of tube balls relation}.  For the second term, as we can now safely lose factors of $R^\delta$, we may ignore the sum over the balls $B$ (as there are only $\mc{O}(R^{ \delta ( n+1)})$ balls). Thus, after replacing $D$ with $D-n-1$, we need to prove the bound
        \begin{equation}\label{eqn - sec induc on scales -
            nonconcentrated sum in Lp} \bigg\|
          \sum_{\substack{\gamma_j \in \Gamma_j\\ \gamma_1 \not \sim B
              \text{ or } \gamma_2 \not \sim B}} \mc{P}_{\gamma_1} u
          \mc{P}_{\gamma_2} v \bigg\|_{L^{\frac{n+3}{n+1}}_{t, x}(B)}
          \lesa R^{\epsilon + D \delta} \big( \# \Gamma_1 \#
          \Gamma_2 \big)^\frac{1}{2} \sup_{\gamma_j \in \Gamma_j} \|
          L^\sharp_{\gamma_1} f_J \|_{\ell^2_J L^2_x}\|
          L^\sharp_{\gamma_2} g_{J'} \|_{\ell^2_{J'}
            L^2_x}. \end{equation}
To this end, an application of H\"older together with the orthogonality property of the tube decomposition gives
        \begin{align*}
          \bigg\| \sum_{\substack{\gamma_j \in \Gamma_j \\ \gamma_1 \not \sim B \text{ or } \gamma_2 \not \sim B }} \mc{P}_{\gamma_1} u \mc{P}_{\gamma_2} v \bigg\|_{L^1_{t, x}(B)} &\lesa R \bigg( \sum_{\gamma_1 \in \Gamma_1}  \| L^\sharp_{\gamma_1} f_J \|_{\ell^2_J L^2_x}^2 \bigg)^\frac{1}{2} \bigg( \sum_{\gamma_2 \in \Gamma_2} \| L^\sharp_{\gamma_2} g_{J'} \|_{\ell^2_{J'} L^2_x}^2 \bigg)^\frac{1}{2}  \\
          &\lesa R \big( \# \Gamma_1 \# \Gamma_2 \big)^\frac{1}{2}
          \sup_{\gamma_j \in \Gamma_j }\| L^\sharp_{\gamma_1} f_J
          \|_{\ell^2_J L^2_x} \| L^\sharp_{\gamma_2} g_{J'}
          \|_{\ell^2_{J'} L^2_x}
        \end{align*}
        In particular, the convexity of the $L^p$ norms implies that
        (\ref{eqn - sec induc on scales - nonconcentrated sum in Lp})
        would follow from the $L^2_{t, x}$ bound
        \begin{equation}\label{eqn - sec induc on scales -
            nonconcentrated sum in L2}
          \begin{split} \bigg\| \sum_{\substack{\gamma_j \in \Gamma_j\\ \gamma_1 \not \sim B \text{ or } \gamma_2 \not \sim B}} & \mc{P}_{\gamma_1} u \mc{P}_{\gamma_2} v \bigg\|_{L^2_{t, x}(B)} \\
            &\lesa R^{\epsilon +  D\delta - \frac{n-1}{4}} \big( \# \Gamma_1 \# \Gamma_2 \big)^\frac{1}{2} \sup_{\gamma_j \in \Gamma_j} \| L^\sharp_{\gamma_1} f_J \|_{\ell^2_J L^2_x}\| L^\sharp_{\gamma_2} g_{J'} \|_{\ell^2_{J'} L^2_x}.
          \end{split}
        \end{equation}
        Thus we have reduced the problem of obtaining the
        $L^{\frac{n+3}{n+1}}_{t, x}$ estimate \eqref{eqn - sec induc
          on scales - loc induc conclusion}, to proving the
        $L^2_{t,x}$ bound \eqref{eqn - sec induc on scales -
          nonconcentrated sum in L2}.
\begin{remark}
  The fact that the above reduction can be done in $U^2_\Phi$, is the
  key reason why we can extend the homogeneous bilinear Fourier
  restriction estimates to $U^2_\Phi$.
\end{remark}

Our goal in the following two subsections is to prove the bound \eqref{eqn - sec induc
  on scales - nonconcentrated sum in L2}, and thus complete the proof of Theorem \ref{thm - loc bilinear Lp estimate}. As in the previous subsections,
we essentially follow the argument of Tao \cite{Tao2003a}, but apply
the results of Section \ref{sect:assu} in place of analogous results
for the paraboloid. The general strategy is to first use the
transversality  via
Lemma \ref{lem - key bilinear L2 bound I} to reduce to counting
intersections of tubes. The number of tubes is then controlled by
using \eqref{it:ass1} in Assumption \ref{assump on phase} via Lemma \ref{lem -
  surface C transverse} together with a ``bush'' argument. The
notation for various cubes and tubes introduced in Subsection
\ref{subsec - tube definitions} is used heavily in what follows.

\subsection{The $L^2$ Bound: Initial Reductions and Transversality}\label{subsect:l2bd-initial}

Recall that the ball $B \in \mc{B}$ is now fixed. Write
	 $$\sum_{\substack{ \gamma_j \in \Gamma_j, \\\gamma_1 \not \sim B \text{ or } \gamma_2 \not \sim B} } \mc{P}_{\gamma_1} u \mc{P}_{\gamma_2} v = \sum_{\substack{ \gamma_j \in \Gamma_j, \\ \gamma_1 \not \sim B} }\mc{P}_{\gamma_1} u \mc{P}_{\gamma_2} v     + \sum_{\substack{ \gamma_j \in \Gamma_j, \\ \gamma_1 \sim B \text{ and } \gamma_2 \not \sim B} } \mc{P}_{\gamma_1} u \mc{P}_{\gamma_2} v. $$
         We only prove the bound for the first term, as an identical
         argument can handle the second term (just replace $\Gamma_1$
         with $\{ \gamma_1 \in \Gamma_1 \mid \gamma_1 \sim B\}$ and
         reverse the roles of $u$ and $v$).  The first step is make a
         number of reductions exploiting the spatial localisation
         properties of the wave packets, together with a dyadic pigeon
         hole argument to fix various quantities. To this end,
         decompose into cubes $q \in \mb{q}$
		$$\bigg\| \sum_{\substack{ \gamma_j \in \Gamma_j, \\\gamma_1 \not \sim B} } \mc{P}_{\gamma_1} u \mc{P}_{\gamma_2} v      \bigg\|_{L^2_{t, x}(B)} \les \bigg( \sum_{q \in \mathbf{q}, q \subset 2B} \bigg\| \sum_{\substack{ \gamma_j \in \Gamma_j, \\ \gamma_1 \not \sim B } } \mc{P}_{\gamma_1} u \mc{P}_{\gamma_2} v      \bigg\|_{L^2_{t, x}(q)}^2 \bigg)^\frac{1}{2}.$$
                Note that the concentration property of the wave
                packet decomposition implies that
	$$ \bigg\| \sum_{\gamma_1 \in \Gamma_1, T_{\gamma_1} \cap R^\delta q = \varnothing} \mc{P}_{\gamma_1} u  \bigg\|_{L^\infty_{t, x}(q)}  \lesa R^{ - \delta (N- \frac{n+3}{2})} \big( \# \Gamma_1\big)^\frac{1}{2} \sup_{\gamma_1 \in \Gamma_1} \| L_{\gamma_1}^\sharp f_J \|_{\ell^2_J L^2_x}.$$
A similar bound holds for $v$. By our choice of $N$, we have $\delta (N - \frac{n+3}{2}) \g 100n$.  Therefore, as $\# \Gamma_j \lesa R^{10n}$ and $\# \mb{q}
        \lesa R^{2n}$, it suffices to prove
	\begin{equation}\label{eqn - sec L2 bound - sum over conc tubes localised to cubes} \begin{split} \bigg( \sum_{q \in \mathbf{q}, q \subset 2B} \bigg\| \sum_{\substack{ \gamma_j \in \Gamma_j(q), \\ \gamma_1 \not \sim B } } \mc{P}_{\gamma_1} u \mc{P}_{\gamma_2} v   &   \bigg\|_{L^2_{t, x}(q)}^2 \bigg)^\frac{1}{2}\\
            &\lesa R^{ \epsilon + D\delta - \frac{n-1}{4}} \big( \#
            \Gamma_1\big)^\frac{1}{2} \big(
            \#\Gamma_2\big)^\frac{1}{2} \sup_{\gamma_j \in \Gamma_j}\|
            L^\sharp_{\gamma_1} f_J \|_{\ell^2_J L^2_x}\|
            L^\sharp_{\gamma_2} g_{J'} \|_{\ell^2_{J'}
              L^2_x}. \end{split} \end{equation}
  Let $\Gamma_1^{ \not\sim B}(q) = \{ \gamma_1 \in \Gamma_1(q) \mid \gamma_1 \not \sim B \}$ and decompose into
	\begin{align*} \bigg( \sum_{q \in \mathbf{q}, q \subset 2B} \bigg\| \sum_{\substack{ \gamma_j \in \Gamma_j(q), \\ \gamma_1 \not \sim B } } &\mc{P}_{\gamma_1} u \mc{P}_{\gamma_2} v      \bigg\|_{L^2_{t, x}(q)}^2 \bigg)^\frac{1}{2}\\
          &\les \sum_{1\les \lambda_1, \mu_1, \mu_2 \lesa R^{100n} }
          \bigg( \sum_{q \in \mathbf{q}(\mu_1, \mu_2), q \subset 2B}
          \bigg\| \sum_{\substack{ \gamma_1 \in \Gamma_1^{\not \sim
                B}(q) \cap \Gamma_1[\lambda_1, \mu_1, \mu_2] \\
              \gamma_2 \in \Gamma_2(q) } } \mc{P}_{\gamma_1} u
          \mc{P}_{\gamma_2} v \bigg\|_{L^2_{t, x}(q)}^2
          \bigg)^\frac{1}{2}.
	\end{align*}
        Clearly, as we can freely lose $R^\epsilon$, (\ref{eqn - sec
          L2 bound - sum over conc tubes localised to cubes}) would
        follow from proving the estimate for fixed $\lambda_1, \mu_1,
        \mu_2$,
        \begin{equation}\label{eqn - sec L2 bound - sum over conc tubes localised to cubes dyadic} \begin{split} \bigg( \sum_{q \in \mathbf{q}(\mu_1, \mu_2), q \subset 2B} &\bigg\| \sum_{\substack{ \gamma_1 \in \Gamma_1^{\not \sim B}(q) \cap \Gamma_1[\lambda_1, \mu_1, \mu_2] \\ \gamma_2 \in \Gamma_2(q) } } \mc{P}_{\gamma_1} u \mc{P}_{\gamma_2} v      \bigg\|_{L^2_{t, x}(q)}^2 \bigg)^\frac{1}{2}\\
            &\lesa R^{ \epsilon + D\delta - \frac{n-1}{4}} \big( \#
            \Gamma_1\big)^\frac{1}{2} \big(
            \#\Gamma_2\big)^\frac{1}{2} \sup_{\gamma_j \in \Gamma_j}\|
            L^\sharp_{\gamma_1} f_J \|_{\ell^2_J L^2_x}\|
            L^\sharp_{\gamma_2} g_{J'} \|_{\ell^2_{J'}
              L^2_x}. \end{split} \end{equation} To make the notation
        slightly less cumbersome, we introduce the short hand
		$$ \Gamma^*_1(q) = \Gamma_1^{\not \sim B}(q) \cap \Gamma_1[\lambda_1, \mu_1, \mu_2].$$
                Given $q \in \mb{q}$ and $\mathfrak{h} \in \RR^{1+n}$,
                we define the set
	$$ \Gamma_1^{**}(q, \mathfrak{h}) = \Gamma_1^{**}[\lambda_1, \mu_1, \mu_2](q, \mathfrak{h}) = \big\{ \gamma_1 \in \Gamma_1^*(q) \,\,\big| \,\, \xi(\gamma_1) \in \Sigma_1(\mathfrak{h}) + \mc{O}(R^{-\frac{1}{2}}) \big\}. $$
        Thus $\Gamma_1^{**}(q, \mathfrak{h})$ consists of all
        $\gamma_1 \in \Gamma_1^*(q)$ such that $\xi(\gamma_1)$ lies
        within $C R^{-\frac{1}{2}}$ of the surface
        $\Sigma_1(\mathfrak{h})$. If we expand the square of the
        $L^2_{t, x}$ in (\ref{eqn - sec L2 bound - sum over conc tubes
          localised to cubes dyadic}) we get
	$$ \bigg\| \sum_{\substack{ \gamma_1 \in \Gamma_1^*(q) \\ \gamma_2 \in \Gamma_2(q) } } \mc{P}_{\gamma_1} u \mc{P}_{\gamma_2} v      \bigg\|_{L^2_{t, x}(q)}^2 \les \sum_{ \substack{\gamma_1 \in \Gamma_1^*(q) \\ \gamma_2' \in \Gamma_2(q) }} \sum_{\gamma_1' \in \Gamma^{*}_1(q)} \sum_{\substack{\gamma_2 \in \Gamma_2(q)}} \left| \left\lr{ \mc{P}_{\gamma_1} u \mc{P}_{\gamma_2} v, \mc{P}_{\gamma_1'} u \mc{P}_{\gamma_2'} v \right}_{L^2_{t, x}}\right|.$$
        We now exploit the Fourier localisation properties of the wave
        packets to deduce that the inner product vanishes unless
        \begin{equation}\label{eqn - lem L2 bound with fourier orthog
            - conditions on xi eta}
          \begin{split}
            \xi(\gamma_1) + \xi(\gamma_2) &= \xi(\gamma_1') + \xi(\gamma_2') + \mc{O}(R^{-\frac{1}{2}}) \\
            \Phi_1\big(\xi(\gamma_1)\big) +
            \Phi_2\big(\xi(\gamma_2)\big)
            &=\Phi_1\big(\xi(\gamma_1')\big)
            +\Phi_2\big(\xi(\gamma_2')\big) + \mc{O}(R^{-\frac{1}{2}})
          \end{split}
        \end{equation}
        In particular, if we take $\mathfrak{h}_{\gamma_1, \gamma_2'}=
        \big( \Phi_1\big(\xi(\gamma_1)\big) -
        \Phi_2\big(\xi(\gamma_2')\big), \xi(\gamma_1) -
        \xi(\gamma_2')\big)$, then an application of Lemma \ref{lem -
          properties of surface Sigma} implies that
        \begin{align*}
          &\bigg\| \sum_{\substack{ \gamma_1 \in \Gamma_1^*(q) \\ \gamma_2 \in \Gamma_2(q) } } \mc{P}_{\gamma_1} u \mc{P}_{\gamma_2} v      \bigg\|_{L^2_{t, x}(q)}^2\\ \les{}& \sum_{ \substack{\gamma_1 \in \Gamma_1^*(q) \\ \gamma_2' \in \Gamma_2(q) }} \sum_{\gamma_1' \in \Gamma^{**}_1(q, \mathfrak{h}_{\gamma_1, \gamma_2'})} \sum_{\substack{\substack{\gamma_2 \in \Gamma_2(q) \\ \xi(\gamma_2) = \xi(\gamma_1') + \xi(\gamma_2') - \xi(\gamma_1) + \mc{O}(R^{-\frac{1}{2}})}}} \left| \left\lr{ \mc{P}_{\gamma_1} u \mc{P}_{\gamma_2} v, \mc{P}_{\gamma_1'} u \mc{P}_{\gamma_2'} v \right}_{L^2_{t, x}}\right|.
        \end{align*}
        On the other hand, an application of Lemma \ref{lem - key
          bilinear L2 bound I} easily gives the $U^2_\Phi$ bound
        $$ \| \mc{P}_{\gamma_1} u \mc{P}_{\gamma_2} v \|_{L^2_{t, x}} \lesa R^{-\frac{n-1}{4}} \| L^\sharp_{\gamma_1} f_J \|_{\ell^2_J L^2_x}\| L^\sharp_{\gamma_2} g_{J'} \|_{\ell^2_{J'} L^2_x}.$$
        If we now note that, for fixed $\gamma_1$, $\gamma_2'$, and
        $\gamma_1'$, and any $q \in \mb{q}$ we have
	$$ \# \big\{ \gamma_2 \in \Gamma_2 \,\,\big| \,\, T_{\gamma_2} \cap R^\delta q \not = 0, \xi(\gamma_2) = \xi(\gamma_1') + \xi(\gamma_2') - \xi(\gamma_1) + \mc{O}(R^{-\frac{1}{2}})  \big\} \lesa R^{n \delta}$$
        then an application of Cauchy-Schwarz gives
      $$ \bigg\| \sum_{\substack{ \gamma_1 \in \Gamma_1^*(q) \\ \gamma_2 \in \Gamma_2(q) } } \mc{P}_{\gamma_1} u \mc{P}_{\gamma_2} v      \bigg\|_{L^2_{t, x}(q)}^2 \lesa R^{D\delta -\frac{n-1}{2}} \#\Gamma_1^*(q) \#\Gamma_2(q)  \sup_{\mathfrak{h}} \#\Gamma_1^{**}(q, \mathfrak{h}) \sup_{\gamma_j \in \Gamma_j}\| L^\sharp_{\gamma_1} f_J \|_{\ell^2_J L^2_x}^2\| L^\sharp_{\gamma_2} g_{J'} \|_{\ell^2_{J'} L^2_x}^2.$$
      Consequently the bound (\ref{eqn - sec L2 bound - sum over conc
        tubes localised to cubes dyadic}) would follow from the
      combinatorial estimate
      \begin{equation}\label{eqn - sec L2 bound - combinatorial bound
          I} \sum_{\substack{ q \in \mb{q}(\mu_1, \mu_2) \\ q \subset
            2B}} \# \Gamma_1^*(q) \# \Gamma_2(q) \sup_{\mathfrak{h}
          \in \RR^{1+n}} \# \Gamma^{**}_1(q, \mathfrak{h}) \lesa R^{D
          \delta} \# \Gamma_1 \# \Gamma_2. \end{equation}
      We now simplify this bound slightly by exploiting the dyadic localisations we preformed earlier. More precisely, by definition, for every  $ q \in \mb{q}(\mu_1, \mu_2)$, we have $ \# \Gamma_2(q) \les 2 \mu_2$. On the other hand, by exchanging the order of summation, recalling the short hand $ \Gamma^*_1(q) = \Gamma_1^{\not \sim B}(q) \cap \Gamma_1[\lambda_1, \mu_1, \mu_2]$, and using the definition of the set $\Gamma_1[\lambda_1, \mu_1, \mu_2]$, we deduce that
      \begin{align*}
        \sum_{ \substack{q \in \mb{q}(\mu_1, \mu_2) \\ q \subset 2B}} \# \Gamma^*_1(q) &\les \sum_{ q \in \mb{q}(\mu_1, \mu_2)} \# \big( \Gamma_1(q)\cap \Gamma[\lambda_1, \mu_1, \mu_2]\big) \\
        &= \sum_{\gamma_1 \in \Gamma[\lambda_1, \mu_1, \mu_2]} \# \{ q \in \mb{q}(\mu_1, \mu_2) \mid T_{\gamma_1} \cap R^\delta q \not = 0 \} \\
        &\les 2 \lambda_1 \# \Gamma_1
      \end{align*}
      Therefore, we have reduced the bound (\ref{eqn - sec L2 bound -
        combinatorial bound I}) to proving the combinatorial Kakeya
      type estimate
      \begin{equation} \label{eqn - sec L2 bound - combinatorial bound
          II} \sup_{\substack{ \mathfrak{h} \in \RR^{1+n} \\ q \in
            \mathbf{q}(\mu_1, \mu_2), q \subset 2B}}
        \Gamma_1^{**}[\lambda_1, \mu_1, \mu_2](q, \mathfrak{h}) \lesa
        R^{D\delta} \frac{\#\Gamma_2}{\lambda_1 \mu_2}.\end{equation}
      The proof of this bound is the focus of the next subsection.

      \subsection{The $L^2$ Bound: The Combinatorial Estimate}\label{subsect:l2bd-comb}
      We have reduced the proof
      of Theorem \ref{thm - loc bilinear Lp estimate} to obtaining the
      combinatorial bound \eqref{eqn - sec L2 bound - combinatorial
        bound II}, which is essentially well-known to experts as it does not see the difference between homogeneous solutions and $V^2_{\Phi_j}$-functions. For completeness, we include the proof here.
      We follow the ``bush'' argument used in \cite{Tao2003a}, making some minor adjustments only to relate it to Assumption \ref{assump
     on phase}.
      Recall that we have fixed a ball $B \in \mc{B}$. Fix any $\mathfrak{h} \in \RR^{1+n}$ and
      $q_0 \in \mb{q}(\mu_1, \mu_2)$ with $q_0 \subset 2B$. Our goal
      is to prove
	$$ \# \Gamma_1^{**}(q_0, \mathfrak{h}) \lesa R^{D\delta} \frac{\# \Gamma_2}{\lambda_1 \mu_2}.$$
        The first step is to exploit the fact that $\gamma_1$ is not
        concentrated on $B$. Recall from Subsection \ref{subsec - tube
          definitions} that for $\gamma_1 \in \Gamma_1$ we have
        defined the ball $B(\gamma_1, \lambda_1, \mu_1, \mu_2) \in
        \mc{B}$ to be (a) maximiser for the quantity
	$$ \# \{  q \in \mb{q}(\mu_1, \mu_2) \mid T_{\gamma_j} \cap R^\delta q \not =  \varnothing, \,\, q\cap B(\gamma_j, \lambda_j, \mu_1, \mu_2) \not =  \varnothing \}.$$
        Let $\gamma_1 \in \Gamma^{**}_1(q_0, \mathfrak{h})$. By
        construction this implies that $\gamma_1 \in \Gamma_1^{\not
          \sim B}(q_0)$, and hence by definition of the relation
        $\sim$, we have $B \not \subset 10B(\gamma_1, \lambda_1,
        \mu_1, \mu_2)$. Since $q_0 \subset 2B$ and the balls in
        $\mc{B}$ have radius $R^{1-\delta}$, we must have
        $\mathrm{dist}(q_0, B(\gamma_1, \lambda_1, \mu_1, \mu_2))
        \gtrsim R^{1-\delta}$. In particular, by (ii) in Remark
        \ref{rem - consequences of tube balls relation}, we have for
        every $\gamma_1 \in \Gamma_1^{**}(q_0, \mathfrak{h})$
	\begin{align*} \# \{ q \in \mb{q}(\mu_1, \mu_2) &\mid T_{\gamma_1} \cap R^\delta q \not = \varnothing, \mathrm{dist}(q, q_0) \gtrsim R^{1-\delta} \} \\
          &\gtrsim \# \{ q \in \mb{q}(\mu_1, \mu_2) \mid T_{\gamma_1} \cap R^\delta q \not = \varnothing, q\cap B(\gamma_1, \lambda_1, \mu_1, \mu_2) \not = \varnothing \}  \\
          &\gtrsim R^{-D\delta} \lambda_1.
	\end{align*}
        On the other hand, since for $ q \in \mb{q}(\mu_1, \mu_2)$ we
        have $\# \Gamma_2(q) \g \mu_2$, we deduce that
$$ \#\{ (q, \gamma_2) \in \mb{q}(\mu_1, \mu_2) \times \Gamma_2 \mid T_{\gamma_1} \cap R^\delta q \not = \varnothing, T_{\gamma_2} \cap R^\delta q \not = \varnothing,  \mathrm{dist}(q, q_0) \gtrsim R^{1-\delta} \} \gtrsim R^{-D\delta} \lambda_1 \mu_2. $$
Summing up over $\gamma_1 \in \Gamma_1^{**}(q_0, \mathfrak{h})$ and
then changing the order of summation gives
\begin{align*}
  \lambda_1 &\mu_2 \#\Gamma_1^{**}(q_0, \mathfrak{h}) \\
  &\lesa R^{D\delta}\sum_{\gamma_1 \in \Gamma_1^{**}(q_0, h)}\#\{ (q, \gamma_2) \in \mb{q}(\mu_1, \mu_2) \times \Gamma_2 \mid T_{\gamma_1} \cap R^\delta q \not = \varnothing, T_{\gamma_2} \cap R^\delta q \not = \varnothing,  \mathrm{dist}(q, q_0) \gtrsim R^{1-\delta} \} \\
  &= R^{D\delta}\sum_{\gamma_2 \in \Gamma_2}\#\{ (q, \gamma_1) \in
  \mb{q}(\mu_1, \mu_2) \times \Gamma_1^{**}(q_0, h) \mid T_{\gamma_1}
  \cap R^\delta q \not = \varnothing, T_{\gamma_2} \cap R^\delta q
  \not = \varnothing, \mathrm{dist}(q, q_0) \gtrsim R^{1-\delta} \}.
\end{align*}
Therefore the required bound (\ref{eqn - sec L2 bound - combinatorial
  bound II}) follows from the following lemma, cf. \cite[Lemma
8.1]{Tao2003a}.

\begin{lemma}\label{lem - conclusion of combinatorial argument}
  Let $q_0 \in \mb{q}$, $\mathfrak{h} \in \RR^{1+n}$, and $\gamma_2
  \in \Gamma_2$. Then
  $$ \#\{ (q, \gamma_1) \in \mb{q}(\mu_1, \mu_2) \times \Gamma_1^{**}(q_0, \mathfrak{h}) \mid T_{\gamma_1} \cap R^\delta q \not = \varnothing, T_{\gamma_2} \cap R^\delta q \not = \varnothing,  \mathrm{dist}(q, q_0) \gtrsim R^{1-\delta} \}\lesa R^{D\delta}. $$
  \begin{proof}Define the bush (or ``fan'') at $q_0$ by
		$$ \text{Bush}(q_0) =  \bigcup_{\gamma_1 \in \Gamma^{**}_1(q_0, \mathfrak{h})} T_{\gamma_1}. $$
                Thus $\text{Bush}(q_0) \subset \RR^{1+n}$ is the union
                of all tubes $T_{\gamma_1}$ (associated to phase space
                elements $\gamma_1 \in \Gamma_1^{**}(q_0,
                \mathfrak{h})$) passing through a neighbourhood of the
                cube $q_0$. Our goal is then to bound the sum
		\begin{equation}\label{eqn - lem conclusion of
                    combinatorial argument - sum over cubes}
                  \sum_{\substack{ q \in \mb{q}(\mu_1, \mu_2),\\ q
                      \subset \text{Bush}(q_0) \cap T_{\gamma_2} +
                      \mc{O}(R^{\frac{1}{2} + \delta}) \\
                      \mathrm{dist}(q, q_0) \gtrsim R^{1-\delta}}} \#\{
                  \gamma_1 \in \Gamma_1^{**}(q_0, \mathfrak{h}) \mid
                  T_{\gamma_1} \cap R^\delta q \not = \varnothing
                  \}.  \end{equation}
                We first count the number of possible cubes in the outer summation. The idea is to first show that
                \begin{equation}\label{eqn - lem conclusion of combinatorial argument - bush contained in conic}
                  \text{Bush}(q_0) \subset (t_0, x_0) + \mc{C}_1(\mathfrak{h}) + \mc{O}(R^{\frac{1}{2} + D\delta})
                \end{equation}
                where $(t_0, x_0)$ denotes the centre of the cube
                $q_0$, and the conic hypersurface
                $\mc{C}_1(\mathfrak{h})$ is given by
		$$ \mc{C}_1(\mathfrak{h}) = \big\{ \big(r, -r \nabla \Phi_1(\xi) \big) \, \big| \, r \in \RR, \xi \in \Sigma_1(\mathfrak{h}) \big\}.$$
                Since if we had (\ref{eqn - lem conclusion of
                  combinatorial argument - bush contained in conic}),
                an application of Lemma \ref{lem - surface C
                  transverse} would then show that $\text{Bush}(q_0)
                \cap T_{\gamma_2}$ is contained in a ball of radius
                $R^{\frac{1}{2} +D\delta}$, and hence the outer
                summation in (\ref{eqn - lem conclusion of
                  combinatorial argument - sum over cubes}) only
                contains $\mc{O}(R^{D\delta})$ terms. To show the
                inclusion (\ref{eqn - lem conclusion of combinatorial
                  argument - bush contained in conic}), suppose that
                $(t, x) \in \text{Bush}(q_0)$. Then $(t, x) \in
                T_{\gamma_1}$ for some $\gamma_1 \in
                \Gamma^{**}_1(q_0, \mathfrak{h})$. By construction, we
                have $\xi(\gamma) = \xi^* + \mc{O}(R^{-\frac{1}{2}})$
                for some $\xi^* \in \Sigma_1(\mathfrak{h})$. On the
                other hand, since $T_{\gamma_1} \cap R^\delta q_0 \not
                = 0$, we have
	$$ x-x_0 + (t - t_0) \nabla \Phi_1\big(\xi(\gamma_1)\big) = \big[x-x(\gamma) + t \nabla \Phi_1\big(\xi(\gamma_1)\big)\big] - \big[ x_0 - x(\gamma) + t_0 \nabla \Phi_1\big(\xi(\gamma_1)\big) \big] =  \mc{O}(R^{\frac{1}{2} + \delta}).$$
        Therefore, since $|t-t_0| \lesa R$, we can write
	\begin{align*} & (t, x)- (t_0, x_0) \\
          ={}& \big( t - t_0, - (t- t_0) \nabla \Phi_1(\xi^*) \big) + \big( 0, x-x_0 + (t-t_0) \nabla \Phi_1(\xi(\gamma_1)) \big) + \big(0, (t-t_0) [ \nabla \Phi_1(\xi^*) - \nabla \Phi_1(\gamma(\xi))]\big)\\
          ={}& \big( t - t_0, - (t- t_0) \nabla \Phi_1(\xi^*) \big) +
          \mc{O}(R^{\frac{1}{2} + \delta})
	\end{align*}
        and hence we have \eqref{eqn - lem conclusion of combinatorial
          argument - bush contained in conic}. Consequently, the
        outer sum in \eqref{eqn - lem conclusion of combinatorial
          argument - sum over cubes} is only over $\mc{O}(R^{C
          \delta})$ cubes.

        Fix $q \in \mb{q}(\mu_1, \mu_2)$ with $\mathrm{dist}(q, q_0)
        \gtrsim R^{1- \delta}$. As the outer sum in (\ref{eqn - lem
          conclusion of combinatorial argument - sum over cubes}) only
        adds $\mc{O}(R^{D\delta})$, the required bound would now
        follow from
        \begin{equation}\label{eqn - lem conclusion of combinatorial
            argument - counting number of tubes}
          \# \big\{ \gamma_1 \in \Gamma_1 \,\,\big| \,\, \xi(\gamma_1) \in \Sigma_1(\mathfrak{h}) + \mc{O}(R^{-\frac{1}{2}}), \,\,T_{\gamma_1} \cap R^\delta q \not = \varnothing,\,\, T_{\gamma_1} \cap R^\delta q_0 \not = \varnothing \big\} \lesa R^\delta.
        \end{equation}
        The point is that since the cubes $q$ and $q_0$ are at a
        distance $R^{1-\delta}$ apart, the condition that
        $T_{\gamma_1}$ must intersect \emph{both} cubes, essentially
        fixes the tube $T_{\gamma_1}$. Since $\xi(\gamma_1) \in
        \Sigma_1(\mathfrak{h}) + \mc{O}(R^{-\frac{1}{2}})$, the bound
        (\ref{eqn - seperation of velocities in Sigma}) implies that
        fixing the tube $T_{\gamma_1}$, also more or less fixes the
        phase space element $\gamma_1$ (note that without the bound
        (\ref{eqn - seperation of velocities in Sigma}), the set in
        (\ref{eqn - lem conclusion of combinatorial argument -
          counting number of tubes}) could potentially contain far
        more than $\mc{O}(R^\delta)$ points). In more detail, let
    $$ \gamma_1, \gamma'_1 \in \big\{ \gamma_1 \in \Gamma_1 \,\,\big| \,\, \xi(\gamma_1) \in \Sigma_1(\mathfrak{h}) + \mc{O}(R^{-\frac{1}{2}}), \,\,T_{\gamma_1} \cap R^\delta q \not = \varnothing,\,\, T_{\gamma_1} \cap R^\delta q_0 \not = \varnothing \big\}.$$
    In light of (\ref{eqn - seperation of velocities in Sigma}), the
    estimate (\ref{eqn - lem conclusion of combinatorial argument -
      counting number of tubes}) would follow from the bounds
    \begin{equation}\label{eqn - lem conclusion of combinatorial
        argument - initial position and velocity} | x(\gamma_1) -
      x(\gamma_1') | \lesa R^{\frac{1}{2} + \delta}, \qquad |
      v(\gamma_1) - v(\gamma_1') | \lesa R^{-\frac{1}{2} +
        \delta} \end{equation}
    where ease of notation we define the \emph{velocity} $v(\gamma_1) = \Phi_1\big( \xi(\gamma_1) \big)$. We now exploit the condition that the tubes $T_{\gamma_1}$ and $T_{\gamma_1'}$ intersect the cubes $q$ and $q_0$. Let  $(t_q, x_q)$ denote the centre of the cube $q$ and $(t_0, x_0)$ the centre of $q_0$. Since $|v(\gamma_1)| \les \mb{D}_2$ and
 	$$ x_0 - x_q + (t_0 - t_q) v(\gamma_1) = \big( x_0 - x(\gamma_1) + t_0 v(\gamma_1) \big) - \big( x_q - x(\gamma_1) + t_q v(\gamma_1) \big) = \mc{O}(R^{\frac{1}{2} + D\delta}),$$
        the separation of the cubes $q$ and $q_0$ implies that $R^{1-C
          \delta} \lesa |t_0 - t_q| \lesa R$. A computation shows that
	$$ (t_0 - t_q) \big(v(\gamma_1) - v(\gamma_1') \big) = \mc{O}(R^{\frac{1}{2} + D\delta}), \qquad \qquad x(\gamma_1) - x(\gamma_1') = t_0 \big( v(\gamma_1') - v(\gamma_1) \big) +  \mc{O}(R^{\frac{1}{2} + D\delta})$$
        and hence the bound on $|t_0 - t_q|$ gives (\ref{eqn - lem
          conclusion of combinatorial argument - initial position and
          velocity}).
      \end{proof}
    \end{lemma}

\section{The Globalisation Lemma}\label{sect:glob}

In this section, we complete the proof of Theorem \ref{thm - bilinear Lp estimate} by showing that it follows from the localised bound in Theorem \ref{thm - loc bilinear Lp estimate}. The proof of Theorem \ref{thm - bilinear Lp estimate} proceeds by using a strategy sketched in Section 8 of \cite{Klainerman2002b}, together with interpolation argument to replace $U^2_{\Phi_j}$ with $V^2_{\Phi_j}$.

\begin{proof}[Proof of Theorem \ref{thm - bilinear Lp estimate}]
   The first step is to show that by, exploiting the (approximate) finite speed of propagation of frequency localised waves, the bilinear estimate on $Q_R$ implies the same estimate holds on $I_R \times \RR^n$ with $I_R = [0, R]$. The second step is to remove the remaining temporal localisation and $R^\alpha$ factor by using  duality, together with the dispersive decay in Lemma \ref{lem - dispersion}. Finally we use a simple interpolation argument to replace $U^2_{\Phi_j}$ with the larger $V^2_{\Phi_j}$ space.

\medskip

{\bf Step 1: From $Q_R$ to $I_R \times \RR^n$.}  Let $R \g (10 \mb{R}_0)^2$,  $u \in U^2_{\Phi_j}$, and $v \in U^2_{\Phi_j}$ with $\supp \widehat{u} \subset \Lambda^*_1$ and $\supp \widehat{v} \subset \Lambda^*_2$. Assuming Theorem \ref{thm - loc bilinear Lp estimate}, our goal is to prove that for every $\alpha >0$ we have
    \begin{equation}\label{eqn - thm bilinear Lp estimate - temporal localisation}
                \| u v \|_{L^{\frac{n+3}{n+1} }_{t, x}(I_R \times \RR^n)} \lesa  R^\alpha \| u \|_{U^2_{\Phi_j}} \| v \|_{U^2_{\Phi_j}}.
    \end{equation}
  It is enough to consider the case where $u$, and $v$ are atoms, thus we have a decomposition
        $$ u = \sum_J \ind_J(t) e^{ it \Phi_1(-i\nabla)} f_J, \qquad v = \sum_{J'} \ind_{J'}(t) e^{ i t \Phi_2(-i\nabla) } g_{J'}$$
  with
        $$ \sum_J \| f_J \|_{L^2}^2  + \sum_{J'} \| g_{J'} \|_{L^2}^2 \les 1$$
  and we may assume that $\supp \widehat{f}_J \subset \Lambda_1^*$ and $\supp \widehat{g}_{J'} \subset \Lambda_2^*$ (using sharp Fourier cutoffs). By translation invariance, the bound (\ref{eqn - thm bilinear Lp estimate - temporal localisation}) would then follow from
    \begin{equation}\label{eqn - thm bilinear Lp estimate - temp localisation with finite speed of propagation}
                \| u v \|_{L^{\frac{n+3}{n+1}}_{t, x}(Q_R)}  \lesa  R^\alpha \bigg( \sum_J \big\| ( 1 + R^{-1} |x|)^{-(n+1)} f_J \big\|_{L^2_x}^2 \bigg)^\frac{1}{2} \bigg( \sum_{J'} \big\| ( 1 + R^{-1} |x|)^{-(n+1)} g_{J'} \big\|_{L^2_x}^2 \bigg)^\frac{1}{2}
    \end{equation}
 since we can then sum up over the centres of balls (or cubes) of radius $R$ which cover $\RR^n$. The inequality (\ref{eqn - thm bilinear Lp estimate - temp localisation with finite speed of propagation}) is a reflection of the fact that, as $u$ and $v$ are localised to frequencies of size $\approx 1$, we expect that the waves $e^{it \Phi_j(-i\nabla)} f_J$ should travel with velocity $1$. In particular, $u$ and $v$ on $Q_R$, should only depend on the data in $\{ |x| \lesa R\}$. It turns out that this is true, modulo a rapidly decreasing tail.

Let $\rho \in \s$ with $\supp \widehat{\rho} \subset \{ |\xi| \les 1\}$ and $\rho \gtrsim 1$ on $|x| \les 1$. To prove (\ref{eqn - thm bilinear Lp estimate - temp localisation with finite speed of propagation}), we start by noting that since the left hand integral is only over $Q_R$, we may replace $uv$ with $\rho(R^{-1} x) u(t, x) \rho(R^{-1} x) v(y)$. Note that we can write
    \begin{align}
          \rho\Big( \frac{x}{R}\Big) \Big(e^{ i t \Phi_j(-i\nabla)} f\Big)(x) &= \int_{\RR^n} \int_{\RR^n} R^n \,\widehat{\rho}\big(\, R(\xi - \eta) \big) \, e^{ it \Phi_j(\eta)} \widehat{f}(\eta)\, d \eta \, e^{ i \xi \cdot x} d\xi \notag \\
                        &= \int_{\RR^n} \int_{\RR^n} R^n \,\widehat{\rho}\big( R(\xi - \eta) \big) \,  \widehat{f}(\eta) F\big(t, R(\xi - \eta), \eta\big)\, d \eta \, e^{ i \xi \cdot x} e^{i  \Phi_j(\xi)} d\xi
                        \label{eqn - thm bilinear Lp estimate - commuting cutoff and exp}
        \end{align}
where $F(t, \xi, \eta) = \chi(\xi, \eta) e^{ i t \big( \Phi_j (\frac{\xi}{R} + \eta) - \Phi_j(\eta)\big) }$ and $\chi \in C^\infty_0\big( \{|\xi|\les 2\} \times (\Lambda_j^* + \frac{1}{\mb{R}_0})\big)$ with $\chi =1$ on $\{|\xi|\les 2\} \times \Lambda_j^*$. The oscillating component of $F$ is essentially constant for $|t| \les R$. To exploit this, we expand $F$ using a Fourier series to get
        $$ F(t, \xi, \eta) = \sum_{k \in \ZZ^{2n}} c_k(t) e^{ i k \cdot (\xi, \eta)}, \qquad \qquad c_k(t) = \int_{\RR^{2n}} F(t, \xi, \eta) e^{ i k \cdot (\xi, \eta)} \,d\xi\,d\eta$$
and by \eqref{it:ass2} in Assumption \ref{assump on phase}, the coefficients satisfy $|c_k(t)| \lesa_{\mb{R}_0, \mb{D}_2} ( 1 + |k_1|)^{-2(n+1)} (1 + |k_2|)^{-2(n+1)}$ with $k=(k_1, k_2)$. Applying this expansion to $\rho(R^{-1} x) u$ and $\rho(R^{-1} x) v$ we obtain the decompositions
    \begin{equation}\label{eqn - thm bilinear Lp estimate - decomp} \rho\big( R^{-1}x\big) u = \sum_J \sum_k c_k(t) \ind_J(t) e^{ it \Phi_1(-i\nabla)} f_{k, J}, \qquad \rho\big( R^{-1}x\big) v = \sum_{J'} \sum_k c'_k(t) \ind_{J'}(t) e^{ it \Phi_2(-i\nabla)} g_{k, J'} \end{equation}
where the coefficients $c_k$, $c_k'$ are independent of $J$ and $J'$, and the functions $f_{k, J}$ and $g_{k, J'}$ are given by
        $$f_{k, J}(x) = \rho\Big( \frac{x}{R} + k_1\Big) f_J(x + k_2), \qquad g_{k, J'}(x) = \rho\Big( \frac{x}{R} + k_1 \Big) g_{J'}( x + k_2)$$
  with $k=(k_1, k_2)$. Note that $\supp \widehat{f}_{k, J} \subset \Lambda^*_1 + \frac{1}{ 2\mb{R}_0}$ since $R \g (10 \mb{R}_0)^2$, thus the $f_{k, J}$ satisfy the support conditions in Theorem \ref{thm - loc bilinear Lp estimate} (with $\Lambda_j^*$ replaced with $\Lambda^*_j+\frac{1}{\mb{R}_0}$, and $\mb{R}_0$ replaced with $2\mb{R}_0$). A similar comment applies to the $g_{k', J}$. Therefore, plugging the decomposition (\ref{eqn - thm bilinear Lp estimate - decomp}) into the left hand side of (\ref{eqn - thm bilinear Lp estimate - temp localisation with finite speed of propagation}), we deduce via an application of Theorem \ref{thm - loc bilinear Lp estimate} that
    \begin{align*}
    \|& u v \|_{L^\frac{n+3}{n+1}_{t, x}(Q_R)} \\
    &\lesa \sum_{k, k'\in \ZZ^n \times \ZZ^n} ( 1 + |k|)^{-2(n+1)} ( 1 + |k'|)^{-2(n+1)}\Big\| \sum_{J, J'} \ind_J(t) e^{ i t \Phi_1(-i\nabla)} f_{k, J}\, \ind_{J'}(t) e^{ i t \Phi_2(-i\nabla)} g_{k', J'} \Big\|_{L^\frac{n+3}{n+1}_{t, x}(Q_R)} \\
            &\lesa R^\alpha \sum_{k, k'} ( 1 + |k|)^{-2(n+1)} ( 1 + |k'|)^{-2(n+1)} \\
            & \qquad \times \Big( \sum_J \big\| ( 1 + R^{-1} | x - k_1 + R k_2|)^{-(n+1)} f_{J} \big\|_{L^2_x}^2 \Big)^\frac{1}{2} \Big( \sum_{J'} \big\| ( 1 + R^{-1}|x - k'_1 + R k'_2|)^{-(n+1)} g_{J'} \big\|_{L^2_x}^2 \Big)^\frac{1}{2} \\
            &\lesa R^\alpha \Big( \sum_J \big\| ( 1 + R^{-1} |x|)^{-(n+1)} f_{J} \big\|_{L^2_x}^2 \Big)^\frac{1}{2} \Big( \sum_{J'} \big\| ( 1 + R^{-1} |x|)^{-(n+1)} g_{J'} \big\|_{L^2_x}^2 \Big)^\frac{1}{2}
  \end{align*}
Thus we obtain \eqref{eqn - thm bilinear Lp estimate - temp localisation with finite speed of propagation} and hence \eqref{eqn - thm bilinear Lp estimate - temporal localisation}.

\medskip

{\bf Step 2: From $I_R \times \RR^n$ to $\RR^{1+n}$.} Let $u \in U^2_{\Phi_1}$ and $v \in U^2_{\Phi_2}$ with $\supp \widehat{u} \subset \Lambda^*_1$ and $\supp \widehat{v} \subset \Lambda^*_2$. Our goal is to show that for every $p>\frac{n+3}{n+1}$
        \begin{equation}\label{eqn - thm bilinear Lp estimate - global U2 bound}
                \| u v \|_{L^p_{t, x}} \lesa \| u \|_{U^2_{\Phi_1}} \| v \|_{U^2_{\Phi_2}}.
        \end{equation}
In fact the argument below gives the marginally stronger (though essentially equivalent) bound
		 \begin{equation}\label{eqn - thm bilinear Lp estimate - global U2 bound x endpoint}
                \| u v \|_{L^p_t L^\frac{n+3}{n+1}_x} \lesa \| u \|_{U^2_{\Phi_1}} \| v \|_{U^2_{\Phi_2}}.
        \end{equation}
To deduce (\ref{eqn - thm bilinear Lp estimate - global U2 bound}) from (\ref{eqn - thm bilinear Lp estimate - global U2 bound x endpoint}), note that dispersive estimate in Lemma \ref{lem - dispersion}, together with the abstract Strichartz estimates of Keel-Tao \cite[Theorem 1.2]{Keel1998}, implies there exists $1<a<b<\infty$ such that $\| u v \|_{L^a_t L^b_x} \lesa 1$. On the other hand, the Fourier support assumptions imply that we have the trivial bound  $ \| u v \|_{L^\infty_t L^p_x(\RR^{1+n})} \lesa 1$ for every $p \g 1$. Thus interpolation gives (\ref{eqn - thm bilinear Lp estimate - global U2 bound}) from (\ref{eqn - thm bilinear Lp estimate - global U2 bound x endpoint}).

We now turn to the proof of (\ref{eqn - thm bilinear Lp estimate - global U2 bound x endpoint}). As in step 1, we may assume that $u$ and $v$ are atoms with the decomposition
     $$ u = \sum_J \ind_J(t) e^{ it \Phi_1(-i\nabla)} f_J, \qquad v = \sum_{J'} \ind_{J'}(t) e^{ i t \Phi_2(-i\nabla) } g_{J'}$$
  with $\supp \widehat{f}_J \subset \Lambda_1^*$, $\supp \widehat{g}_{J'} \subset \Lambda_2^*$, and
        $$ \sum_J \| f_J \|_{L^2}^2  + \sum_{J'} \| g_{J'} \|_{L^2}^2 \les 1.$$
By real interpolation it is enough to show that for every $q>\frac{n+3}{n+1}$ we have
    $$ \| u v \|_{L^{q, \infty}_t L^{\frac{n+3}{n+1}}_x} \lesa 1$$
where $L^{q, \infty}_t$ is the Lorentz norm. Applying duality, this would follow from the estimate
            \begin{equation}\label{eqn - thm bilinear Lp estimate - global lorenz bound}  \int_{\Omega} \| u v \|_{L^{\frac{n+3}{n+1}}_x} dt \lesa |\Omega|^{\frac{1}{q'}} \end{equation}
for every measurable $\Omega \subset \RR$. Define the Fourier localised solution operator $\mc{U}_j(t)[h] = e^{ i t \Phi_j(-i\nabla)} P_{\Lambda^*_j} h$  where we let $\widehat{P_{\Lambda_j^*} h}(\xi)  = \rho_{\Lambda_j^*}(\xi) \widehat{h}(\xi)$ with $\rho \in C^\infty_0(\Lambda^*_j + \frac{1}{10 \mb{R}_0})$ and $\rho = 1$ on $\Lambda_j^*$. If we interpolate Lemma \ref{lem - dispersion} with the trivial $L^\infty_t L^2_x$ bound and apply duality, we deduce that for every $1\les a \les 2$
    \begin{align}\label{eqn - thm bilinear Lp estimate - dispersive estimate}
     \int_{\substack{(t, t') \in \Omega \times \Omega\\ |t-t'| \gtrsim R} } \Big\langle \mc{U}_j^*(t)[G(t)], \, \mc{U}_j^*(t')[G(t')] \Big\rangle_{L^2_x} \,dt\,dt'
           \lesa |\Omega|^2 R^{ - \frac{n-1}{2} ( \frac{2}{a}- 1)} \| G \|_{L^\infty_t L^a_x}^2
    \end{align}
where $\mc{U}_j^*$ denotes the $L^2_x$ adjoint of $\mc{U}_j$. The dispersive bound (\ref{eqn - thm bilinear Lp estimate - dispersive estimate}) together with the bilinear estimate (\ref{eqn - thm bilinear Lp estimate - temporal localisation}) are the key inequalities required in the proof of (\ref{eqn - thm bilinear Lp estimate - global lorenz bound}).

We now begin the proof of (\ref{eqn - thm bilinear Lp estimate - global lorenz bound}). If $|\Omega| \lesa 1$, then (\ref{eqn - thm bilinear Lp estimate - global lorenz bound}) follows by putting $uv \in L^\infty_t L^{\frac{n+3}{n+1}}_x$ and using Sobolev embedding. Thus we may assume that $|\Omega| \gg 1$. Let us set $J'_\Omega:=\Omega\cap J'$. An application of duality gives
    \begin{align*}
      \int_{\Omega} \| u v \|_{L^{\frac{n+3}{n+1}}_x} dt  &\les \sup_{\| F \|_{L^\infty_t L^{\frac{n+3}{2}}_x} \les 1} \bigg| \int_{\Omega} \lr{F, uv}_{L^2_x} \,dt\bigg| \\
      &= \sup_{\| F \|_{L^\infty_t L^{\frac{n+3}{2}}_x} \les 1} \bigg| \sum_{J'} \int_{J'_\Omega} \lr{ F, u \,\mc{U}_2(t)[g_{J'}]}_{L^2_x} \,dt\bigg|\\
      &\lesa \sup_{\| F \|_{L^\infty_t L^{\frac{n+3}{2}}_x} \les 1} \Bigg( \sum_{J'} \Big\| \int_{J'_\Omega} \mc{U}^*_2(t)[ F \overline{u}]\,dt \Big\|_{L^2_x}^2 \Bigg)^\frac{1}{2}
    \end{align*}
If we expand the square of the $L^2_x$ norm, we have via (\ref{eqn - thm bilinear Lp estimate - dispersive estimate}) with $\frac{1}{a} = \frac{2}{n+3} + \frac{1}{2}$
    \begin{align*}
      \sum_{J'} \Big\| \int_{J'_\Omega} \mc{U}^*_2(t)[ F \overline{u}]\,dt \Big\|_{L^2_x}^2  &= \sum_{J'} \int_{t, t' \in J'_\Omega}  \big\langle \mc{U}^*_2(t)[F\overline{u}], \, \mc{U}^*_2(t')[F\overline{u}] \, \big\rangle_{L^2_x} \,dt\,dt' \\
      &= \sum_{J'} \int_{\substack{t, t' \in J'_\Omega \\ |t-t'| \gtrsim R}} \big\langle \mc{U}^*_2(t)[F\overline{u}], \, \mc{U}_2^*(t')[F\overline{u}] \, \big\rangle_{L^2_x} \,dt\,dt' \\
       &\qquad + \sum_{J'} \sum_{|I-I'| \les R} \int_{J'_\Omega\cap I}\int_{  J'_\Omega \cap I'} \big\langle \mc{U}^*_2(t)[F\overline{u}], \, \mc{U}^*_2(t')[F\overline{u}] \, \big\rangle_{L^2_x} \,dt\,dt'\\
      &\lesa |\Omega|^2 R^{ - \frac{n-1}{2} ( \frac{2}{a} - 1)} \| F \overline{u} \|_{L^\infty_t L^a_x}^2  + \sum_{J', I} \Big\| \int_{J'_\Omega \cap I} \mc{U}^*_2(t) [F \overline{u}] \, dt \Big\|_{L^2_x}^2\\
      &\lesa |\Omega|^2 R^{ - \frac{2(n-1)}{n+3}} \| F\|_{L^\infty_t L^{\frac{n+3}{2}}_x}^2 \|  u \|_{L^\infty_t L^2_x}^2  + \sum_{J', I} \Big\| \int_{J'_\Omega \cap I} \mc{U}^*_2(t) [F \overline{u}] \, dt \Big\|_{L^2_x}^2
    \end{align*}
here we always take $I$ (and $I'$) to be a decomposition of $\RR$ into intervals of size $R$. We now essentially repeat the previous argument, but expand $u$ instead of $v$ to obtain
    \begin{align*}
      \sum_{J', I} \Big\| \int_{J'_\Omega \cap I} \mc{U}^*_2(t) [F \overline{u}] \, dt \Big\|_{L^2_x}^2 &\les \sup_{\sum_{J', I} \| g_{J', I} \|_{L^2_x}^2 \les 1 } \bigg| \sum_{ J', I} \int_{J'_\Omega \cap I} \big\langle F , \overline{u}\, \mc{U}_2(t) g_{J', I} \rangle_{L^2_x} dt \bigg|^2 \\
      &\lesa \sup_{\sum_{J', I} \| g_{J', I} \|_{L^2_x}^2 \les 1 } \bigg| \sum_{J, I} \int_{J_\Omega \cap I} \big\langle \mc{U}_1^*(t)[F \overline{v}_I] , f_J \rangle_{L^2_x}dt \bigg|^2  \\
      &\lesa \sup_{\sum_{J', I} \| g_{J', I} \|_{L^2_x}^2 \les 1 }  \sum_J  \Big\| \sum_I \int_{J_\Omega \cap I} \mc{U}^*_1(t)[F \overline{v}_I]  dt\Big\|_{L^2_x}^2
    \end{align*}
where we take $ v_I = \sum_{J'} \ind_{J'}(t)  \mc{U}_2(t) g_{J', I}$. Again expanding out the $L^2_x$ norm, and applying (\ref{eqn - thm bilinear Lp estimate - dispersive estimate}), we have
 \begin{align*}
      \sum_{J} \Big\| \sum_I \int_{J_\Omega \cap I} \mc{U}^*_1(t)[ F \overline{v}_I]\,dt \Big\|_{L^2_x}^2
      &= \sum_{J} \sum_{|I-I'| \gg R} \int_{J_\Omega \cap I} \int_{J_\Omega \cap I'} \big\langle \mc{U}^*_1(t)[F \overline{v}_I], \, \mc{U}_1(t')[F \overline{v}_{I'}] \, \big\rangle_{L^2_x} \,dt\,dt' \\
       &\qquad + \sum_{J} \sum_{|I-I'| \lesa R} \int_{J_\Omega\cap I} \int_{J_\Omega\cap I'} \big\langle \mc{U}^*_1(t)[F \overline{v}_I], \, \mc{U}^*_1(t')[F \overline{v}_{I'}] \, \big\rangle_{L^2_x}  \,dt\,dt' \\
      &\lesa |\Omega|^2 R^{- \frac{2(n-1)}{n+3}} \| F \|_{L^\infty_t L^{\frac{n+3}{2}}_x}^2 \sup_I \|v_I \|_{L^\infty_t L^2_x}^2  + \sum_{J, I} \Big\| \int_{J_\Omega \cap I} \mc{U}_1(t) [F v_I] \, dt \Big\|_{L^2_x}^2.
    \end{align*}
 Collection the above chain of estimates together, and using the fact that
    $$\| v_I \|_{L^\infty_t L^2_x}^2 \les \sum_{I, J'} \| g_{J', I} \|_{L^2_x}^2 \les 1$$
 together with another application of duality,  we see that
    \begin{align*}
      \int_{\Omega} \| u v \|_{L^{\frac{n+3}{n+1}}_x} dt
      &\lesa |\Omega|  R^{ - \frac{n-1}{n+3}}
       + \sup_{\substack{ \| F \|_{L^\infty_t L^{\frac{n+3}{2}}_x} \les 1 \\
       \sum_{I, J'} \| g_{I,J'} \|_{L^2_x}^2 \les 1}}  \bigg(\sum_{J, I} \Big\| \int_{J_\Omega \cap I} \mc{U}_1(t) [F \overline{v}_I] \, dt \Big\|_{L^2_x}^2 \bigg)^\frac{1}{2}\\
       &\les   |\Omega|  R^{ - \frac{n-1}{n+3}}
       + \sup_{\sum_{I, J'} \| g_{I,J'} \|_{L^2_x}^2, \sum_{I, J} \| f_{I,J} \|_{L^2_x}^2 \les 1} \sum_I \int_{\Omega \cap I} \| u_I v_I \|_{L^{\frac{n+3}{n+2}}_x} \, dt
    \end{align*}
 where we define $ u_I = \sum_{I, J} \ind_J(t) \mc{U}_1(t)[f_{I, J}]$.  Observe that $ \sum_I \| u_I \|_{U^2_{\Phi_1}}^2 \les \sum_{I, J} \| f_{I, J} \|_{L^2_x}^2 \les 1$, and that $u_I$ satisfies the support properties in Theorem \ref{thm - loc bilinear Lp estimate} (with $\Lambda^*_j$ replaced by $\Lambda^*_j + \frac{1}{10\mb{R}_0}$, and $\mb{R}_0$ replaced by $2 \mb{R}_0$). A similar comment applies to $v_I$. Consequently, an application of (\ref{eqn - thm bilinear Lp estimate - temporal localisation}) gives for any $\alpha >0$
   \begin{align*} \sum_I \int_{\Omega \cap I} \| u_I v_I \|_{L^{\frac{n+3}{n+1}}_x} \, dt  &\les |\Omega|^{\frac{2}{n+3}} \sum_I \| u_I v_I \|_{L^{\frac{n+3}{n+1}}_{t, x}(I\times \RR^n)} \\
   &\lesa |\Omega|^{\frac{2}{n+3}} R^\alpha \bigg(\sum_{I, J} \| f_{I, J}\|_{L^2_x}^2 \bigg)^\frac{1}{2} \bigg( \sum_{I, J'} \| g_{I, J'} \|_{L^2_x}^2 \bigg)^\frac{1}{2} \les |\Omega|^{\frac{2}{n+3}} R^\alpha \end{align*}
 and therefore
    $$ \int_{\Omega} \| u v \|_{L^{\frac{n+3}{n+1}}_x} dt  \lesa |\Omega| R^{- \frac{n-1}{n+3}} + |\Omega|^{\frac{2}{n+3}} R^\alpha.$$
 To complete the proof, we choose $R=|\Omega|^{C}$ with $C>0$ sufficiently large so that $|\Omega| R^{- \frac{n-1}{n+3}} \les |\Omega|^{\frac{1}{q'}}$. On the other hand, since $q>\frac{n+3}{n+1}$, we can take $\alpha = \frac{1}{2C}(\frac{n+1}{n+3} - \frac{1}{q})$ which implies that $|\Omega|^{\frac{2}{n+3}} R^\alpha = |\Omega|^{\frac{2}{n+3} + \alpha C} \les |\Omega|^{\frac{1}{q'}}$. Therefore we obtain (\ref{eqn - thm bilinear Lp estimate - global lorenz bound}) as required.

\medskip

{\bf Step 3: From $U^2_\Phi$ to $V^2_\Phi$.} Let $p>\frac{n+3}{n+1}$, $u \in V^2_{\Phi_1}$, $v \in V^2_{\Phi_2}$, and $\supp \widehat{u} \subset \Lambda^*_1$ and $\supp \widehat{v} \subset \Lambda^*_2$. An application of \cite[Lemma 6.4]{Koch2005}, see also \cite[Proposition 2.5 and Proposition 2.20]{Hadac2009}, gives a decomposition $u = \sum_{k \in \NN} u_k$ and $v = \sum_{k \in \NN} v_k$ such that $u_k$, $v_k$ retain the correct Fourier support properties (we can just use sharp Fourier cutoffs here) and for any $ r \g 2$  we have the bounds
    $$ \| u_k \|_{U^{r}_{\Phi_1}} \lesa 2^{ k ( \frac{2}{r} - 1)} \| u \|_{V^2_{\Phi_1}}, \qquad \| v_k \|_{U^{r}_{\Phi_2}} \lesa 2^{ k ( \frac{2}{r} -1 )} \| v \|_{V^2_{\Phi_2}}.$$
Let $\frac{n+3}{n+1}<q<p$, and take $\theta = \frac{ q}{p}<1$. Then an application of (\ref{eqn - thm bilinear Lp estimate - global U2 bound}) (with $p=q$) together with the convexity of $L^p$ norms, gives
\begin{align*}
  \big\| u v \big\|_{L^p_{t, x}} \les \sum_{ k, k'} \big\| u_k v_{k'} \big\|_{L^p_{t, x}} &\les \sum_{k, k'} \big\| u_k v_{k'} \big\|_{L^q_{t, x}}^{\theta} \big\| u_k v_{k'} \big\|_{L^\infty_{t, x}}^{1-\theta} \\
                                &\les \sum_{k', k} \Big( \| u_k \|_{U^{2}_{\Phi_1}} \| v_{k'} \|_{U^{2}_{\Phi_2}} \Big)^\theta \Big( \| u_k \|_{U^{\infty}_{\Phi_1}} \| v_{k'} \|_{U^{\infty}_{\Phi_2}} \Big)^{1-\theta} \\
                                &\lesa \| u \|_{V^2_{\Phi_1}} \| v \|_{V^2_{\Phi_2}} \sum_{k, k'} 2^{ -k (1-\theta)} 2^{ - k' (1-\theta)}\\
                                &\lesa \| u \|_{V^2_{\Phi_1}} \| v \|_{V^2_{\Phi_2}}
\end{align*}
where we used Sobolev embedding and the fact that the Fourier support of $u, v$ is contain in the unit ball to control the $L^\infty_{t, x}$ norm. Thus Theorem \ref{thm - bilinear Lp estimate} follows.
\end{proof}
\begin{remark}\label{rem:proof-endpoint}
The argument in Step 3 above, using \eqref{eqn - thm bilinear Lp estimate - global U2 bound x endpoint}, also implies the slightly stronger estimate
$$  \|u v \|_{L^p_{t}L_x^{\frac{n+3}{n+1}}(\RR^{1+n})} \les C \| u \|_{V^2_{\Phi_1}} \| v \|_{V^2_{\Phi_2}},$$
This is well known in the case of homogeneous solutions,
see e.g.\ \cite{Tao2003a}. However, the estimate in the endpoint
$p=q=\frac{n+3}{n+1}$ remains open.  For homogeneous solutions it is
known only in the case of the cone \cite{Tao2001}.
\end{remark}
\begin{remark}\label{rem:vec}
In fact, since Tao's endpoint result  \cite[Theorem 1.1]{Tao2001} holds for Hilbert space valued waves, we observe that one can deduce the $U^2$-estimate for the cone directly. This follows by noting that, given $U^2$-atoms $u=\sum_{I \in \mathcal{I}}\ind_I u_I$ and $v=\sum_{J \in \mathcal{J}}\ind_J v_J$, we have
\[
|uv|\leq \Big( \sum_{I \in \mathcal{I}} |u_I|^2\Big )^{\frac12}\Big( \sum_{J \in \mathcal{J}} |v_J|^2\Big )^{\frac12}=|U||V|
\]
with $\ell^2$-valued waves $U$ and $V$.
\end{remark}

\section{Mixed Norms and Generalisations to Small Scales}\label{sect:small}
In this section we give some consequences of the bilinear estimate in Theorem \ref{thm - bilinear Lp estimate}. Namely, we state an extension to mixed $L^q_t L^r_x$ spaces, and, in the case of the hyperboloid, we give a small scale version of Theorem \ref{thm - bilinear Lp estimate}. The small scale estimate will play a key role in the our application to the Dirac-Klein-Gordon system.

\subsection{Mixed Norms}\label{subsect:mixed}
Let $\Phi_1$ and $\Phi_2$ be phases satisfying Assumption \ref{assump on phase}. A standard $T T^*$ argument (see for instance \cite{Keel1998}), together Lemma \ref{lem - dispersion} implies that, provided $\frac{1}{q} + \frac{n-1}{2r} \les \frac{n-1}{4}$ and $q>2$ we have the Strichartz type bound
		\begin{equation}\label{eqn:strichartz} \| e^{ i t \Phi_j(-i\nabla)} f \|_{L^q_t L^r_x(\RR^{1+n})} \lesa \| f \|_{L^2_x}.\end{equation}
As in Step 3 of the proof of the globalisation lemma, by decomposing $V^2$ into $U^a$ atoms (see \cite[Lemma 6.4]{Koch2005} or \cite[Proposition 2.5 and Proposition 2.20]{Hadac2009}) we see that, for any $\frac{1}{a} + \frac{n-1}{2b} \les \frac{n-1}{2}$,
	$$ \| u v \|_{L^a_t L^b_x} \lesa \| u \|_{V^2_{\Phi_1}} \| v \|_{V^2_{\Phi_2}}.$$
Interpolating with Theorem \ref{thm - bilinear Lp estimate} then gives the following mixed norm version.

\begin{corollary}\label{cor:mixed}
Let $n \g 2$ and assume that $a>1$, $\frac{1}{a} + \frac{n+1}{2b} < \frac{n+1}{2}$, and
    \begin{equation}\label{eqn:cond on a b in cor} \frac{1}{a} + \frac{n-1}{4b}< \begin{cases} \frac{n+1}{4} \qquad &n\g 3 \\
    		 \frac{1}{2} + \frac{5}{12b} &n=2 \end{cases}.\end{equation}
Let $\Phi_1$, $\Phi_2$, and $u, v$ be as in the statement of Theorem \ref{thm - bilinear Lp estimate}. Then
	$$ \| u v \|_{L^a_t L^b_x} \lesa \| u \|_{V^2_{\Phi_1}} \| v \|_{V^2_{\Phi_2}}.$$
\end{corollary}

\begin{remark}\label{rem:inter}
  Let $p>\frac{n+3}{n+1}$. It is possible to deduce a weaker version of Theorem \ref{thm - bilinear Lp estimate} and Corollary \ref{cor:mixed} directly from the homogeneous estimate
  	\begin{equation}\label{eqn:homogeneous bilinear est} \| e^{ i t \Phi_1(-i\nabla)} f e^{ i t \Phi_2(-i\nabla)} g \|_{L^p_{t,x}(\RR^{1+n})} \lesa \|f \|_{L^2_x} \| g \|_{L^2_x} \end{equation}
 where the phases satisfy the conditions in Assumption \ref{assump on phase}, and $f, g \in L^2$ have the required support conditions. We sketch the argument as follows. By interpolating (\ref{eqn:homogeneous bilinear est}) with the trivial $L^\infty_t L^2_x$ bound, we deduce that for every $a>2$ we have
	$$ \| e^{ i t \Phi_1(-i\nabla)} f e^{ i t \Phi_2(-i\nabla)} g \|_{L^{a}_t L^{\frac{n+1}{n}}_x} \lesa \| f \|_{L^2_x} \| g \|_{L^2_x} $$
 By decomposing $V^2$ functions into $U^a$ atoms \cite{Koch2005,Hadac2009,Koch2014} and using the convexity of the $L^p$ spaces, we see that for $a>2$
  \[
  \|uv\|_{L^{a}_t L^{\frac{n+1}{n}}_x} \lesa \| u \|_{V^2_{\Phi_1}} \| v \|_{V^2_{\Phi_2}}.
  \]
Consequently, as in the proof of Corollary \ref{cor:mixed}, by interpolating with the standard Strichartz estimates, we obtain
	$$ \| u v \|_{L^a_t L^b_x} \lesa \| u \|_{V^2_{\Phi_1}} \| v \|_{V^2_{\Phi_2}}$$
provided that $a>1$, $\frac{1}{a} + \frac{n+1}{2b} < \frac{n+1}{2}$, and
	\begin{equation}\label{eqn:cond bilinear weak}
		\frac{1}{a} < \begin{cases} \frac{n-1}{n+3} \Big( \frac{n}{2} - \frac{n+1}{2b}\Big) + \frac{1}{2} \qquad &n\g 3 \\
	\frac{1}{2} &n=2. \end{cases}
	\end{equation}
In particular, the homogeneous bounds contained in the work of Lee-Vargas \cite{Lee2010} and Bejenaru \cite{Bejenaru2016}, implies a weaker version of our main result, with (\ref{eqn:cond on a b in cor}) in Corollary \ref{cor:mixed} replaced with (\ref{eqn:cond bilinear weak}). Note that condition (\ref{eqn:cond bilinear weak}) is much more restrictive than (\ref{eqn:cond on a b in cor}). This is most apparent in the low dimensional cases, for instance if $n=2$ then Corollary \ref{cor - small scale bilinear estimate} allows $a<2$ while (\ref{eqn:cond bilinear weak}) only allows the somewhat trivial (from a $V^2$ perspective) $a>2$. To summarise, our main result, Theorem \ref{thm - bilinear Lp estimate} not only  clarifies the dependence of the constant on the global properties of the  phases $\Phi_1$ and $\Phi_2$, but also presents a significant strengthening of the allowed exponents for the $V^2$ estimate.

We observe that the above argument, namely deducing a $V^2$ bound directly from the homogeneous estimate, has been used in \cite[Lemma 5.7 and its proof]{Sterbenz2010} in the case of the cone.
\end{remark}

\begin{remark}\label{rem:interpolation with K-G}
In the special case of the hyperboloid, $\Phi_j = \lr{\xi}_{m_j}$, or the paraboloid, $\Phi_j = |\xi|^2$, the Strichartz bound (\ref{eqn:strichartz}) holds in the larger region $\frac{1}{q} + \frac{n}{2r} \les \frac{n}{4}$. This can be used to improve the range of exponents in Corollary \ref{cor:mixed}, in particular (\ref{eqn:cond on a b in cor}) can be replaced with
  $$ \frac{1}{a} + \frac{n}{3b} < \frac{n+1}{3}. $$	
However, it is important to note that, in the case of the hyperboloid, some care has to be taken as the constant will now depend on the masses $m_j$.
\end{remark}

\subsection{Small Scale Bilinear Restriction Estimates}
In the case of hyperboloids we now generalise Theorem \ref{thm - bilinear Lp estimate}, similarly to \cite{Lee2008} in the case of the cone. Given $0<\alpha \lesa 1$, we define $\mc{C}_\alpha$ to be a collection of finitely overlapping caps of radius $\alpha$ on the sphere $\sph^{n-1}$. If $\kappa \in \mc{C}_\alpha$, we define $\omega(\kappa)$ to be the centre of the cap $\kappa$.

We consider the case $\Phi_j(\xi) = - \pm_j \lr{\xi}$ and define the corresponding $V^2_{\pm, m}$ space
as $ V^2_{\pm, m} = V^2_{\Phi_j}$, thus
	\begin{equation}\label{eqn:KG V2 defn} \| u \|_{V^2_{\pm, m}} = \| e^{ \pm it\lr{\nabla}_m} u(t) \|_{V^2}.\end{equation}
We define the corresponding $U^2_{\pm, m}$ space similarly.  Rescaling Theorem \ref{thm - bilinear Lp estimate} then gives the following optimal result.

\begin{corollary}\label{cor - small scale bilinear estimate}
Let $p>\frac{n+3}{n+1}$, $0\les m_1, m_2 \les 1$.
\begin{enumerate}
    \item\label{it:cor-part1} For any $\lambda \gtrsim m_1 + m_2$, $\frac{m_1+m_2}{\lambda}\lesa \alpha\lesa 1$,  $\kappa, \kappa' \in \mc{C}_\alpha$ with $\theta(\pm_1\kappa, \pm_2 \kappa') \approx  \alpha$, and
		$$ \supp \widehat{u} \subset \{ |\xi| \approx \lambda, \,\,  \tfrac{\xi}{|\xi|} \in \kappa\}, \qquad \supp \widehat{v} \subset \{ |\xi| \approx \lambda, \,\,   \tfrac{\xi}{|\xi|} \in \kappa'\},$$
we have the bilinear estimate
	$$ \| u v \|_{L^p_{t, x}} \lesa \alpha^{ n-1 - \frac{n+1}{p}}  \lambda^{ n-\frac{n+1}{p}} \| u \|_{V^2_{\pm_1, m_1}}  \| v \|_{V^2_{\pm_2, m_2}}.$$

    \item\label{it:cor-part2} For any $\lambda \gtrsim m_1 + m_2$,  $0<\alpha\ll \frac{m_1 + m_2}{ \lambda}$,  $\kappa, \kappa' \in \mc{C}_\alpha$, $c_1 \approx c_2 \approx \lambda$ with
        $$\theta(\pm_1\kappa, \pm_2 \kappa') \lesa   \alpha, \qquad | m_1 c_1 -  m_2 c_2| \approx \alpha \lambda^2,$$ and
		$$ \supp \widehat{u} \subset \big\{ \big||\xi\cdot \omega(\kappa)| - c_1 \big| \ll \alpha \lambda^2, \,\, \tfrac{\xi}{|\xi|} \in \kappa\big\}, \qquad \supp \widehat{v} \subset \big\{ \big||\xi\cdot \omega(\kappa')| - c_2 \big| \ll \alpha \lambda^2, \,\,\tfrac{\xi}{|\xi|} \in \kappa'\big\},$$
we have the bilinear estimate
	$$ \| u v \|_{L^p_{t, x}} \lesa \alpha^{ n - \frac{n+2}{p}}  \lambda^{ n+1 -\frac{n+2}{p}} \| u \|_{V^2_{\pm_1, m_1} } \| v \|_{V^2_{\pm_2, m_2}}.$$
\end{enumerate}
\end{corollary}

\begin{proof}
Fix $\pm_1=+$ and $\pm_2 = \pm$, the remaining cases follow from a reflection. We start with \eqref{it:cor-part1}. If $\alpha \approx 1$, then estimate follows from rescaling in $x$ together with an application of Theorem \ref{thm - bilinear Lp estimate}. Thus we may assume that $0<\alpha \ll 1$, and after a rotation, that
$\kappa$ is centred at $e_1$ and $\kappa'$ is centred at $\pm (1-\alpha^2)^\frac{1}{2} e_1 + \alpha e_2$. Similarly to \cite{Lee2008}, we define the rescaled  functions
	$$ u_{\lambda, \alpha}(t, x) = u\Big( \frac{t}{\alpha^2 \lambda}, \frac{ x_1}{\lambda} + \frac{t}{\alpha^2 \lambda}, \frac{ x'}{\alpha \lambda}\Big), \qquad v_{\lambda, \alpha}(t, x) =  v\Big( \frac{t}{\alpha^2 \lambda}, \frac{ x_1}{\lambda} + \frac{t}{\alpha^2 \lambda}, \frac{ x'}{\alpha \lambda}\Big)$$
(where we write $x = (x_1, x') \in \RR\times \RR^{n-1}$) and the phases
	$$ \Phi_1(\xi) = \frac{-1}{\alpha^2 \lambda} \Big(\big(  m_1^2 + \lambda^2\xi_1^2 + \alpha^2 \lambda^2 |\xi'|^2 \big)^{\frac{1}{2}} - \lambda \xi_1\Big), \qquad \Phi_2(\xi) = \frac{\mp 1}{\alpha^2 \lambda} \Big(  \big(  m_2^2 + \lambda^2 \xi_1^2 + \alpha^2\lambda^2 |\xi'|^2 \big)^{\frac{1}{2}} \mp \lambda \xi_1\Big)$$
with associated sets $ \Lambda_1 = \{ \xi_1 \approx 1, |\xi'| \ll 1\}$ and  $\Lambda_2 =  \{  \xi_1 \approx  \pm 1, \xi_2 \approx 1, |\xi''| \ll 1\}$ (where we write $\xi = (\xi_1, \xi_2, \xi'') \in \RR \times \RR \times \RR^{n-2}$). A computation gives $\supp \widehat{u}_{\alpha, \lambda} \subset \Lambda_1$ and
    $$ \Big[e^{ - i t \Phi_1( - i \nabla)} u_{\alpha, \lambda}(t) \Big](x) =  \Big[ e^{  i \frac{t}{\alpha^2 \lambda} \lr{\nabla}_{m_1} } u\big( \tfrac{t}{\alpha^2 \lambda}\big) \Big]\big( \tfrac{x_1}{\lambda}, \tfrac{x'}{\alpha \lambda} \big). $$
Similarly we can check that $\supp \widehat{v}_{\alpha, \lambda} \subset \Lambda_2$ and
         $$ \Big[e^{ - i t \Phi_2( - i \nabla)} v_{\alpha, \lambda}(t) \Big](x) =  \Big[ e^{ \pm i \frac{t}{\alpha^2 \lambda} \lr{\nabla}_{m_2} } v\big( \tfrac{t}{\alpha^2 \lambda}\big) \Big]\big( \tfrac{x_1}{\lambda}, \tfrac{x'}{\alpha \lambda} \big). $$
Therefore, after rescaling together with an application of Theorem \ref{thm - bilinear Lp estimate}, it is enough to check that the phases $\Phi_j$ satisfy Assumption \ref{assump      on phase} on the sets $\Lambda_j$. To this end, we start by noting that we can write
\[\nabla \Phi_1(\xi) = \frac{1}{ \big( \lambda^{-2} m_1^2 + \xi_1^2 + \alpha^2 |\xi'|^2 \big)^{\frac{1}{2}}}\bigg( \frac{ - \big( \frac{m_1}{\alpha \lambda}\big)^2 - |\xi'|^2 }{ \big( \lambda^{-2} m_1^2 + \xi_1^2 + \alpha^2 |\xi'|^2 \big)^{\frac{1}{2}} + \xi_1}, \xi' \bigg)\]
which shows that \eqref{it:ass2} in Assumption \ref{assump on phase} holds with $\mb{D}_2$ depending only on $N$ and $n$. A similar argument shows that $\Phi_2$ satisfies \eqref{it:ass2} in Assumption \ref{assump on phase}. On the other hand, to check condition \eqref{it:ass1} in Assumption \ref{assump on phase}, we invoke Lemma \ref{lem - simplification of cond ii}.
First, we observe that for any $\xi \in \Lambda_1$, $\eta \in \Lambda_2$, we have
	\begin{align*}
	\big| \nabla \Phi_1(\xi) - \nabla \Phi_2(\eta) \big| &\g \big| \p_2 \Phi_1(\xi) - \p_2 \Phi_2(\eta) \big| \\
	&= \Big| \frac{  \xi_2}{ ( \lambda^{-2} m_1^2 + \xi_1^2 + \alpha^2 |\xi'|^2)^\frac{1}{2}} \mp \frac{ \eta_2}{( \lambda^{-2} m_2^2 + \eta_1^2 + \alpha^2 |\eta'|^2)^\frac{1}{2} } \Big| \gtrsim 1	
	\end{align*}
and hence we can take $\mb{A}_1 \approx 1$. It remains to check \eqref{eqn - cond on phase II} in Lemma \ref{lem - simplification of cond ii}. We make use of the following elementary inequality;  if $(h^*, a^*) \in \RR^{n+1} \times \RR^1$ and $x, y \in \{ z \in \RR^{n+1} \mid |z|=|z-h^*| + a^* \}$, then
    \begin{equation}\label{eqn - cor small scale bilinear est - angle vs wedge identity}
        \bigg| \frac{x}{|x|} - \frac{y}{|y|} \bigg|^2 \g \frac{1}{4 |x| |y|} \bigg(  \frac{|x\wedge y|^2}{|x| |y|} + \frac{|(x-h^*) \wedge (y-h^*)|^2}{|x-h^*| |y-h^*|} \bigg).
    \end{equation}
To prove (\ref{eqn - cor small scale bilinear est - angle vs wedge identity}), we start by observing that since $x, y \in \{ |z| = |z-h^*| + a^* \}$, we have
     \begin{align*} \Big| \frac{x}{|x|} - \frac{y}{|y|} \Big|^2 &= \frac{1}{|x| |y|} \big( |x-y|^2 - \big| |x| - |y| \big|^2 \big) \\
        &= \frac{1}{|x| |y|} \big(\big|(x-h^*) - (y-h^*)\big|^2 - \big| |x-h^*| - |y-h^*| \big|^2 \big) \\
        &= \frac{|x-h^*| |y-h^*|}{|x| |y|} \Big| \frac{x-h^*}{|x-h^*|} - \frac{y-h^*}{|y-h^*|} \Big|^2.
     \end{align*}
The inequality (\ref{eqn - cor small scale bilinear est - angle vs wedge identity}) now follows from the identity $|\omega - \omega^*|^2 \g \frac{|\omega \wedge \omega^*|^2}{2}$ for $\omega, \omega^* \in \sph^{n+1}$. We now return to checking \eqref{eqn - cond on phase II} in Lemma \ref{lem - simplification of cond ii}, we only check the case $j=1$ as the remaining case is identical. Let  $\xi, \eta \in \Sigma_1(a, h)$ for some $(a, h) \in \RR^{1+n}$ such that $\xi - h, \eta -h \in \Lambda_2$. A computation gives
\begin{align}
	 \big| \big(\nabla \Phi_j(\xi) - &\nabla \Phi_j(\eta) \big) \cdot (\xi - \eta) \big|\notag\\
	 			&= \alpha^{-2} \Bigg| \Bigg( \frac{(\xi_1, \alpha^2 \xi')}{|(\lambda^{-1} m_1, \xi_1, \alpha^2 \xi')|} - \frac{(\eta_1, \alpha^2 \eta') }{|(\lambda^{-1} m_1, \xi_1, \alpha \xi')|} \Bigg) \cdot (\xi - \eta) \Bigg| \notag\\
	 			&=  \alpha^{-2} \frac{ |(\lambda^{-1} m_1, \xi_1, \alpha \xi')|+ |(\lambda^{-1} m_1, \eta_1, \alpha \eta')|}{2}  \Bigg|  \frac{(\lambda^{-1} m_1, \xi_1, \alpha \xi')}{|(\lambda^{-1} m_1, \xi_1, \alpha \xi')|} - \frac{(\lambda^{-1} m_1, \eta_1, \alpha^2 \eta') }{|(\lambda^{-1} m_1, \xi_1, \alpha \xi')|} \Bigg|^2 \notag \\
&\approx \alpha^{-2} \bigg| \frac{x}{|x|} - \frac{y}{|y|} \bigg|^2 \label{eqn - cor small scale bilinear est - curvature bound (i)}
\end{align}
where we take $x= (\lambda^{-1} m_1, \xi_1, \alpha \xi')$ and $y = (\lambda^{-1} m_1, \eta_1, \alpha \eta')$. Note that the condition $\xi \in \Sigma_1(a, h)$ becomes $|x| = |x-h^*| + a^*$ with $h^* = (\lambda^{-1}m_2 - \lambda^{-1} m_1, h_1, \alpha h')$ and $a^* = \alpha^2 a$. In particular, since $|x|\approx |y| \approx |x-h^*| \approx |y-h^*| \approx 1$, an application of (\ref{eqn - cor small scale bilinear est - angle vs wedge identity}) gives
   \begin{equation}\label{eqn - cor small scale bilinear est - first x y ineq} \bigg| \frac{x}{|x|} - \frac{y}{|y|} \bigg|^2  \gtrsim |x \wedge y|^2 + |(x-h^*) \wedge (y-h^*)|^2. \end{equation}
The required bound \eqref{eqn - cond on phase II} with $\mb{A}_2\approx 1$ now follows in the region $|\xi_1 - \eta_1| \lesa |\xi' - \eta'|$ by noting that
    $$ | x \wedge y| \g  \alpha | \xi_1 \eta' - \eta_1 \xi'| \g \alpha \big(|\xi' - \eta'| |\xi_1| - |\xi'| |\xi_1 - \eta_1|\big) \approx \alpha |\xi' - \eta'| \approx \alpha |\xi - \eta| $$
and applying the inequalities \eqref{eqn - cor small scale bilinear est - curvature bound (i)} and \eqref{eqn - cor small scale bilinear est - first x y ineq}. On the other hand, if $|\xi_1 - \eta_1| \gg |\xi' - \eta'| $, then as $\xi - h, \eta - h \in \Lambda_2$, we have
     \begin{align*} | (x-h^*) \wedge (y - h^*)| &\g \alpha | (\xi_1 - h_1) (\eta_2 - h_2) - (\eta_1 - h_1)(\xi_2 - h_2)| \\
        &\g \alpha \big(|\xi_1 - \eta_1| |\eta_2 - h_2| - |\xi_2 - \eta_2| |\eta_1 - h_1| \big) \approx \alpha |\xi_1 - \eta_1| \approx \alpha |\xi - \eta|
     \end{align*}
which again gives \eqref{eqn - cond on phase II} with $\mb{A}_2\approx 1$. Thus the phases $\Phi_j$ satisfy Assumption \ref{assump      on phase} with $\mb{D}_1\approx \mb{D}_2\approx 1$ and therefore Part (i) follows.

We now turn to the proof of Part \eqref{it:cor-part2}. The argument is similar to \eqref{it:cor-part1}, but we need a further rescaling to exploit the radial separation condition. As before, after rotating, we may assume that $\omega(\kappa_1) = e_1$. Define the rescaled functions
	$$ u_{\lambda, \alpha}^\#(t, x) = u\Big( \frac{t}{\alpha^2 \lambda}, \frac{ x_1}{\alpha \lambda^2} + \frac{t c_1}{\alpha^2 \lambda\langle c_1 \rangle_{m_1}}, \frac{ x'}{\alpha \lambda}\Big)\qquad v_{\lambda, \alpha}^\#(t, x) =  v\Big( \frac{t}{\alpha^2 \lambda}, \frac{ x_1}{\alpha \lambda^2}  + \frac{t c_1}{\alpha^2 \lambda\langle c_1 \rangle_{m_1}}, \frac{ x'}{\alpha \lambda}\Big)$$
(where, as previously, we write $x = (x_1, x') \in \RR\times \RR^{n-1}$) and the phases
    $$ \Phi_1(\xi) = \frac{-1}{\alpha^2 \lambda} \Big(\big(  m_1^2 + (\alpha \lambda^2\xi_1)^2 + \alpha^2 \lambda^2 |\xi'|^2 \big)^{\frac{1}{2}} - \frac{\alpha \lambda^2 c_1}{\langle c_1 \rangle_{m_1}} \xi_1\Big)$$
and
    $$\Phi_2(\xi) = \frac{\mp 1}{\alpha^2 \lambda} \Big(  \big(  m_2^2 + ( \alpha \lambda^2 \xi_1)^2 + \alpha^2\lambda^2 |\xi'|^2 \big)^{\frac{1}{2}} \mp \frac{\alpha \lambda^2 c_1}{\langle c_1 \rangle_{m_1}}\xi_1 \Big)$$
with associated sets $ \Lambda_1 = \{ |\xi_1 -  \frac{1}{\alpha \lambda^2} c_1| \ll 1, |\xi'| \ll 1\}$ and  $\Lambda_2 =  \{  |\xi_1  \mp \frac{1}{\alpha \lambda^2} c_2| \ll 1, |\xi'| \lesa 1\}$. As previously, a computation shows that
$\supp \widehat{u}^\#_{\alpha, \lambda} \subset \Lambda_1$, $\supp  \widehat{v}^\#_{\alpha, \lambda} \subset \Lambda_2$ and we have the identities
    $$ \Big[e^{ - i t \Phi_1(-i \nabla)} u^\#_{\alpha, \lambda}(t) \Big](x) = \Big[ e^{  i t \lr{\nabla}_{m_1}} u\big( \tfrac{t}{\alpha^2 \lambda}\big) \Big]\big( \tfrac{x_1}{\alpha \lambda^2}, \tfrac{x'}{\alpha \lambda} \big)$$
and
    $$ \Big[e^{ - i t \Phi_2(-i \nabla)} v^\#_{\alpha, \lambda}(t) \Big](x) = \Big[ e^{  \pm i t \lr{\nabla}_{m_2}} v\big( \tfrac{t}{\alpha^2 \lambda}\big) \Big]\big( \tfrac{x_1}{\alpha \lambda^2}, \tfrac{x'}{\alpha \lambda} \big).$$
Thus, as in the proof of (i), after rescaling and an application of Theorem \ref{thm - bilinear Lp estimate}, it is enough to check that the phases $\Phi_j$ satisfy Assumption \ref{assump      on phase} on the sets $\Lambda_j$. To this end, note that we can write
\[\p_1 \Phi_1 =  \frac{ \frac{m_1^2}{\alpha \lambda^3} \big( (\alpha \lambda^2 \xi_1)^2 - c_1^2\big) - \big( \frac{c_1}{\lambda}\big)^2 \alpha \lambda |\xi'|^2}{f(\alpha \lambda \xi_1, \alpha \xi')} \]
for some smooth function $f$ with $f\approx 1$ on $\Lambda_1$. Since $\p^M_{\xi_1} [(\alpha \lambda^2 \xi_1)^2 - c_1^2] \lesa \alpha \lambda^3$ for all $M \g 0$ and $\xi_1 \in \Lambda$, we see that $\Phi_1$ satisfies \eqref{it:ass2} in Assumption \ref{assump      on phase} with constant depending only on $n$ and $N$. A similar argument, using the fact that $\frac{\lambda}{\alpha} |\frac{c_1}{\langle c_1 \rangle_{m_1}} - \frac{c_2}{\langle c_2 \rangle_{m_2}}| \approx 1$, shows that $\Phi_2$ also satisfies \eqref{it:ass2} in Assumption \ref{assump      on phase}. On the other hand, to check \eqref{it:ass1} in Assumption \ref{assump      on phase}, we use Lemma \ref{lem - simplification of cond ii}.  Concerning the transversality condition \eqref{eq:trans}, we observe that for $\xi \in \Lambda_1$, $\eta \in \Lambda_2$, we have $|\xi_1| \approx |\eta_1| \approx \frac{1}{\alpha \lambda}$ and
     $$ | \xi_1^2 m_2^2 - \eta_1^2 m_1^2| \approx \frac{m_1 + m_2}{\alpha \lambda}, \qquad \alpha^2 \big| \xi^2_1 |\eta'|^2 - \eta_1^2 |\xi'|^2 \big| \lesa \lambda^{-2} \ll \lambda^{-2} \frac{m_1 + m_2}{\alpha \lambda}$$
Therefore
    \begin{align*}
        \big| \nabla \Phi_1(\xi) - \nabla \Phi_2(\eta) \big| &= \bigg| \frac{ ( \lambda^2 \xi_1, \xi')}{ ( \lambda^{-2} m_1^2 + \alpha^2 \lambda^2 \xi_1^2 + \alpha^2 |\xi'|^2)^\frac{1}{2}} \mp \frac{ ( \lambda^2 \eta_1, \eta')}{ ( \lambda^{-2} m_2^2 + \alpha^2 \lambda^2 \eta_1^2 + \alpha^2 |\eta'|^2)^\frac{1}{2}} \bigg| \\
        &\gtrsim  \lambda^3 \alpha \big| \xi_1^2 ( \lambda^{-2} m_2^2 + \alpha^2 \lambda^2 \eta_1^2 + \alpha^2 |\eta'|^2) -  \eta_1^2 ( \lambda^{-2} m_1^2 + \alpha^2 \lambda^2 \xi_1^2 + \alpha^2 |\xi'|^2) \big| \\
        &\approx m_1 + m_2 \gtrsim 1
    \end{align*}
so that \eqref{eq:trans} holds with $\mb{A}_1\approx 1$. We now check the curvature condition \eqref{eqn - cond on phase II} for $j=1$.  Let $\xi, \eta \in \Sigma_1(a, h)$. Repeating the computation (\ref{eqn - cor small scale bilinear est - curvature bound (i)}) we deduce that
    \begin{align*}
      \big| \big(\nabla \Phi_1(\xi) - \nabla \Phi_1(\eta) \big) \cdot (\xi - \eta) \big| \approx \alpha^{-2} \bigg| \frac{x}{|x|} - \frac{y}{|y|}\bigg|^2 \gtrsim \alpha^{-2} \big( |x\wedge y|^2 + |(x-h^*) \wedge (y-h^*)|^2\big)
    \end{align*}
where $x= (\lambda^{-1} m_1, \alpha \lambda \xi_1, \alpha \xi')$,  $y = (\lambda^{-1} m_1, \alpha \lambda \eta_1, \alpha \eta')$, $h^* = (\lambda^{-1}m_2 - \lambda^{-1} m_1, \alpha \lambda h_1, \alpha h')$, and we used the fact that $x, y, x-h^*, y-h^*$ all have length $1$. It thus remains to show that 
      $$|x\wedge y| + |(x-h^*) \wedge (y-h^*)| \gtrsim \alpha |\xi-\xi'|$$
since then \eqref{eqn - cond on phase II} holds with $\mb{A}_2\approx 1$. If $|\xi_1 - \eta_1| \lesa |\xi' - \eta'|$ we simply observe as previously that
    $$ | x \wedge y| \g  \alpha | \alpha \lambda \xi_1 \eta' - \alpha \lambda \eta_1 \xi'| \g \alpha \big(|\xi' - \eta'| \alpha \lambda |\xi_1| - |\xi'| \alpha \lambda |\xi_1 - \eta_1|\big) \approx \alpha |\xi' - \eta'| \approx \alpha |\xi - \eta| $$
 On the other hand, if $|\xi_1 - \eta_1| \gtrsim |\xi' - \eta'| $, then as $\xi - h, \eta - h \in \Lambda_2$, we have
    $$ |x\wedge y| + |(x-h^*)\wedge (y-h^*)| \g \alpha m_1 |\xi_1 - \eta_1| + \alpha m_2 | (\xi_1 - h_1) - (\eta_1 - h_2)| \gtrsim \alpha |\xi-\eta|.$$
 An identical argument shows that $\Phi_2$ also satisfies the curvature condition. Thus the phases $\Phi_j$ satisfy Assumption \ref{assump      on phase} with $\mb{D}_1\approx \mb{D}_2\approx 1$ and therefore Part (ii) follows.
\end{proof}

The $\alpha$ and $\lambda$ dependence in Corollary \ref{cor - small scale bilinear estimate} is sharp. At least for (ii), this can be seen with the following example. Let
    $$\Omega_j = \{ |\xi_1 - c_j|  \ll \alpha \lambda^2, \,\, |\xi'| \ll \alpha \lambda \}$$
with $|c_1 - c_2| \lesa \alpha \lambda^2$, $c_1 \approx c_2 \approx \lambda$, and $\alpha \ll \lambda^{-1}$. Define $\widehat{f}(\xi) = \ind_{\Omega_1}(\xi)$, $\widehat{g}(\xi) = \ind_{\Omega_2}(\xi)$ and
    $$ u = e^{ i t \lr{\nabla}} f, \qquad v = e^{ it \lr{\nabla}} g. $$
Then
    $$ \| u \|_{V^2_{\lr{\nabla}}} = \| f \|_{L^2_x} = |\Omega_1|^{\frac{1}{2}} $$
and similarly $\| v \|_{V^2_{\lr{\nabla}}} = |\Omega_2|^{\frac{1}{2}}$. On the other hand we have
    \begin{align*}
      (uv)(t, x) &=  \int_{\RR^n} \int_{\RR^n} \widehat{u}(t, \xi) \widehat{v}(t, \eta) e^{ i x \cdot (\xi + \eta) } d\xi d\eta =  \int_{\Omega_1} \int_{\Omega_2} e^{ i t (\lr{\xi} + \lr{\eta})} e^{ i x \cdot (\xi + \eta) } d\xi d\eta.
    \end{align*}
The idea is to try and find a set $A \subset \RR^{1+n}$ such that the phase is essentially constant for $(t, x) \in A$. We start by noting that for  $\xi \in \Omega_1$ we have
    $$ \lr{\xi} - \frac{ 1 + c_1 \xi_1}{\lr{c_1}}\approx \lambda^{-3} \big| ( 1 + |\xi|^2) ( 1 + c_1^2) - (1 + c_1 \xi_1)^2 \big| = \lambda^{-3} \big| ( \xi_1 - c_1)^2 + (1+c_1^2) |\xi'|^2 \big| \approx \alpha^2 \lambda$$
and hence
    $$ \Big| \lr{\xi}  - \lr{c_1}^{-1} - \frac{c_1}{\lr{c_1}} \xi_1 \Big| \lesa \alpha^2 \lambda.$$
Similarly, since
    $$ \bigg| \frac{c_1}{\lr{c_1}} - \frac{c_2}{\lr{c_2}} \bigg| \approx \lambda^{-2} | c_1 \lr{c_2} - c_2 \lr{c_1} | \approx \lambda^{-3} |c_1 - c_2| \approx \frac{\alpha}{\lambda}$$
we deduce that for $\eta \in \Omega_2$
    $$ \Big|\lr{\eta} - \lr{c_2}^{-1} - \Big( \frac{c_2}{\lr{c_2}} - \frac{c_1}{\lr{c_1}}\Big) c_2  - \frac{c_1}{\lr{c_1}} \eta_1\Big|  \les \Big|\lr{\eta} - \lr{c_2}^{-1} - \frac{c_2}{\lr{c_2}} \eta_1 \Big| + \Big| \frac{c_1}{\lr{c_1}} - \frac{c_2}{\lr{c_2}}\Big| | \eta_1 - c_2| \lesa \alpha^2 \lambda. $$
In particular, for $|t| \ll (\alpha^2 \lambda)^{-1}$, $|x_1 + \frac{c_1}{\lr{c_1}} t| \ll (\alpha \lambda^2)^{-1}$, and $|x'| \ll (\alpha \lambda)^{-1}$, the phase is essentially constant and hence
    \begin{align*} | (&uv)(t, x)|\\
        &= \Bigg|\int_{\Omega_1} \int_{\Omega_2} e^{ i t ( \lr{\xi}  - \lr{c_1}^{-1} - \frac{c_1}{\lr{c_1}} \xi_1)} e^{ it (\lr{\eta} - \lr{c_2}^{-1} - ( \frac{c_2}{\lr{c_2}} - \frac{c_1}{\lr{c_1}}) c_2  - \frac{c_1}{\lr{c_1}} \eta_1)} e^{ i (x_1 + t \frac{c_1}{\lr{c_1}}) (\xi_1 + \eta_1 - c_1 - c_2) + x' \cdot (\xi' + \eta')} d\xi d\eta \Bigg| \\
            &\gtrsim |\Omega_1| |\Omega_2|
    \end{align*}
which then implies that
    $$ \| u v \|_{L^p_{t, x}} \gtrsim \big( \alpha^{n+2} \lambda^{n+2} \big)^{-\frac{1}{p}} \times  |\Omega_1| |\Omega_2|. $$
Therefore, if the estimate
    $$ \| uv \|_{L^p_{t, x}} \les C(\alpha, \lambda) \| u \|_{V^2_{\lr{\nabla}}} \| v \|_{V^2_{\lr{\nabla}}}$$
holds, then we must have
        $$ \big( \alpha \lambda \big)^{-\frac{n+2}{p}}   |\Omega_1| |\Omega_2| \lesa C  |\Omega_1|^\frac{1}{2} |\Omega_2|^\frac{1}{2}.$$
Since $|\Omega_1| \approx |\Omega_2| \approx \alpha^n \lambda^{n+1}$, after rearranging, this becomes $C \gtrsim  \alpha^{ n - \frac{n+2}{p}} \lambda^{n+1 - \frac{n+2}{p}}$, which matches the bound obtained in Corollary \ref{cor - small scale bilinear estimate}.

\section{The Dirac-Klein-Gordon System}\label{sect:prep}

In this section we set up notation and reduce the DKG system to the first order system (\ref{eqn - DKG system reduced}). We then give the proof of Theorem \ref{thm:dkg}, up to the crucial nonlinear estimates, which are postponed to Section \ref{sect:mult}. In the remainder of this article, as we now only consider the DKG system,  the dimension is fixed $n=3$.

\subsection{Notation and Setup}\label{subsect:flp}
Fix a smooth function $\rho \in C^\infty_0(\RR)$ such that  $\supp \rho \subset \{ \tfrac{1}{2}< t< 2\}$ and
	$$  \sum_{ \lambda \in 2^{\ZZ}} \rho\big( \tfrac{t}{\lambda}\big) =1,$$
and let $\rho_1 = \sum_{\lambda \les 1} \rho(\frac{t}{\lambda})$ with $\rho_1(0) = 1$. Similarly, we let $Q_\mu$ be a finitely overlapping collection of cubes of diameter $\frac{\mu}{1000}$ covering $\RR^3$, and fix $(\rho_q)_{q \in Q_\mu}$ to be a corresponding subordinate partition of unity. We now define the standard dyadic  Fourier cutoffs, for $\lambda \in 2^\NN$, $\lambda>1$,  $q \in Q$, $d \in 2^\ZZ$
	$$P_\lambda = \rho\big( \tfrac{ |-i\nabla|}{ \lambda}\big), \qquad P_1 = \rho_1(|-i\nabla|),  \qquad P_q = \rho_q(|-i \nabla|), \qquad C^{\pm, m}_d = \rho\big( \tfrac{- i \p_t \pm \lr{ - i \nabla }_m}{d}\big).$$
We also let  $C^{\pm, m}_{\les d} = \sum_{ d' \les d} C^{\pm, m}_{d'}$, any related multipliers such as $C^{\pm, m}_{\g d}$ are defined analogously. To simplify notation somewhat, we make the convention that
    $$ C_d = C^{+, 1}_d, \qquad \mc{C}^{\pm}_d = \Pi_\pm C^{\pm, M}_d$$
where $M$ will denote the mass of the spinor in (\ref{eq:dkg}) and $\Pi_{\pm}$ as defined below. Given $ \alpha \les 1$, we let $(\rho_\kappa)_{\kappa \in \mc{C}_\alpha}$ be a smooth partition of unity subordinate to the conic sectors $\{ \xi \not = 0, \frac{\xi}{|\xi|} \in \kappa \}$, and define the angular Fourier localisation multipliers as
		$$ R_\kappa = \rho_\kappa( - i \nabla).$$

We use the well-known fact that for any $1\leq p,q \leq \infty$ the modulation cutoff multipliers are uniformly disposable in $L^q_t L^r_x$ for certain scales, namely we have the bounds
\begin{equation}\label{eqn:dispose}
\|C^{\pm, m}_{d}P_\lambda R_\kappa u\|_{L^q_t L^r_x}+\|C^{\pm, m}_{\les d}P_\lambda R_\kappa u\|_{L^q_t L^r_x}\lesa  \|P_\lambda R_\kappa u\|_{L^q_t L^r_x},
\end{equation}
provided that $\kappa \in \mathcal{C}_\alpha$ and $d\gtrsim \alpha^2 \lambda$. Similarly, by writing $C_d^{\pm, m} = e^{ \mp i t \lr{\nabla}_m} \rho(\frac{ - i \p_t}{d}) e^{ \pm i t \lr{\nabla}_m}$, and using the fact that convolution with $L^1_t(\RR)$ functions is bounded on $V^2$, we deduce that for every $d \in 2^\ZZ$
    \begin{equation}\label{eqn:dispose2}
        \| C_{\les d}^{\pm, m} u \|_{V^2_{\pm,m}} \lesa \| u \|_{V^2_{\pm,m}}.
    \end{equation}

To deal with solutions to the Dirac equation, we follow the by now standard approach used in \cite{D'Ancona2007b, Bejenaru2015} and define the projections
 $$ \Pi_\pm(\xi) = \frac{1}{2} \Big( I \pm  \frac{1}{\lr{\xi}_M} \big( \xi_j \gamma^0\gamma^j + M \gamma^0\big) \Big)$$
and the associated Fourier multiplier $\widehat{(\Pi_\pm f)}(\xi) = \Pi_\pm(\xi) \widehat{f}(\xi)$. A computation shows that $\Pi_+ \Pi_- = \Pi_- \Pi_+ = 0$ and $\Pi_\pm^2 = \Pi_\pm$. Moreover, given any spinor $\psi$ we have
	$$\psi = \Pi_+ \psi + \Pi_- \psi, \qquad  (- i \gamma^\mu \p_\mu + M )\Pi_\pm \psi = \gamma^0 (  - i \p_t \pm \lr{ - i \nabla}_M) \psi. $$
As in the paper of Bejenaru-Herr \cite{Bejenaru2015}, we can now reduce the original system (\ref{eq:dkg}) to a first order system as follows. Suppose we have a solution $(\psi_\pm, \phi_+)$ to
    \begin{equation}\label{eqn - DKG system reduced}\begin{split}
       \big(- i \p_t  \pm \lr{\nabla}_M \big) \psi_\pm &= \Pi_\pm \big( \Re(\phi_+) \gamma^0 \psi \big) \\
         \big(- i \p_t + \lr{\nabla}_m \big)\phi_+ &= \lr{\nabla}^{-1}_m (\psi^\dagger \gamma^0 \psi)\\
         \psi_\pm(0) &= f_\pm, \qquad \phi_+(0) = g_+
    \end{split}
\end{equation}
where $\psi = \Pi_+ \psi_+ + \Pi_- \psi_-$ and the data $(f_{\pm}, g_+)$ satisfies $\Pi_\pm f_\pm = f_\pm$. If we let $\phi = \Re(\phi_+)$, then since $\psi^\dagger \gamma^0 \psi$ is real-valued,  we deduce that
     \begin{align*} 2 (\phi + i \lr{\nabla}^{-1}_m \p_t \phi) &= \phi_+ + i \lr{\nabla}^{-1}_m \p_t \phi_+ + \overline{ (\phi_+ - i \lr{\nabla}^{-1}_m \p_t \phi_+ )} \\
     &= 2 \phi_+ - \lr{\nabla}^{-2}_m (\psi^\dagger \gamma^0 \psi) +  \lr{\nabla}^{-2}_m \overline{(\psi^\dagger \gamma^0 \psi)} = 2 \phi_+.\end{align*}
Consequently, if we take $g_+ = \phi(0) + i \lr{\nabla}^{-1}_m\p_t \phi(0)$, a simple computation shows that $(\psi, \phi)$ is a solution to the original DKG system (\ref{eq:dkg}). Note that, after rescaling, it suffices to consider the case $m=1$. Therefore, to prove Theorem \ref{thm:dkg}, it is enough to construct global solutions to the reduced system (\ref{eqn - DKG system reduced}) with $m=1$.

\subsection{Analysis on the Sphere}\label{subsect:anasphere}
We require some basic facts on analysis on the sphere $\sph^2$ which can be found in, for instance, \cite{Stein1971,Strichartz1972,Sterbenz2007}. Let $Y_\ell$ denote the set of homogeneous harmonic polynomials of degree $\ell$, and let $y_{\ell, n}$, $n=0, ..., 2\ell$ be an orthonormal basis for $Y_\ell$ with respect to the inner product
	$$ \lr{ y_{\ell, n}, y_{\ell', n'} }_{L^2(\sph^2)} = \int_{\sph^2} \big[ y_{\ell, n}(\omega) \big]^\dagger y_{\ell', n'}(\omega) d\sph(\omega).$$
Given $f \in L^2(\RR^3)$, we have the orthogonal (in $L^2(\RR^3)$) decomposition
	$$ f(x)  = \sum_{\ell} \sum_{n=0}^{2\ell} \lr{f(|x| \omega), y_{\ell, n}(\omega)}_{L^2_\omega(\sph^2)} y_{\ell, n}\big( \tfrac{x}{|x|}\big) .$$
For $N>1$, we define the spherical Littlewood-Paley projections
	$$ (H_N f)(x) = \sum_{\ell \in \NN} \sum_{n=0}^{2\ell}  \rho\big( \tfrac{\ell}{N} \big)  \lr{f(x), y_{\ell, n}}_{L^2(\sph^2)} y_{\ell, n}\big(\tfrac{x}{|x|}\big) \qquad  H_1 = \sum_{\ell \in \NN} \sum_{n=0}^{2\ell}  \rho_1( \ell) \lr{f(x), y_{\ell, n}}_{L^2(\sph^2)} \, y_{\ell, n}\big(\tfrac{x}{|x|}\big).$$
Fractional powers of the angular derivatives $\lr{\Omega}$ are then defined as
\begin{equation}\label{eqn - spec defn of ang der} \lr{\Omega}^\sigma f = \sum_{N  \in 2^\NN} N^\sigma H_N f.\end{equation}
If we let $\Omega_{ij} = x_j \p_j - x_j \p_i$ denote the standard infinitesimal generators of the rotations on $\RR^3$, then a computation gives
    $$ \| \Omega_{ij} H_N f \|_{L^2_x(\RR^3)} \approx N \| H_N f \|_{L^2_x(\RR^3)}.$$
In addition, if $\Delta_{\sph^2}$ denotes the Laplacian on the sphere of radius $|x|$, then $\Delta_{\sph^2} = \sum_{j<k} \Omega_{ij}^2$. These facts are not explicitly required in the following, and we shall only make use of the spectral definition (\ref{eqn - spec defn of ang der}). More important for our purposes, are the basic properties of the multipliers $H_N$.

\begin{lemma}\label{lem - prop of spherical proj}
Let $N \g 1$. Then $H_N$ is uniformly (in $N$) bounded on $L^p(\RR^3)$, and $H_N$ commutes with all radial Fourier multipliers. Moreover, if $N' \g 1$, then either $N \sim N'$ or
		$$ H_N \Pi_\pm H_{N'} = 0.$$
\end{lemma}	
\begin{proof}
The first claim follows from \cite{Strichartz1972}. To prove the second claim, let $T$ be a radial Fourier multiplier with $\widehat{Tf}(\xi) = \sigma(|\xi|) \widehat{f}(\xi)$. It is enough to show that, if $f(x) = a(|x|) y_\ell(\frac{x}{|x|})$ for some $y_\ell \in Y_\ell$, then $Tf = b(|x|) y_\ell(\frac{x}{|x|})$ for some $b(|x|)$ depending on $a$ and $\sigma$. But this follows directly from  \cite[page 158]{Stein1971}. To prove the final claim, suppose that $N \gg N'$ or $N \ll N'$. Our goal is to show that $H_N \Pi_\pm H_{N'} =0$. Since $H_N$ commutes with radial Fourier multipliers, it is enough to show that $H_N (\p_j f) = 0$ in the case $f(x) = a(|x|) y_{\ell'}(\frac{x}{|x|})$ with $y_{\ell'} \in Y_{\ell'}$ and $\frac{N'}{2} \les \ell' \les 2N'$. Since $\p_j = \frac{x_j}{|x|} \p_r + \sum_k \frac{x_k}{|x|^2} \Omega_{jk}$ where $\p_r = \frac{x}{|x|} \cdot \nabla$, and $\p_r \big( y_{\ell'}(\frac{x}{|x|}) \big) = 0$, we can reduce further to just showing that $H_N( x_k \Omega_{jk} y_{\ell'} ) = 0$ which corresponds to checking that 	
	\begin{equation}\label{eqn - lem prop of spherical proj - orthog est} \lr{ y_{\ell}, x_k \Omega_{kj} y_{{\ell'}} }_{L^2(\sph^2)} = 0
	\end{equation}
for every $\frac{N}{2} \les \ell \les 2 N$. Since $x_k \Omega_{kj} y_{{\ell'}}$ is a polynomial of order ${\ell'} + 1$, by the orthogonality of the polynomials $y_{\ell}$, (\ref{eqn - lem prop of spherical proj - orthog est}) clearly holds  if $\ell > \ell' + 1$. On the other hand, after an application of integration by parts, we obtain
	$$ \lr{ y_{\ell}, x_k \Omega_{kj} y_{\ell'} }_{L^2(\sph^2)} = \lr{ \Omega_{kj}(x_k y_{\ell}),  y_{\ell'} }_{L^2(\sph^2)}$$
since $\Omega_{kj}(x_k y_{\ell})$ is a polynomial of order $\ell + 1$, we see that again (\ref{eqn - lem prop of spherical proj - orthog est}) holds if $\ell' > \ell + 1$.
\end{proof}
	
An application of Lemma \ref{lem - prop of spherical proj} shows that $H_N$ commutes with the $P_\lambda$ and $C_d$ multipliers since we may write $C_d^{\pm, m} = e^{\mp i t \lr{\nabla}_m} \rho(\frac{-i\p_t}{d}) e^{\pm it \lr{\nabla}_m}$. On the other hand, it is important to note that $H_N$ \emph{does not} commute with the cube and cap localisation operators $R_\kappa$ and $P_q$.

\subsection{Norms and the energy inequality}
Fix $0<\sigma \ll 1$, $\frac{1}{2}<\frac{1}{a} <\frac{1}{2} + \frac{\sigma}{1000}$, and $b = \frac{3}{a} -1$, and define
	$$ \| u \|_{Y^{\pm, m}_{\lambda, N}} = \lambda^{\frac{1}{a} - b} \sup_{d\in 2^\ZZ} d^{b} \| C^{\pm, m}_d P_\lambda H_N u \|_{L^a_t L^2_x} $$
and
    $$ \| u \|_{F^{\pm, m}_{\lambda, N}} = \| P_\lambda H_N u \|_{V^2_{\pm,m}} + \| u \|_{Y^{\pm, m}_{\lambda, N}}.$$
We also let
    $$ \| u \|_{F^{s, \sigma}_{\pm,m}} = \Big( \sum_{\lambda \g 1} \sum_{N \g 1}  \lambda^{2 s} N^{2\sigma } \| u \|_{F^{\pm, m}_{\lambda, N}}^2 \Big)^\frac{1}{2}$$
and define the Banach space
    $$ F^{s, \sigma}_{\pm, m} = \big\{ u \in C( \RR,
    \lr{\Omega}^{-\sigma} H^s) \, \big| \, \| u \|_{F^{s,
        \sigma}_{\pm, m}}<\infty \big\}. $$
For the remainder of this section, let $\sigma_M=\sigma$ if $M\g\tfrac12$ and $\sigma_M=\tfrac{7}{30}+\sigma$ if $0<M<\tfrac{1}{2}$. Thus $\sigma_M$ corresponds to amount of angular regularity in the statement of Theorem \ref{thm:dkg}. We will construct a solution $(\psi_\pm, \phi_+) \in F^{0, \sigma_M}_{\pm, M} \times F^{\frac{1}{2}, \sigma_M}_{\pm, 1}$ to the reduced system (\ref{eqn - DKG system reduced}). Thus we work in a frequency localised $V^2$ space, with the additional component $Y^{\pm, m}_{\lambda, N}$ needed to control the solution in the high modulation region, for the latter cp.\ \cite[Section 4]{Bejenaru2015b}.

There are three basic properties of $V^2_{\pm,m}$ which we exploit in the following. The first is a simple bound in the high modulation region, see \cite[Corollary 2.18]{Hadac2009} for a proof.

\begin{lemma}\label{lem - norm controls Xsb}
Let $m \g 0$ and $2\les q \les \infty$. For any $d \in 2^\ZZ$  we have
        $$ \| C_d^{\pm, m}  u \|_{L^q_t L^2_x} \lesa d^{-\frac{1}{q}} \| u \|_{V^2_{\pm,m}}.$$
\end{lemma}

The second key property is a standard energy inequality, which reduces the problem of estimating a Duhamel integral in $F^{\pm, M}_{\lambda, N}$, to controlling a trilinear integral.

\begin{lemma}\label{lem - energy ineq for F norms}
  Let  $F \in L^\infty_t L^2_x$, and suppose that
      $$ \sup_{ \| P_\lambda H_N v \|_{V^2_{\pm,m}} \lesa 1} \Big| \int_\RR \lr{ P_\lambda H_N v(t), F(t) }_{L^2_x} dt\Big| < \infty.$$
  If $u \in C(\RR, L^2_x)$ satisfies  $- i \p_t u  \pm \lr{\nabla}_m u = F$,  then $P_\lambda H_N u \in V^2_{\pm,m }$ and we have the bound
     \begin{equation}\label{eqn - lem energy ineq for V - main bound}\| P_\lambda H_N u \|_{V^2_{\pm ,m}} \lesa \| P_\lambda H_N u(0) \|_{L^2} + \sup_{ \| P_\lambda H_N v \|_{V^2_{\pm ,m}} \lesa 1} \int_\RR \lr{ P_\lambda H_N v(t), F(t) }_{L^2_x} dt .\end{equation}
\end{lemma}
\begin{proof}
See \cite{Koch2015} or \cite[Proposition 2.10]{Hadac2009} for details on the duality. It is also possible to prove this directly as follows. Clearly it is enough to consider the case $u(0) = 0$, thus $u(t) = \int_0^t e^{ \mp i (t-s) \lr{\nabla}_m} F(s) ds$. Let $K>0$ and $(t_k) \in \mc{Z}$. A computation gives the identity
    $$ \Big( \sum_{|k|<K} \| e^{ \pm i t_k \lr{\nabla}_m} P_\lambda H_N u(t_k) - e^{ \pm i t_{k-1} \lr{\nabla}_m} P_\lambda H_N u(t_{k-1}) \|_{L^2_x}^2 \Big)^\frac{1}{2} = \int_\RR \lr{ P_\lambda H_N v(s), F(s) }_{L^2_x} ds $$
with
    $$ v(s) = A^{-1} \sum_{|k|<K} \ind_{[t_{k-1}, t_k)}(s) \Big( e^{ \mp i (s-t_k) \lr{\nabla}_m} u(t_k) - e^{\mp i (s-t_{k-1}) \lr{\nabla}_m} u(t_{k-1})\Big)$$
and
    $$ A = \Big( \sum_{|k|<K} \| e^{ \pm i t_k \lr{\nabla}_m} P_\lambda H_N u(t_k) - e^{ \pm i t_{k-1} \lr{\nabla}_m} P_\lambda H_N u(t_{k-1}) \|_{L^2_x}^2 \Big)^\frac{1}{2}.$$
 It is easy to check that $\| P_\lambda H_N v\|_{V^2_{\pm,m}} \lesa 1$. Thus, by taking the sup over $\|P_\lambda H_N v \|_{V^2_{\pm,m}}\lesa 1$, and then letting $K \rightarrow \infty$ we deduce the bound (\ref{eqn - lem energy ineq for V - main bound}). Since $u$ is also continuous, we obtain $u \in V^2_{\pm,m}$ as required.
\end{proof}

Note that the norm on $v$ can in fact be taken to be the stronger $U^2_{\pm,m}$ norm, but we do not require this improvement here.

The final result we require on the $V^2_{\pm,m}$ spaces, concerns the question of scattering.

\begin{lemma}\label{lem - scattering in V2}
Let $u \in V^2_{\pm,m}$. Then there exists $f \in L^2_x$ such that $\| u(t) - e^{ \mp i t\lr{\nabla}} f \|_{L^2_x} \rightarrow 0$ as $t \rightarrow \infty$.
\end{lemma}

Clearly, this result can be extended to elements of the space $F^{s, \sigma_M}_{\pm, m}$. In other words, if we construct a solution in $F^{s, \sigma_M}_{\pm, m}$, then we immediately deduce the solution must scatter to a linear solution as $t \rightarrow \pm \infty$.

\subsection{Proof of Theorem \ref{thm:dkg}}\label{subsect:proof-dkg}
We now come to the proof of Theorem \ref{thm:dkg}. In light of Lemma \ref{lem - scattering in V2}, it is enough to construct a solution $(\psi_\pm, \phi_+) \in F^{0, \sigma_M}_{\pm, M} \times F^{\frac{1}{2}, \sigma_M}_{+, 1}$ to the reduced system (\ref{eqn - DKG system reduced}). Note that we may always assume that $\psi_\pm = \Pi_\pm \psi_\pm$, provided that this is satisfied at $t=0$. Define the Duhamel integral
    $$\mc{I}^{\pm}_m[F] = \int_0^t e^{ \mp i (t-s) \lr{\nabla}_m} F(s) ds .$$
Note that $\mc{I}^\pm_m[F]$ solves the equation
    $$ (- i \p_t \pm \lr{\nabla}_m ) \mc{I}^\pm_m[F] = F $$
with vanishing data at $t=0$. Moreover, we can check that for every $1<p < \infty$ we have
	\begin{equation}\label{eqn - Cd mult and duhamel int}
		\| C^{\pm, m}_d \mc{I}^\pm_m[F] \|_{L^p_t L^2_x} \lesa d^{-1} \| C^{\pm, m}_d F \|_{L^p_t L^2_x}.
	\end{equation}
If we had the bounds
    \begin{equation}\label{eqn - bilinear est without freq localisation}
        \begin{split}
            \big\| \Pi_{\pm_1} \mc{I}^{\pm_1}_{M} \big[ \phi \gamma^0 \Pi_{\pm_2} \varphi \big]\big\|_{F^{0, \sigma_M}_{\pm_1, M}} &\lesa \| \phi\|_{F^{\frac{1}{2}, \sigma_M}_{+, 1}} \| \varphi \|_{F^{0, \sigma_M}_{M, \pm_2}} \\
            \big\| \lr{\nabla}^{-1} \mc{I}^{+}_{1} \big[ (\Pi_{\pm_1} \psi)^\dagger \gamma^0 \Pi_{\pm_2} \varphi\big] \big\|_{F^{\frac{1}{2}, \sigma_M}_{+, 1}} &\lesa  \| \psi \|_{F^{0, \sigma_M}_{M, \pm_1}} \| \varphi \|_{F^{0, \sigma_M}_{M, \pm_2}}
        \end{split}
    \end{equation}
then a standard fixed point argument in $F^{0, \sigma_M}_{\pm, M} \times F^{\frac{1}{2}, \sigma_M}_{+, 1}$ would give the required solution to (\ref{eqn - DKG system reduced}), provided of course that the data $(f_\pm, g_+)$ satisfied
    $$ \| \lr{\Omega}^{\sigma_M} f_\pm \|_{L^2} + \| \lr{\Omega}^{\sigma_M} g_+ \|_{H^\frac{1}{2}} \ll 1.$$
Let
    $$ \phi_{\mu, N} = P_\mu H_N \phi, \qquad \psi_{\lambda_1, N_1} = P_{\lambda_1} H_{N_1}, \qquad \varphi_{\lambda_2, N_2} = P_{\lambda_2} H_{N_2} \varphi.$$
We have the following frequency localised estimates.

\begin{theorem}\label{thm - bilinear F freq loc endpoint}
Fix $M>0$. Then there exists $\epsilon>0$ such that
  	\begin{equation}\label{eqn - thm bilinear F freq loc endpoint - main ineq I} \begin{split}
  		\big\| \Pi_{\pm_1} \mc{I}^{\pm_1}_M\big[ \phi_{\mu, N} \gamma^0 \Pi_{\pm_2} &\varphi_{\lambda_2, N_2}\big] \big\|_{F^{\pm_1, M}_{\lambda_1, N_1}}\\
  		&\lesa \mu^\frac{1}{2} (\min\{ N, N_2\})^{\sigma_M}   \Big( \frac{\min\{ \mu, \lambda_1, \lambda_2\}}{\max\{ \mu, \lambda_1, \lambda_2\}}\Big)^\epsilon  \|  \phi\|_{F^{+, 1}_{\mu, N}} \| \varphi \|_{F^{\pm_2, M}_{\lambda_2, N_2}}
  \end{split}
  	\end{equation}
and
  \begin{equation}\label{eqn - thm bilinear F freq loc endpoint - main ineq II} \begin{split} \big\| \mc{I}^{+}_1\big[  (\Pi_{\pm_1} \psi_{\lambda_1, N_1})^\dagger \gamma^0& \Pi_{\pm_2} \varphi_{\lambda_2, N_2}\big] \big\|_{F^{+,1}_{\mu, N}} \\
    &\lesa \mu^{\frac{1}{2}} ( \min\{ N_1, N_2\})^{\sigma_M} \Big( \frac{\min\{ \mu, \lambda_1, \lambda_2\}}{\max\{ \mu, \lambda_1, \lambda_2\}}\Big)^\epsilon \| \psi \|_{F^{\pm_1, M}_{\lambda_1, N_1}} \| \varphi \|_{F^{\pm_2, M}_{\lambda_2, N_2}}.
  \end{split}
  \end{equation}
\end{theorem}

\begin{remark}\label{rmk:trivial-int}
The proof of Theorem \ref{thm - bilinear F freq loc endpoint} in the resonant regime $0<M<\frac{1}{2}$ relies on the small scale $V^2$ estimates in Corollary \ref{cor - small scale bilinear estimate}. However, it is possible to prove a weaker version of Theorem \ref{thm - bilinear F freq loc endpoint}, with $\sigma_M$ replaced with some larger $\sigma$, provided only that a \emph{robust} version of the \emph{homogeneous} bilinear restriction estimate (\ref{eqn:homogeneous bilinear est}) holds. More precisely, by
following the proof of Corollary \ref{cor - small scale bilinear estimate}, and then interpolating with the K-G Strichartz estimates as in Remarks \ref{rem:inter} and \ref{rem:interpolation with K-G}, it is possible to show that \eqref{eqn:homogeneous bilinear est} implies the $V^2$ bound
	$$ \| u v \|_{L^a_t L^b_x(\RR^{1+3})} \lesa \lambda^{1 + \frac{1}{a} - \frac{1}{b}} \| u \|_{V^2_{\pm_1,m_1}} \| v \|_{V^2_{\pm_2,m_2}}$$
in the range $\frac{1}{a} + \frac{2}{b} < 2$, $\frac{1}{a} + \frac{6}{5b} < \frac{7}{5}$ where $u$ and $v$ have Fourier support in $1-$separated angular wedges of size $1 \times 1 \times \lambda$ at distance $\lambda$ from the origin. The case $a = 2-$ and $b=\frac{4}{3} +$ can be used together with the $L^{2+}_t L^{4-}_x$ angular Strichartz bound from  \cite[Theorem 1.1]{Cho2013} instead of the argument used in the high-high case in the proof of Theorem \ref{thm - trilinear freq loc integral} below. However, the estimate obtained is weaker than the one in Theorem \ref{thm - bilinear F freq loc endpoint}. Moreover, it still requires a robust version of the homogeneous bilinear estimate (\ref{eqn:homogeneous bilinear est}) for which we can track the dependence of the constant on the phases $\Phi_j$ due to the lack of homogeneity of the Klein-Gordon phase. Irrespective of fact the Theorem \ref{thm - bilinear Lp estimate} applies to $V^2$-functions, a key advantage of our formulation of Theorem \ref{thm - bilinear Lp estimate}, in comparison to \cite{Bejenaru2016,Lee2010}, is that it allows us to read off the above mentioned dependence.
\end{remark}

The standard Littlewood-Paley trichotomy implies that the lefthand sides of (\ref{eqn - thm bilinear F freq loc endpoint - main ineq I}) and (\ref{eqn - thm bilinear F freq loc endpoint - main ineq II}) are zero unless
    \begin{equation}\label{eqn - littlewood-paley cond} \max\{ \mu, \lambda_1, \lambda_2\} \approx \text{med}\{ \mu, \lambda_1, \lambda_2\} \gtrsim \min\{\mu, \lambda_1, \lambda_2\}\end{equation}
and
    $$ \max\{ N, N_1, N_2\} \approx \text{med}\{N, N_1, N_2\} \gtrsim \min\{N, N_1, N_2 \}$$
It is now easy to check that the bilinear estimates (\ref{eqn - bilinear est without freq localisation}), follow from Theorem \ref{thm - bilinear F freq loc endpoint}. Consequently, we have reduced the proof of Theorem \ref{thm:dkg}, to proving the frequency localised bilinear estimates in Theorem \ref{thm - bilinear F freq loc endpoint}. As the proof of Theorem \ref{thm - bilinear F freq loc endpoint} requires a number of preliminary results, we postpone the proof till to Subsection \ref{subsect:proof2} below.

\section{Linear and Multilinear Estimates}\label{sect:mult}

In this section our goal is give the proof of Theorem \ref{thm - bilinear F freq loc endpoint}. To this end, we first provide some linear estimates and adapt them to our functional setup, prove an auxiliary trilinear estimate in $V^2$, and eventually give the proof of the crucial Theorem \ref{thm - bilinear F freq loc endpoint} in Subsection \ref{subsect:proof2}.

\subsection{Auxiliary Estimates}

As is well-known, see e.g.\ \cite{D'Ancona2007b}, the system \ref{eqn - DKG system reduced} exhibits null structure.
To exploit the null structure of the product $\psi^\dagger \gamma^0 \psi$, we start by noting that for any $x, y \in \RR^3$, we have the identity
   \begin{align*}
      \big[\Pi_{\pm_1} f\big]^\dagger \gamma^0 \Pi_{\pm_2} g = \big[(\Pi_{\pm_1} &- \Pi_{\pm_1}(x))  f\big]^\dagger \gamma^0 \Pi_{\pm_2} g\\
        &+ \big[\Pi_{\pm_1}(x) f\big]^\dagger \gamma^0  (\Pi_{\pm_2} - \Pi_{\pm_2}(y) ) g + f^\dagger \Pi_{\pm_1}(x) \gamma^0  \Pi_{\pm_2}(y)  g
   \end{align*}
 This is then exploited by using the null form type bound
        \begin{equation}\label{eqn:nullsymbolbound}|\Pi_{\pm_1}(x) \gamma^0 \Pi_{\pm_2}(y)| \lesa \theta(\pm_1 x, \pm_2 y) + \frac{\big| \pm_1 |x| \pm_2 |y| \big|}{\lr{x} \lr{y}},\end{equation}
 which follows from (\ref{eqn:KG transverse}) by observing that
 	\begin{align*}
 	  \Pi_{\pm_1}(x) \gamma^0 \Pi_{\pm_2}(y) &= \Pi_{\pm_1}(x) \Big( \Pi_{\pm_1}(x) \gamma^0 - \gamma^0 \Pi_{\mp_2}(y) \Big) \Pi_{\pm_2}(y)  \\
 	  &= \Pi_{\pm_1}(x) \bigg( \Big(  \frac{\pm_2 \eta_j}{\lr{\eta}_M}-   \frac{\pm_1 \xi_j}{\lr{\xi}_M}\Big) \gamma^j + \Big( \frac{\pm_1 M}{\lr{\xi}_M} + \frac{\pm_2 M}{\lr{\eta}_M}\Big) I  \bigg) \Pi_{\pm_2}(y),
 	\end{align*}
 together with the following lemma, see \cite[Lemma 3.3]{Bejenaru2016} for a similar statement to Part (i).
  \begin{lemma}\label{lem - null form bound} Let $1<r<\infty$.
\begin{enumerate}
\item
If $\lambda \g 1$, $\alpha \gtrsim \lambda^{-1}$, $\kappa \in \mc{C}_\alpha$, then
    \begin{align*}\big\| \big( \Pi_{\pm_1} - \Pi_{\pm_1}(\lambda \omega(\kappa)) \big) R_\kappa P_\lambda f \big\|_{L^r_x} & \lesa \alpha \| R_\kappa P_\lambda u\|_{L^r_x}.
  \end{align*}
\item If $\lambda\g 1$, $0<\alpha\lesssim  \lambda^{-1}$, $\kappa \in \mc{C}_\alpha$, $q\in Q_{\lambda^2\alpha}$ with center $\xi_0$, then
    \begin{align*}\big\| \big( \Pi_{\pm_1} - \Pi_{\pm_1}(\xi_0) \big) R_\kappa P_q P_\lambda f \big\|_{L^r_x} & \lesa \alpha \| R_\kappa P_q P_\lambda u\|_{L^r_x}.
  \end{align*}
\end{enumerate}
\end{lemma}
	
\begin{proof}
Concerning Part (i), see \cite[Proof of Lemma 3.3]{Bejenaru2016}. Concerning Part (ii), we may assume $|\xi_0|\approx \lambda$ and, due to boundedness, we may replace the symbol of $R_\kappa P_q P_\lambda$ by a smooth cutoff $\chi_E$ to the parallelepiped $E$ with center $\xi_0$ of side lenghts $\alpha\mu^2 \times \alpha\mu \times \alpha\mu $ with long side pointing in the direction $\xi_0$. After rotating $\xi_0$ to $\xi_0=|\xi_0|(1,0,0)$, the operator has the symbol
\[m(\xi)=\Big(\pm B^j\Big(\frac{\xi_j}{\lr{\xi}_M}-\frac{\xi_{0,j}}{\lr{\xi_0}_M}\Big)\pm\tfrac{1}{2}\gamma^0\Big(\frac{1}{\lr{\xi}_M}-\frac{1}{\lr{\xi_0}_M}\Big)\Big)\chi_E(\xi),
\]
for certain $B^1,B^2,B^3\in \CC^{4\times4}$. It suffices to prove the kernel bound
\begin{equation}\label{eq:ker-bd}
|(\mathcal{F}_x^{-1}m)(x)|\lesa \alpha^{4}\lambda^4(1+\alpha\lambda^2|x_1|+\alpha\lambda|x'|)^{-4}, \quad x=(x_1,x'),
\end{equation}
as it implies $\|\mathcal{F}_x^{-1}m\|_{L^1(\RR^3)}\lesa \alpha$.
In the support of $\chi_E$ we obtain, from (\ref{eqn:KG transverse}) and a simple computation,
\[
|m(\xi)|\lesa \lambda^{-3}||\xi|-|\xi_0||+\theta(\xi,\xi_0)+\lambda^{-2}||\xi|-|\xi_0||\lesa \alpha.
\]
From the localisation of $\chi_E$, where $|\partial^\ell_{\xi_1} \frac{\xi_j}{\lr{\xi}_M}|\lesa \lambda^{-\ell-1}$, and the Leibniz rule, we conclude for $\ell>0$
\[
|\partial^\ell_{\xi_1} m(\xi)|\lesa \alpha(\alpha\lambda^2)^{-\ell}+\sum_{0< \ell_1\les \ell} \lambda^{-\ell_1-1}(\alpha\lambda^2)^{\ell_1-\ell}\lesa \alpha(\alpha\lambda^2)^{-\ell}.
\]
Integration by parts now implies \eqref{eq:ker-bd} if $\alpha\lambda^2
|x_1|\geq \alpha\lambda|x'|$. For $k=2,3$, we have
$|\partial^\ell_{\xi_k} \frac{\xi_j}{\lr{\xi}_M}|\lesa
\lambda^{-\ell}$  within the support of $\chi_E$, hence
we conclude for $\ell>0$
\[
|\partial^\ell_{\xi_k} m(\xi)|\lesa \alpha(\alpha\lambda)^{-\ell}+\sum_{0< \ell_1\les \ell} \lambda^{-\ell_1}(\alpha\lambda)^{\ell_1-\ell}\lesa \alpha(\alpha\lambda)^{-\ell}.
\]
Integration by parts now implies \eqref{eq:ker-bd} in the
region where $\alpha\lambda^2 |x_1|\les \alpha\lambda|x_k|$.
\end{proof}

The proof of Theorem \ref{thm - bilinear F freq loc endpoint} requires a number of standard linear estimates for homogeneous solutions to the Klein-Gordon equation. We start be recalling the Strichartz estimates for the wave and Klein-Gordon equation.

\begin{lemma}[Wave Strichartz]\label{lem - wave strichartz}
Let $m\g0$ and $2<q \les \infty$. If $0< \mu \les \lambda$, $N \g1$, and $\frac{1}{r}  = \frac{1}{2} - \frac{1}{q}$ then for every $q \in Q_\mu$ we have
            $$ \| e^{\mp i t \lr{\nabla}_m} P_q P_\lambda f\|_{L^q_t L^r_x} \lesa \mu^{\frac{1}{2} - \frac{1}{r}} \lambda^{\frac{1}{2} - \frac{1}{r}} \| P_q P_\lambda f\|_{L^2_x}.$$
Moreover, by spending additional angular regularity we have
        $$ \| e^{\mp i t \lr{\nabla}_m}  P_\lambda H_N f \|_{L^q_t L^4_x} \lesa \lambda^{\frac{3}{4}- \frac{1}{q}} N \| P_\lambda H_N f \|_{L^2_x}.$$
\end{lemma}
\begin{proof} The proof of the first estimate can be found in \cite[Lemma 3.1]{Bejenaru2015}. The second follows by simple modification of the argument in the appendix to \cite{Sterbenz2007}. More precisely, after interpolating with the $L^\infty_t L^2_x$ estimate, we need to show that
        $$ \| e^{ \mp i t \lr{\nabla}_m } H_N P_\lambda f \|_{L^2_t L^r_x} \lesa N \lambda^{3(\frac{1}{2} - \frac{1}{r}) - \frac{1}{2} }\| H_N \lambda f \|_{L^2_x}.$$
After rescaling, and following the argument on \cite[pp.\ 226--227]{Sterbenz2007}, it is enough to prove that for every $\epsilon>0$ we have the space-time Morawetz type bound
        \begin{equation}\label{eqn - morawetz bound} \| ( 1 + |x|)^{ -\frac{1}{2} - \epsilon }\nabla  u \|_{L^2_{t, x}} \lesa \big\| \big(\p_t u(0), \nabla u(0) \big) \big\|_{L^2_x}\end{equation}
for functions $u$ with $\Box u + m u = 0$, and the constant in (\ref{eqn - morawetz bound}) is independent of $m$. However the proof of (\ref{eqn - morawetz bound}) follows the same argument as the wave case in \cite{Sterbenz2007}, the only change is to replace the wave energy-momentum tensor with the Klein-Gordon version
        $$ Q_{\alpha \beta} =  \tfrac{1}{2} \Big(\p_\alpha \phi
        \overline{\p_\beta \phi}  + \p_\beta \phi \overline{\p_\alpha
          \phi}  - g_{\alpha \beta} \big( \p^\gamma \phi
        \overline{\p_\gamma \phi} + m^2 |\phi|^2\big) \Big),$$
we omit the details.
\end{proof}
The amount of angular regularity required for the $L^{2+}_t L^4_x$ Strichartz estimate to hold, is much less than that stated in Lemma \ref{lem - wave strichartz}. In fact, in \cite{Sterbenz2007}, it is shown that the same estimate holds with $N^{\frac{1}{2}  +}$. However, as the sharp number of angular derivatives is not required in the arguments we use in the present paper, we have elected to simply state the result with a whole angular derivative. On the other hand, the number of angular derivatives required in the following Klein-Gordon regime, plays a crucial role.

\begin{lemma}[Klein-Gordon Strichartz]\label{lem - KG strichartz}
Let $m>0$ and $\frac{3}{10}<\frac{1}{r} < \frac{5}{14}$. Then for every $\epsilon>0$ we have
	$$ \| e^{ \mp i t \lr{\nabla}_m} P_\lambda H_N f  \|_{L^{r}_{t, x}} \lesa \lambda^{2-\frac{5}{r}} N^{7(\frac{1}{r} - \frac{3}{10})+\epsilon} \|P_\lambda H_N f\|_{L^2_x}.$$
\end{lemma}
\begin{proof}
  This is a special case of \cite[Theorem 1.1]{Cho2013}.
\end{proof}

\begin{remark}\label{rmk:ang-reg}
Without angular regularity, the optimal $L^r_{t, x}$ Strichartz estimate for the Klein-Gordon equation is $r=\frac{10}{3}$, see for instance \cite{Machihara2003}. However, in the resonant region, we are forced to take $r$ slightly below $3$, thus the additional angular regularity is essential to obtain the additional integrability in time. In other words, the angular regularity is used not just to obtain the scale invariant endpoint, but also plays a crucial role in controlling the resonant interaction. Note that the number of angular derivatives required in Lemma \ref{lem - KG strichartz} is not expected to be optimal, and any improvement in this direction has an impact on Theorem \ref{thm:dkg}.
\end{remark}

We have seen that the addition of angular regularity improves the range of available Strichartz estimates. An alternative way to exploit additional angular regularity is given by the following angular concentration type bound.

\begin{lemma}[{\cite[Lemma 5.2]{Sterbenz2007}}] \label{lem - angular concentration} Let $2\les p < \infty$, and $0 \les s <\frac{2}{p}$. If $\lambda, N \g 1$,  $\alpha \gtrsim \lambda^{-1}$, and $\kappa \in \mc{C}_\alpha$ we have
        $$  \| R_\kappa P_\lambda H_N f \|_{L^p_x(\RR^3)} \lesa  \alpha^s N^s \| P_\lambda H_N f \|_{L^p_x(\RR^3)}.$$
\end{lemma}

Finally, we need to estimate various square sums of norms. As we work in $V^2$, this causes a  slight loss in certain estimates. However, as we have some angular derivatives to work with, this loss can always be absorbed elsewhere.

\begin{lemma}\label{lem - square sum loss} Let $(P_j)_{j \in \mc{J}}$ and $(\mc{M}_j)_{j \in \mc{J}}$ be a collection of spatial Fourier multipliers. Suppose that the symbols of $P_j$ have finite overlap, and
		$$ \| \mc{M}_j P_j f \|_{L^2_x} \lesa \delta \| P_j f \|_{L^2_x}$$
for some $\delta >0$.
	\begin{enumerate}
		\item Let $q>2$, $r \g 2$. Suppose that there exists $A>0$ such that for every $j$ we have the bound
				$$ \| e^{\mp i t \lr{\nabla}_m} P_j f \|_{L^q_t L^r_x} \les A \| P_j f \|_{L^2_x}.$$
			Then for every $\epsilon>0$ we have
				$$ \Big( \sum_{j \in \mc{J}} \| \mc{M}_j P_j v \|_{L^q_t L^r_x}^2 \Big)^\frac{1}{2}
								\lesa \delta \big( \# \mc{J}\big)^\epsilon A \| v \|_{V^2_{\pm,m}}.$$
								
		\item Fix $p_0>1$. Suppose that there exists $A>0$ such that  $\| P_j f\|_{L^\infty_x} \lesa A \| f \|_{L^2_x}$. Moreover, suppose that for every $p>p_0$ there exists $B_p>0$, and for any $j \in \mc{J}$ there exists $\mc{K}_j \subset \mc{J}$ with $\# \mc{K}_j \lesa 1$, such that for every $k \in \mc{K}_j$
				$$   \| P_j u P_k v \|_{L^p_{t, x}} \lesa B_p \| P_j u \|_{U^2_{\pm_1,m_1}} \| P_k v \|_{U^2_{\pm_2,m_2}}.$$
			Then for every $q > p_0$ and $\frac{p_0}{q} < \theta < 1$ we have
				$$ \sum_{\substack{ j \in \mc{J}, k \in \mc{K}_j }} \| P_j u \mc{M}_k P_k v \|_{L^q_{t, x}} \lesa \delta  \big(\#\mc{J}\big)^{1-\theta} A^{1-\theta} B_{\theta q}^\theta \| u \|_{V^2_{\pm_1,m_1}} \| v \|_{V^2_{\pm_2,m_2}}.$$
	\end{enumerate}
\end{lemma}
\begin{proof}
We start with the proof of (i). Let $2 \les p \les q$ and suppose that $\phi = \sum_{I \in \mc{I}} \ind_{I}(t) e^{ \mp i t \lr{\nabla}_m} f_I$ is  a $U^p$ atom, thus $\sum_{I} \|f_I\|_{L^2_x}^p \les 1$. The assumed linear estimate, together with the finite overlap of the Fourier multipliers $P_j$ implies that
    \begin{align*}
	\Big(\sum_{j \in \mc{J}} \|\mc{M}_j P_j \phi \|_{L^q_t L^r_x}^p \Big)^\frac{1}{p} &\les \Big( \sum_{I \in \mc{I}} \sum_{j \in \mc{J}} \|e^{\mp i t \lr{\nabla}_m}\mc{M}_j P_j f_I \|_{L^q_t L^r_x}^p \Big)^\frac{1}{p}\\
&\les	A \Big( \sum_{j \in \mc{J}} \sum_{ I \in \mc{I}} \| \mc{M}_j P_j f_I \|_{L^2_x}^p \Big)^\frac{1}{p}
						\les \delta A  \Big( \sum_{ I \in \mc{I}} \Big( \sum_{j \in \mc{J}}  \| P_j f_I \|_{L^2_x}^2 \Big)^\frac{p}{2} \Big)^\frac{1}{p}\lesa \delta A.
	\end{align*}
Consequently the atomic definition of $U^p_{\pm,m}$ then implies that for any $2\les p\les q$
	\begin{equation}\label{eqn - lem square sum loss - linear Up bound} \Big( \sum_{j \in \mc{J}} \| \mc{M}_j P_j u \|_{L^q_t L^r_x}^p \Big)^\frac{1}{p}  \lesa A \delta \| u\|_{U^p_{\pm,m}}.\end{equation}
 Let $v \in V^2_{\pm,m}$. There exists a decomposition $v = \sum_{\ell\in\NN} v_\ell$ such that for every $p \g 2$ we have $\| v_\ell \|_{U^p_{\pm,m}} \lesa 2^{ \ell ( \frac{2}{p} - 1)} \|v\|_{V^2_{\pm,m}} $, see e.g.\ \cite[Lemma 6.4]{Koch2005} or \cite[Proposition 2.5 and Proposition 2.20]{Hadac2009}. An application of H\"older's inequality, together with (\ref{eqn - lem square sum loss - linear Up bound}) gives for any $2<p \les q$
	\begin{align*}
		\Big( \sum_{j \in \mc{J}} \| \mc{M}_j P_j v \|_{L^q_t L^r_x}^2 \Big)^\frac{1}{2} &\lesa (\# \mc{J})^{\frac{1}{2} - \frac{1}{p}}  \sum_{\ell\in \NN}  \Big( \sum_{j \in \mc{J}} \| \mc{M}_j P_j v_\ell \|_{L^q_t L^r_x}^p \Big)^\frac{1}{p}  \\
		&\lesa \delta A (\# \mc{J})^{\frac{1}{2} - \frac{1}{p}}   \sum_{\ell \in \NN} \| v_\ell \|_{U^p_{\pm,m}} \\
		&\lesa \delta A (\# \mc{J})^{\frac{1}{2} - \frac{1}{p}}  \| v \|_{V^2_{\pm,m}} \sum_{\ell \in \NN} 2^{ \ell (\frac{2}{p} -1)}  \\
		&\lesa \delta A (\# \mc{J})^{\frac{1}{2} - \frac{1}{p}}  \| v \|_{V^2_{\pm,m}}.
	\end{align*}
Thus (i) follows by taking $p$ sufficiently close to 2.

We now turn to the proof of (ii). As in the proof of $(i)$, we decompose $u = \sum_{\ell \in \NN} u_\ell$ and $v = \sum_{\ell \in \NN} v_\ell$ with $\| u_\ell\|_{U^r_{\pm_1,m_1}} \lesa 2^{\ell ( \frac{2}{r} - 1)}$ and $\| v_\ell \|_{U^r_{\pm_2,m_2}} \lesa 2^{ \ell ( \frac{2}{r} -1)}$ for every $r\g 2$. Let $q>p_0$ and $\frac{p_0}{q} <\theta<1$. Then the convexity of the $L^q$ norms together with H\"{o}lder's inequality, our assumed bilinear estimate,  and the $U^2$ summation argument used in (i) implies that
	\begin{align*}
		\sum_{\substack{ j \in \mc{J}, k \in \mc{K}_j }} &\| P_j u \mc{M}_k P_k v \|_{L^q_{t, x}} \\
					&\lesa  (\# \mc{J})^{1-\theta} \sum_{\ell, \ell' \in \NN} \Big(\sum_{\substack{ j \in \mc{J}, k \in \mc{K}_j }} \| P_j u \mc{M}_k P_k v \|_{L^{\theta q}_{t, x}}\Big)^\theta \Big( \sup_{j, k \in \mc{J}} \| P_j u_\ell \mc{M}_k P_k v_{\ell'} \|_{L^\infty_{t, x}}\Big)^{1-\theta} \\
					&\lesa \delta  (\# \mc{J})^{1-\theta}  A^{1-\theta} B_{\theta q}^{\theta}  \sum_{\ell, \ell' \in \NN}  \big( \| u_\ell \|_{U^2_{\pm_1,m_1}} \| v_\ell \|_{U^2_{\pm_2 ,m_2}} \big)^\theta \big( \| u_\ell \|_{U^\infty_{\pm_1,m_1}} \| v_{\ell'} \|_{U^\infty_{\pm_2, m_2}}\big)^{1-\theta} \\
					&\lesa \delta  (\# \mc{J})^{1-\theta}  A^{1-\theta} B_{\theta q}^{\theta}   \| u \|_{V^2_{\pm_1,m_1} } \| v \|_{V^2_{\pm_2 ,m_2}} \sum_{\ell, \ell' \in \NN} 2^{  - \ell ( 1- \theta)} 2^{  - \ell' ( 1- \theta)} \\
			&\lesa \delta  (\# \mc{J})^{1-\theta}  A^{1-\theta} B_{\theta q}^{\theta}  \| u \|_{V^2_{\pm_1,m_1} } \| v \|_{V^2_{\pm_2 ,m_2}}.
	\end{align*}
Therefore (ii) follows.
\end{proof}

Clearly the previous lemma allows us to extend Corollary \ref{cor - small scale bilinear estimate}, and the linear estimates discussed above, to frequency localised functions in $V^2_{\pm,m}$. For instance, for any $1\les \mu \lesa \lambda$, $\alpha \gtrsim \lambda^{-1}$, and $\epsilon>0$, $q>2$, we have by Lemma \ref{lem - wave strichartz}
\begin{align}
\label{eqn:L4sqrsum}
\Big( \sum_{q\in Q_\mu} \sum_{\kappa \in \mc{C}_{\alpha}} \big\| R_\kappa  P_q u_{\lambda, N} \big\|_{L^4_{t, x}}^2 \Big)^\frac{1}{2} \lesa{} & \alpha^{-\epsilon} \Big( \frac{\mu}{\lambda} \Big)^{\frac{1}{4}-\epsilon} \lambda^{\frac{1}{2}} \| u_{\lambda, N} \|_{V^2_{\pm,m}},\\
\Big( \sum_{\kappa \in \mc{C}_{\alpha}} \big\| R_\kappa  u_{\lambda, N} \big\|_{L^q_t L^4_x}^2 \Big)^\frac{1}{2} \lesa{} &  \alpha^{-\epsilon } \lambda^{\frac{3}{4} - \frac{1}{q}} N \| u_{\lambda, N} \|_{V^2_{\pm,m}} \label{eqn:LqL4sqrsum}
\end{align}
where we use the shorthand $u_{\lambda, N} = P_\lambda P_N u $. Similarly, an application of Corollary \ref{cor - small scale bilinear estimate}, Lemma \ref{lem - null form bound}, and (ii) in Lemma \ref{lem - square sum loss} gives for every $q>\frac{3}{2}$ and $\epsilon>0$
    \begin{equation}\label{eqn:sqrsum bilinear}\begin{split}
        \Big( \sum_{\substack{\kappa,  \kappa'' \in \mc{C}_{\mu^{-1}} }} \sum_{\substack{q, q'' \in Q_\mu \\ |q - q''| \approx \mu \text{ or } |\kappa - \kappa''| \approx \mu^{-1}}} \big\| R_{\kappa''} P_{q''} \phi_{\mu, N} \big[&\big( \Pi_+  - \Pi_{+}(\mu \omega(\kappa))\big) R_\kappa P_q \psi_{\mu, N_1}\big]^\dagger \big\|_{L^q_{t, x}(\RR^{1+3})}^2 \Big)^\frac{1}{2} \\
                &\lesa \mu^{\epsilon} \| \phi_{\mu, N} \|_{V^2_{+,1}} \| \psi_{\mu, N_1} \|_{V^2_{+,M}}
        \end{split}
    \end{equation}
where $\omega(\kappa)$ denotes the centre of the cap $\kappa \in \mc{C}_{\mu^{-1}}$. This bilinear bound plays a key role in controlling the solution to the DKG system in the resonant region.

\subsection{General Resonance Identity}


 After an application of Lemma \ref{lem - energy ineq for F norms}, proving the bilinear estimates in Theorem \ref{thm - bilinear F freq loc endpoint} for the $V^2$ component of the norm, reduces to estimating trilinear expressions of the form
    \begin{equation}\label{eqn - resonance func motivation} \int_{\RR^{1+3}} \phi \psi^\dagger \gamma^0 \varphi dx dt. \end{equation}
Suppose that $\phi$, $\psi$, and $\varphi$ have small modulation, thus $\supp \widetilde{\phi} \subset \{ |\tau + \lr{\xi}| \les d \}$, $\supp \widetilde{\psi} \subset \{ |\tau \pm_1 \lr{\xi}_M| \les d \}$, and $\supp \widetilde{\varphi} \subset \{ |\tau \pm_2 \lr{\xi}_M| \les d \}$ for some $d \in 2^\ZZ$. If $\xi \in \supp \widehat{\psi}$ and $\eta \in \supp \widehat{\varphi}$, then it is easy to check that the integral (\ref{eqn - resonance func motivation}) vanishes unless
        $$ \big| \lr{\xi - \eta} \mp_1 \lr{\xi}_M \pm_2 \lr{\eta}_M \big| \lesa d.$$
To exploit this, we define the modulation function
    $$ \mathfrak{M}_{\pm_1, \pm_2}(\xi, \eta)  = \big| \lr{\xi - \eta}  \mp_1 \lr{\xi}_M \pm_2 \lr{\eta}_M \big|.$$
Clearly we have the symmetry properties $\mathfrak{M}_{+, +}(\xi, \eta) = \mathfrak{M}_{-, -}(\eta, \xi)$ and $\mathfrak{M}_{\pm, \mp}(\xi, \eta) = \mathfrak{M}_{\pm, \mp}(\eta, \xi)$. The proof of our global existence results requires a careful analysis of the zero sets of $\mathfrak{M}_{\pm_1, \pm_2}$, the key tool is the following.
\begin{lemma}\label{lem - modulation bound} Let $M>0$.
\begin{enumerate}
  \item (Nonresonant interactions). We have
    $$ \mathfrak{M}_{-, +}(\xi, \eta) \gtrsim \lr{\xi} + \lr{\eta}, \qquad \mathfrak{M}_{\pm, \pm}(\xi, \eta) \gtrsim \frac{1}{\lr{\xi - \eta}} \Bigg(   \frac{ (|\xi| - |\eta|)^2}{\lr{\xi} \lr{\eta}} + |\xi| |\eta| \theta^2(\xi, \eta) + 1 \Bigg),$$
    and
        $$ \mathfrak{M}_{-, -}(\xi, \eta) \gtrsim \frac{|\xi - \eta| |\xi|}{\lr{\xi} + \lr{\eta}} \theta^2(\xi - \eta,  -\xi), \qquad \mathfrak{M}_{+, +}(\xi, \eta) \gtrsim \frac{|\xi - \eta| |\eta|}{\lr{\xi} + \lr{\eta}} \theta^2(\xi - \eta, \eta).$$

  \item (Resonant interactions). We have
         $$\mathfrak{M}_{+, -}(\xi, \eta) \approx  \frac{1}{\lr{\xi} + \lr{\eta}} \bigg| M^2 \frac{ (|\xi| - |\eta|)^2}{\lr{\xi}_{M} \lr{\eta}_{M} + |\xi| |\eta| + M^2}  + |\xi| |\eta| + \xi \cdot \eta + \frac{ 4 M^2 - 1}{2} \bigg|$$
 and
 	$$\mathfrak{M}_{+, -}(\xi, \eta) \gtrsim  \frac{1}{\lr{\eta}} \bigg| \frac{ ( |\xi| - M |\xi - \eta|)^2}{\lr{\xi}_{M} \lr{\xi - \eta} + |\xi | |\xi - \eta| + M}  + |\xi| |\xi - \eta| - \xi \cdot (\xi - \eta)  + \frac{ 2M - 1 }{2} \bigg|.$$
\end{enumerate}
\end{lemma}
\begin{proof}
We begin by noting that, if we let $m_1, m_2, m_3 \g 0$, then for any $x, y \in \RR^n$ we have the identity
    \begin{align}
      \big| \lr{x-y}_{m_3}^2 - &(\lr{x}_{m_1} \pm \lr{y}_{m_2})^2 \big| \notag \\
        &=  \big|  \mp 2 \lr{x}_{m_1} \lr{y}_{m_2} -  2 x\cdot y  + (m_3^2 - m_1^2 - m_2^2 ) \big|         \notag \\
        &=  \big| 2\big( \lr{x}_{m_1} \lr{y}_{m_2} - (|x| |y| + m_1 m_2) \big)+ 2 \big(|x| |y| \pm x\cdot y\big)  \pm \big( (m_1 \pm m_2)^2 - m_3^2\big)  \big| \notag \\
        &= 2 \Big| \frac{ ( m_1 |y| - m_2 |x|)^2}{\lr{x}_{m_1} \lr{y}_{m_2} + |x| |y| + m_1 m_2}  + |x| |y| \pm x \cdot y  \pm \frac{ (m_1 \pm m_2)^2 - m_3^2}{2} \Big|.
        \label{eqn - lem modulation bound - general modulation identity}
    \end{align}
We now turn to (i). The bound for $\mathfrak{M}_{-, +}$ is clear. On the other hand, by taking $x=\xi$, $y=\eta$, $m_1=m_2=M$, $m_3 = 1$ in (\ref{eqn - lem modulation bound - general modulation identity}), we have
    \begin{align*} \mathfrak{M}_{\pm, \pm}(\xi, \eta)\g \big| \lr{\xi - \eta} - | \lr{\xi}_M - \lr{\eta}_M | \big|
        &\approx \frac{1}{\lr{\xi- \eta}} \big| \lr{\xi - \eta}^2 - \big( \lr{\xi}_M - \lr{\eta}_M\big)^2 \big| \\
        &\approx \frac{1}{\lr{\xi - \eta}} \Bigg(  \frac{(|\xi| - |\eta|)^2}{\lr{\xi} \lr{\eta}} + |\xi| |\eta| \theta^2(\xi, \eta) + 1 \Bigg).
    \end{align*}
Similarly, taking $x = \xi - \eta$ and $y = \xi$,  gives
    \begin{align*}
      \mathfrak{M}_{-, -}(\xi, \eta)= \frac{ \big| \lr{\eta}_M^2 - ( \lr{\xi - \eta} + \lr{\xi}_M)^2 \big|}{\lr{\eta}_M +  \lr{\xi - \eta} + \lr{\xi}_M} &\gtrsim  \frac{|\xi-\eta| |\xi| }{\lr{\xi} + \lr{\eta}} \theta^2(\xi-\eta, -\xi).
        \end{align*}
Using the symmetry $\mathfrak{M}_{-, -}(\xi, \eta) = \mathfrak{M}_{+, +}(\eta, \xi)$ gives the remaining bound in (i). To prove (ii), we note that another application of (\ref{eqn - lem modulation bound - general modulation identity}) gives
    \begin{align*}
      \mathfrak{M}_{+, -}(\xi, \eta) &\approx \frac{1}{\lr{\xi} + \lr{\eta}} \big| \lr{\xi - \eta}^2 - \big( \lr{\xi}_M + \lr{\eta}_M \big)^2 \big| \\
        &\approx  \frac{1}{\lr{\xi} + \lr{\eta}} \Big| M^2 \frac{ (|\xi| - |\eta|)^2}{\lr{\xi}_{M} \lr{\eta}_{M} + |\xi| |\eta| + M^2}  + |\xi| |\eta| + \xi \cdot \eta + \frac{ 4 M^2 - 1}{2} \Big|
    \end{align*}
from which the first inequality in (ii). The second inequality in (ii) follows from a similar application of (\ref{eqn - lem modulation bound - general modulation identity}).
\end{proof}

\subsection{The Trilinear Estimates}\label{subsect:proof1}
Suppose we would like to bound an expression of the form $P_{\lambda} H_N \mc{I}^\pm_m[F]$ in $V^2_{\pm,m}$.  An application of the energy inequality, Lemma \ref{lem - energy ineq for F norms}, implies that
we have
	$$ \| P_\lambda H_N \mc{I}^\pm_m[F] \|_{V^2_{\pm,m}} \lesa \sup_{ \| P_\lambda H_N u \|_{V^2_{\pm,m}} \lesa 1 } \Big| \int_{\RR^{1+3}} (P_\lambda H_N u)^\dagger  F dx dt \Big|. $$
Thus to bound the $V^2$ component of $\| \mc{I}^\pm_m[F] \|_{F^{\pm, m}_{\lambda, N}}$, it is enough to control the integral $\int_{\RR^{1+3}} (P_\lambda H_N u)^\dagger  F dx dt$. Consequently, to estimate the $V^2$ component of norms in Theorem \ref{thm - bilinear F freq loc endpoint}, the key step is to prove the following trilinear estimate. To simplify notation somewhat, we define $\mb{B}_\epsilon = ( \frac{\min\{ \mu, \lambda_1, \lambda_2\}}{\max\{ \mu, \lambda_1, \lambda_2\}})^\epsilon$ if $M \g \frac{1}{2}$, and if $0<M<\frac{1}{2}$ we let
    $$ \mb{B}_\epsilon = \begin{cases} \Big( \frac{\min\{ \mu, \lambda_1, \lambda_2\}}{\max\{ \mu, \lambda_1, \lambda_2\}}\Big)^\epsilon \qquad &\mu \ll \max\{\lambda_1, \lambda_2\} \text{ or } \mu \gg \min\{ \lambda_1, \lambda_2 \} \\
    1  + \mu^{-\frac{1}{6} + \sigma} (\min\{N, N_1, N_2\})^{\frac{7}{30}} & \mu \approx \lambda_1 \approx \lambda_2.
    \end{cases} $$

\begin{theorem}\label{thm - trilinear freq loc integral}
Let $M>0$. For every $\frac{\sigma}{100} <\delta \ll 1$ we have
	\begin{equation}\label{eqn - trilinear freq loc integral I} \begin{split} \Big|\int_{\RR^{3+1}} \phi_{\mu, N} \big( \Pi_{\pm_1} \psi_{\lambda_1, N_1}&\big)^\dagger \gamma^0 \Pi_{\pm_2} \varphi_{\lambda_2, N_2} dx dt\Big| \\
        &\lesa \mu^\frac{1}{2} (\min\{N, N_2\})^\delta \mb{B}_{\min\{\frac{\delta}{8}, \frac{1}{2a}-\frac{1}{4}\}} \| \phi\|_{F^{+, 1}_{\mu, N}} \| \psi_{\lambda_1, N_1} \|_{V^2_{\pm_1,M}} \| \varphi \|_{F^{\pm_2, M}_{\lambda_2, N_2}}\end{split}
	\end{equation}
and
	\begin{equation}\label{eqn - trilinear freq loc integral II} \begin{split} \Big|\int_{\RR^{3+1}} \phi_{\mu, N} \big( \Pi_{\pm_1} \psi_{\lambda_1, N_1}&\big)^\dagger \gamma^0 \Pi_{\pm_2} \varphi_{\lambda_2, N_2} dx dt\Big|\\
        &\lesa  \mu^\frac{1}{2}  (\min\{N_1, N_2\})^\delta \mb{B}_{\min\{\frac{\delta}{8}, \frac{1}{2a}-\frac{1}{4}\}} \| \phi_{\mu, N} \|_{V^2_{+,1}} \| \psi \|_{F^{\pm_1, M}_{\lambda_1, N_1}} \| \varphi \|_{F^{\pm_2, M}_{\lambda_2, N_2}}.\end{split}
	\end{equation}
In the region $\lambda_2 \gg \lambda_1$ we have the slightly stronger bound
	\begin{equation}\label{eqn - trilinear freq loc integral III} \begin{split} \Big|\int_{\RR^{3+1}} \phi_{\mu, N} &\big( \Pi_{\pm_1} \psi_{\lambda_1, N_1}\big)^\dagger \gamma^0 \Pi_{\pm_2} \varphi_{\lambda_2, N_2} dx dt\Big| \\
        &\lesa \mu^\frac{1}{2} (\min\{N, N_2\})^\delta \Big( \frac{\lambda_1}{\lambda_2}\Big)^{\frac{\delta}{8}} \| \phi_{\mu, N} \|_{V^2_{+,1}}  \| \psi_{\lambda_1, N_1} \|_{V^2_{\pm_1,M }} \| \varphi_{\lambda_2, N_2} \|_{V^2_{\pm_2,M}}.\end{split}
	\end{equation}
Similarly, when $\mu \ll \lambda_1$, we have
	\begin{equation}\label{eqn - trilinear freq loc integral IV} \begin{split} \Big|\int_{\RR^{3+1}} \phi_{\mu, N} &\big( \Pi_{\pm_1} \psi_{\lambda_1, N_1}\big)^\dagger \gamma^0 \Pi_{\pm_2} \varphi_{\lambda_2, N_2} dx dt\Big|\\
        &\lesa  \mu^\frac{1}{2}  (\min\{N_1, N_2\})^\delta \Big(\frac{\mu}{\lambda_1}\Big)^{\frac{\delta}{8}} \| \phi_{\mu, N} \|_{V^2_{+,1}}  \| \psi_{\lambda_1, N_1} \|_{V^2_{\pm_1,M }} \| \varphi_{\lambda_2, N_2} \|_{V^2_{\pm_2,M}}.\end{split}
	\end{equation}
\end{theorem}
\begin{proof} We begin by decomposing the modulation (or distance to the relevant characteristic surface) and writing
    \begin{align*} \phi_{\mu, N}& \big( \Pi_{\pm_1} \psi_{\lambda_1, N_1}\big)^\dagger \gamma^0 \Pi_{\pm_2} \varphi_{\lambda_2, N_2} \\
    &= \sum_{ d \in 2^\ZZ}  C_d \phi_{\mu, N} \big(  \mc{C}^{\pm_1}_{\leq d}\psi_{\lambda_1, N_1}\big)^\dagger \gamma^0 \mc{C}^{\pm_2}_{\leq d} \varphi_{\lambda_2, N_2}+ C_{< d} \phi_{\mu, N} \big( \mc{C}^{\pm_1}_{d}\psi_{\lambda_1, N_1}\big)^\dagger \gamma^0 \mc{C}^{\pm_2}_{\leq d} \varphi_{\lambda_2, N_2} \\
    &\qquad \qquad \qquad\qquad \qquad  + C_{< d} \phi_{\mu, N} \big( \mc{C}^{\pm_1}_{<d}\psi_{\lambda_1, N_1}\big)^\dagger \gamma^0 \mc{C}^{\pm_2}_{d} \varphi_{\lambda_2, N_2} \\
   &= \sum_{d \in 2^\ZZ}  A_0 + A_1 + A_2
   \end{align*}
 Keeping in mind (\ref{eqn - littlewood-paley cond}), we now divide the proof into cases depending on the relative sizes of the frequency and the modulation. Namely, we consider separately the low modulation cases
$$\lambda_1 \approx \lambda_2 \gg \mu \text{ and } d \ll \lambda_1, \qquad \mu \gg \min\{ \lambda_1, \lambda_2\} \text{ and } d \ll \mu, \qquad \lambda_1 \approx \lambda_2 \approx \mu \text{ and } d\ll \mu$$
and the high modulation cases
$$\lambda_1 \approx \lambda_2 \gtrsim \mu \text{ and } d \gtrsim \lambda_1, \qquad \mu \gg \min\{ \lambda_1, \lambda_2\} \text{ and } d \gtrsim \mu.$$
 Clearly, this covers all possible frequency combinations. The first case in the low modulation regime, where the two spinors are high frequency, is the easiest, as this case interacts very favourably with the null structure. The second case, when $\mu \gg \min\{ \lambda_1, \lambda_2\}$, is more difficult, and is the main obstruction to the scale invariant Sobolev result. The final case, when $\mu \approx \lambda_1 \approx \lambda_2$ is the only resonant interaction, and this is where the bilinear estimates in Corollary \ref{cor - small scale bilinear estimate} play a crucial role. In the remaining high modulation cases $d \gtrsim \max\{\mu, \lambda_1, \lambda_2\}$, the null structure of the system no longer plays any role, and we need to exploit the $Y^{\pm, m}_{\lambda, N}$ norms to gain the off diagonal decay term.

\medskip

{\bf High-Low I: $\mu \ll \lambda_1 \approx \lambda_2$ and $d \ll \lambda_1$.} Our goal is to show that
\begin{equation}\label{eqn - high low I main bound}
              \begin{split} \sum_{d \ll \lambda_1 } \Big| \int_{\RR^{1+3}} A_0 dx dt \Big| + &\Big| \int_{\RR^{1+3}} A_1 dx dt \Big| + \Big| \int_{\RR^{1+3}} A_2 dx dt \Big|\\
              &\lesa  \mu^{\frac{1}{2}}  N_{\min}^{\delta} \Big( \frac{\mu}{\lambda_1} \Big)^\frac{1}{4} \| \phi_{\mu, N} \|_{V^2_{+,1}}  \| \psi_{\lambda_1, N_1} \|_{V^2_{\pm_1,M }} \| \varphi_{\lambda_2, N_2} \|_{V^2_{\pm_2,M}}
              \end{split}
    \end{equation}
where we let $N_{min} = \min\{N, N_1, N_2\}$. Clearly this gives the bounds (\ref{eqn - trilinear freq loc integral I}), (\ref{eqn - trilinear freq loc integral II}), and (\ref{eqn - trilinear freq loc integral IV}) in the region $\mu \ll \lambda_1 \approx \lambda_2$ and $d \ll \lambda_1$.

We now prove the bound (\ref{eqn - high low I main bound}). The condition $d \ll \lambda_1$, together with an application of Lemma \ref{lem - modulation bound}, implies that we must have $\pm_1 = \pm_2$ and moreover, that the sum over the modulation is restricted to the region $\mu^{-1} \lesa d \lesa \mu$ (in particular this case is non-resonant). To estimate the first term, $A_0$, we note that after another application of Lemma \ref{lem - modulation bound},  we have the almost orthogonal decomposition
    $$ A_0 = \sum_{\substack{ \kappa, \kappa' \in \mc{C}_\alpha \\ |\kappa - \kappa'| \lesa \alpha} } \sum_{ \substack{ q, q' \in Q_\mu \\ |q-q'| \lesa \mu} }  C_d \phi_{\mu, N} \big(  \mc{C}^{\pm_1}_{\leq d}R_{\kappa} P_{q} \psi_{\lambda_1, N_1}\big)^\dagger \gamma^0 \mc{C}^{\pm_2}_{\leq d} R_{\kappa'} P_{q'} \varphi_{\lambda_2, N_2}$$
where $\alpha=(d\mu)^{\frac12}\lambda_1^{-1}$. Then, using the null-structure by writing
\[
 \mc{C}^{\pm_1}_{\les d}R_{\kappa}P_{\lambda_1}=C^{\pm_1,M}_{\les d}(\Pi_{\pm_1}-\Pi_{\pm_1}(\lambda_1\omega))R_{\kappa}P_{\lambda_1}+C^{\pm_1,M}_{\les d}\Pi_{\pm_1}(\lambda_1\omega_\kappa)R_{\kappa}P_{\lambda_1}
\]
(here $\omega_\kappa$ denotes the centre of the cap $\kappa$) and applying Lemma \ref{lem - null form bound}, together with the uniform disposability  of $C^{\pm_1,M}_{\les d}$ from \eqref{eqn:dispose}, we obtain for every $\epsilon>0$
\begin{align}\nonumber
\Big|\int A_0 dxdt \Big|& \lesa \sum_{\substack{ \kappa, \kappa' \in \mc{C}_\alpha \\ |\kappa - \kappa'| \lesa \alpha} } \sum_{ \substack{ q, q' \in Q_\mu \\ |q-q'| \lesa \mu} } \alpha \|C_d \phi_{\mu, N} \|_{L^2_{t,x}} \|R_{\kappa} P_{q} \psi_{\lambda_1, N_1}\|_{L^4_{t,x}} \| R_{\kappa'} P_{q'} \varphi_{\lambda_2, N_2}\|_{L^4_{t,x}}\\
& \lesa \mu^\frac{1}{2} \alpha^{-\epsilon} \Big( \frac{\mu}{\lambda_1}\Big)^{\frac{1}{2} - \epsilon} \| \phi_{\mu, N} \|_{V^2_{+,1}}  \| \psi_{\lambda_1, N_1} \|_{V^2_{\pm_1,M }} \| \varphi_{\lambda_2, N_2} \|_{V^2_{\pm_2,M}},\label{eq:A0-first}
\end{align}
where we used Lemma \ref{lem - norm controls Xsb} to control the $L^2_{t,x}$ norm of the high-modulation term, and the bound (\ref{eqn:L4sqrsum}). On the other hand, we can decompose
$$ A_0 = \sum_{\substack{ \kappa, \kappa',\kappa'' \in \mc{C}_\beta \\ |\kappa - \kappa'|, | \kappa'' \pm_2 \kappa'| \lesa \beta} }   C_d  R_{\kappa''} \phi_{\mu, N} \big(  \mc{C}^{\pm_1}_{\leq d}R_{\kappa} \psi_{\lambda_1, N_1}\big)^\dagger \gamma^0 \mc{C}^{\pm_2}_{\leq d} R_{\kappa'} \varphi_{\lambda_2, N_2}$$
where $\beta=d^{\frac12}\mu^{-\frac12}$, again by almost orthogonality and Lemma \ref{lem - modulation bound}. As above, we obtain for every $\epsilon>0$
\begin{align} \nonumber
\Big|\int A_0 dxdt \Big|& \lesa  \sum_{\substack{ \kappa, \kappa',\kappa'' \in \mc{C}_\beta \\ |\kappa - \kappa'|, |\kappa'' \pm_2 \kappa'| \lesa \beta} } \beta \|C_d R_{\kappa''} \phi_{\mu, N} \|_{L^2_{t,x}} \|R_{\kappa}  \psi_{\lambda_1, N_1}\|_{L^4_{t,x}} \| R_{\kappa'} \varphi_{\lambda_2, N_2}\|_{L^4_{t,x}}\\
& \lesa \beta^{1-\epsilon} d^{-\frac{1}{2}} \lambda (\beta N_{\min})^{\frac{1}{4}} \| \phi_{\mu, N} \|_{V^2_{+,1}}  \| \psi_{\lambda_1, N_1} \|_{V^2_{\pm_1,M }} \| \varphi_{\lambda_2, N_2} \|_{V^2_{\pm_2,M}},\label{eq:A0-ang}
\end{align}
where we used the angular concentration Lemma \ref{lem - angular concentration} on the lowest angular frequency term. Combining \eqref{eq:A0-first} and \eqref{eq:A0-ang}, by taking $\epsilon>0$ sufficiently small, we obtain for every $0<\delta \ll 1$
\begin{align*}
\sum_{\mu^{-1} \lesa d \lesa \mu} \Big|\int A_0 dxdt \Big| & \lesa \sum_{\mu^{-1} \lesa d \lesa \mu}
\Big(\frac{d}{\mu}\Big)^{\frac{\delta}{4}}N_{\min}^{\delta} \Big(\frac{\mu}{\lambda}\Big)^{\frac14}\mu^{\frac12}
\| P_\mu H_N \phi\|_{V^2_{+,1}} \| \psi \|_{F^{\pm_1, M}_{\lambda_1, N_1}} \| \varphi \|_{F^{\pm_2, M}_{\lambda_2, N_2}}\\
 & \lesa N_{\min}^{\delta} \Big(\frac{\mu}{\lambda}\Big)^{\frac14}\mu^{\frac12}\| \phi_{\mu, N} \|_{V^2_{+,1}}  \| \psi_{\lambda_1, N_1} \|_{V^2_{\pm_1,M }} \| \varphi_{\lambda_2, N_2} \|_{V^2_{\pm_2,M}}
\end{align*}
which gives (\ref{eqn - high low I main bound}) for the $A_0$ term. Next, we deal with the $A_1$ term. The argument is similar to the above, but the initial decomposition is slightly different as we no longer require the cube decomposition. Instead, we need to decompose the $\phi$ term into caps to ensure that the $C_{<d}$ multiplier is disposable. In more detail, the resonance bound in Lemma \ref{lem - modulation bound} gives
   $$A_1 = \sum_{\substack{ \kappa, \kappa' \in \mc{C}_\alpha \\ |\kappa - \kappa'| \lesa \alpha} } \sum_{\substack{\kappa'' \in \mc{C}_\beta \\ |\kappa'' \pm_2 \kappa'| \lesa \beta}}   C_{<d} R_{\kappa''} \phi_{\mu, N} \big(  \mc{C}^{\pm_1}_{d}R_{\kappa} \psi_{\lambda_1, N_1}\big)^\dagger \gamma^0 \mc{C}^{\pm_2}_{\leq d} R_{\kappa'} \varphi_{\lambda_2, N_2}$$
where $\alpha=( \frac{d\mu}{\lambda_1^2} )^\frac{1}{2}$ and $\beta = ( \frac{d}{\mu})^\frac{1}{2}$. By exploiting the null structure as previously, we then obtain for every $\epsilon>0$
\begin{align}\nonumber
\Big|\int A_1 dxdt \Big|& \lesa \sum_{\substack{ \kappa, \kappa' \in \mc{C}_\alpha \\ |\kappa - \kappa'| \lesa \alpha} } \sum_{\substack{\kappa'' \in \mc{C}_\beta \\ |\kappa'' \pm_2 \kappa'| \lesa \beta}} \alpha \|  R_{\kappa''} \phi_{\mu, N} \|_{L^4_{t,x}} \| \mc{C}^{\pm_1}_d R_{\kappa}  \psi_{\lambda_1, N_1}\|_{L^2_{t,x}} \| R_{\kappa'} \varphi_{\lambda_2, N_2}\|_{L^4_{t,x}}\\
& \lesa \alpha^{1-\epsilon} \mu^\frac{1}{2}  d^{-\frac{1}{2}} \lambda_2^\frac{1}{2} \| \phi_{\mu, N} \|_{V^2_{+,1}}  \| \psi_{\lambda_1, N_1} \|_{V^2_{\pm_1,M }} \| \varphi_{\lambda_2, N_2} \|_{V^2_{\pm_2,M}},\label{eq:A1-first}
\end{align}
where we used Lemma \ref{lem - norm controls Xsb} to control the $L^2_{t,x}$ norm of the high-modulation term, and again used (\ref{eqn:L4sqrsum}) To gain a power of $d$, we again exploit the angular concentration estimate by exploiting a similar argument to (\ref{eq:A0-ang}) to deduce that
    \begin{align} \nonumber
\Big|\int A_1 dxdt \Big|& \lesa  \sum_{\substack{ \kappa, \kappa',\kappa'' \in \mc{C}_\beta \\ |\kappa - \kappa'|, |\kappa'' \pm_2 \kappa'| \lesa \beta} } \beta \| R_{\kappa''} \phi_{\mu, N} \|_{L^4_{t,x}} \| \mc{C}^{\pm_1}_d R_{\kappa}  \psi_{\lambda_1, N_1}\|_{L^2_{t,x}} \| R_{\kappa'} \varphi_{\lambda_2, N_2}\|_{L^4_{t,x}}\\
& \lesa \beta^{1-\epsilon} d^{-\frac12} \lambda^\frac{1}{2} \mu^\frac{1}{2} (\beta N_{\min})^\frac{1}{4} \| \phi_{\mu, N} \|_{V^2_{+,1}}  \| \psi_{\lambda_1, N_1} \|_{V^2_{\pm_1,M }} \| \varphi_{\lambda_2, N_2} \|_{V^2_{\pm_2,M}}.\label{eq:A1-ang}
\end{align}
Combining (\ref{eq:A1-ang}) and (\ref{eq:A1-ang}) as in the $A_0$ case, and summing up over $\mu^{-1} \lesa d \lesa \mu$ with $\epsilon$ sufficiently small, we obtain (\ref{eqn - high low I main bound}). The remaining term $A_2$ can be handled in an identical manner to the $A_1$. Thus the bound (\ref{eqn - high low I main bound}) follows.

\medskip

{\bf High-Low II: $\mu \gg \min\{ \lambda_1, \lambda_2\}$ and $d\ll \mu$.} Let $\{j, k\} = \{1, 2\}$ and $\lambda_j \g \lambda_k$. Our goal is  to prove that
	\begin{equation}\label{eqn - high low II main bound strong}
               \sum_{d \ll \mu } \Big| \int_{\RR^{1+3}} A_0 dx dt \Big| + \Big| \int_{\RR^{1+3}} A_j dx dt \Big|
              \lesa  \mu^{\frac{1}{2}}  N_{\min}^{\delta} \Big( \frac{\lambda_k}{\mu} \Big)^\frac{1}{8} \| \phi_{\mu, N} \|_{V^2_{+,1}}  \| \psi_{\lambda_1, N_1} \|_{V^2_{\pm_1,M }} \| \varphi_{\lambda_2, N_2} \|_{V^2_{\pm_2,M}}.
    \end{equation}
On the other hand, for the $A_k$ term, we have the weaker bounds
   \begin{equation}\label{eqn - high low II main bound weak I}
             \sum_{d \ll \mu }  \Big| \int_{\RR^{1+3}} A_k dx dt \Big| \lesa \mu^\frac{1}{2} \Big( \frac{\lambda_k}{\mu}\Big)^{\frac{\delta}{8}}  (\min\{N, N_j\})^\delta  \| \phi_{\mu, N} \|_{V^2_{+,1}}  \| \psi_{\lambda_1, N_1} \|_{V^2_{\pm_1,M }} \| \varphi_{\lambda_2, N_2} \|_{V^2_{\pm_2,M}}
\end{equation}
and
 \begin{equation}\label{eqn - high low II main bound weak II}
	\sum_{d \ll \mu }  \Big| \int_{\RR^{1+3}} A_k dx dt \Big| \lesa \mu^\frac{1}{2} \Big( \frac{\lambda_k}{\mu}\Big)^{\frac{1}{2a}-\frac{1}{4}} N_k^{\delta} \| \phi_{\mu, N} \|_{V^2_{+,1}}
	\begin{cases}   \| \psi \|_{F^{\pm_1, M}_{\lambda_1, N_1}} \| \varphi_{\lambda_2, N_2} \|_{V^2_{\pm_2,M}} \qquad &k=1 \\
	 \| \psi_{\lambda_1, N_1} \|_{V^2_{\pm_1,1}} \| \varphi \|_{F^{\pm_2, M}_{\lambda_2, N_2}} \qquad &k=2 \end{cases}
 \end{equation}
where $\frac{1}{2}<\frac{1}{a}<\frac{1}{2} + \frac{\sigma}{1000}$ is as in the definition of the $Y^{\pm, m}_{\lambda, N}$ norm. Clearly (\ref{eqn - high low II main bound strong}), (\ref{eqn - high low II main bound weak I}), and (\ref{eqn - high low II main bound weak II}) give the estimates claimed in Theorem \ref{thm - trilinear freq loc integral} in the region $\mu \gg \min\{ \lambda_1, \lambda_2\}$ and $d\ll \mu$. Note that we have a weaker bound when the low frequency term has modulation away from the hyperboloid, and for this interaction, we are forced to exploit the $Y^{\pm, m}_{\lambda, N}$ norms.

We begin the proof of (\ref{eqn - high low II main bound strong}), (\ref{eqn - high low II main bound weak I}), and (\ref{eqn - high low II main bound weak II}) by observing that since the estimate is essentially symmetric in $\psi$ and $\q$, it is enough to consider the case $\mu \approx \lambda_1 \gg \lambda_2$, in other words, we only consider the case $j=1$ and $k=2$. As in the previous case, as $d \ll \mu$, Lemma \ref{lem - modulation bound} implies that we only have a non-zero contribution if $\pm_1 = +$ and $ \lambda_2^{-1} \lesa d \lesa \lambda_2$. To control the $A_0$ term, we decompose into caps of radius $\beta = ( \frac{d}{\lambda_2})^\frac{1}{2}$ and cubes of diameter $\lambda_2$. Lemma \ref{lem - modulation bound} implies that we have the almost orthogonality identity
	$$A_0 = \sum_{\substack{ \kappa, \kappa' \in \mc{C}_\beta \\ |\kappa \mp_2 \kappa'| \lesa \beta} } \sum_{ \substack{ q, q' \in Q_{\lambda_2} \\ |q-q'| \lesa \lambda_2} }  P_{q'} C_d \phi_{\mu, N} \big(  P_q R_{\kappa} \mc{C}^{+}_{\leq d}\psi_{\lambda_1, N_1}\big)^\dagger \gamma^0 R_{\kappa'} \mc{C}^{\pm_2}_{\leq d} \varphi_{\lambda_2, N_2}. $$
Thus exploiting the null structure as previously, disposing of the $C^{\pm, m}_d$ multipliers using (\ref{eqn:dispose}), and applying the $L^4_{t, x}$ Strichartz estimate we obtain for every $\epsilon>0$
	\begin{align}
		\Big| \int_{\RR^{1+3}} A_0 dx dt\Big| &\lesa \sum_{\substack{ \kappa, \kappa' \in \mc{C}_\beta \\ |\kappa \mp_2 \kappa'| \lesa \beta} } \sum_{ \substack{ q, q' \in Q_{\lambda_2} \\ |q-q'| \lesa \lambda_2} } \beta \| P_{q'} C_d \phi_{\mu, N}\|_{L^2_{t, x}} \| P_q R_{\kappa} \psi_{\lambda_1, N_1}\|_{L^4_{t, x}} \| R_{\kappa'} \varphi_{\lambda_2, N_2} \|_{L^4_{t, x}} \notag\\
		&\lesa \beta^{-\epsilon} \mu^\frac{1}{2} \Big( \frac{\lambda_2}{\mu} \Big)^{\frac{1}{4}-\epsilon} \| \phi_{\mu, N} \|_{V^2_{+,1}}  \| \psi_{\lambda_1, N_1} \|_{V^2_{+,M }} \| \varphi_{\lambda_2, N_2} \|_{V^2_{\pm_2,M}}. \label{eqn:B_1-first}
    \end{align}
On the other hand, by decomposing into
	$$A_0 =  \sum_{\substack{ \kappa, \kappa',\kappa'' \in \mc{C}_\beta \\ |\kappa \mp_2 \kappa'|, |\kappa'' \pm_2 \kappa'| \lesa \beta} } R_{\kappa''} C_d \phi_{\mu, N} \big(   R_{\kappa} \mc{C}^{+}_{\leq d}\psi_{\lambda_1, N_1}\big)^\dagger \gamma^0 R_{\kappa'} \mc{C}^{\pm_2}_{\leq d} \varphi_{\lambda_2, N_2}$$
and using the angular concentration bound Lemma \ref{lem - angular concentration} on the smallest angular frequency term, a similar argument gives
	\begin{align}
		\Big| \int_{\RR^{1+3}} A_0 dx dt\Big| &\lesa \sum_{\substack{ \kappa, \kappa',\kappa'' \in \mc{C}_\beta \\ |\kappa \mp_2 \kappa'|, |\kappa'' \pm_2 \kappa'| \lesa \beta} }  \beta \|  C_d R_{\kappa''} \phi_{\mu, N}\|_{L^2_{t, x}} \|  R_{\kappa} \psi_{\lambda_1, N_1}\|_{L^4_{t, x}} \| R_{\kappa'} \varphi_{\lambda_2, N_2} \|_{L^4_{t, x}} \notag\\
		&\lesa \mu^\frac{1}{2} \beta^{\frac{1}{4}-\epsilon} N_{min}^{\frac{1}{4}} \| \phi_{\mu, N} \|_{V^2_{+,1}}  \| \psi_{\lambda_1, N_1} \|_{V^2_{ +,M }} \| \varphi_{\lambda_2, N_2} \|_{V^2_{\pm_2,M}}. \label{eqn:B_1-ang}
    \end{align}
As in the previous case, combining (\ref{eqn:B_1-first}) and (\ref{eqn:B_1-ang}) with $\epsilon$ sufficiently small gives (\ref{eqn - high low II main bound strong}) for the $A_0$ term. The $A_1$ term can be estimated by an identical argument (since the high modulation term is again at frequency $\mu$). To control the $A_2$ component, we start by again applying Lemma \ref{lem - modulation bound} and decomposing into
  $$ A_2 = \sum_{\substack{ \kappa, \kappa' \in \mc{C}_\beta \\ |\kappa \mp_2 \kappa'| \lesa \beta} } \sum_{ \substack{ q, q' \in Q_{\lambda_2} \\ |q-q'| \lesa \lambda_2} }  P_{q'} C_{<d} \phi_{\mu, N} \big(  P_q R_{\kappa} \mc{C}^{+}_{\leq d}\psi_{\lambda_1, N_1}\big)^\dagger \gamma^0 R_{\kappa'} \mc{C}^{\pm_2}_{ d} \varphi_{\lambda_2, N_2} $$
 where as usual $\beta = (\frac{d}{\lambda_2})^\frac{1}{2}$. Applying the by now standard null form bound, (\ref{eqn:dispose}), and the $L^4_{t, x}$ Strichartz estimate, we conclude that for every $\epsilon>0$
	\begin{align}
		\Big| \int_{\RR^{1+3}} A_2 dx dt\Big| &\lesa \sum_{\substack{ \kappa, \kappa' \in \mc{C}_\beta \\ |\kappa \mp_2 \kappa'| \lesa \beta} } \sum_{ \substack{ q, q' \in Q_{\lambda_2} \\ |q-q'| \lesa \lambda_2} } \beta \| P_{q'}\phi_{\mu, N}\|_{L^4_{t, x}} \| P_q R_{\kappa} \psi_{\lambda_1, N_1}\|_{L^4_{t, x}} \| R_{\kappa'}  \mc{C}^{\pm_2}_{ d} \varphi_{\lambda_2, N_2} \|_{L^2_{t, x}} \notag\\
		&\lesa \mu^\frac{1}{2} \beta^{-\epsilon} \Big( \frac{\mu}{\lambda_2}\Big)^\epsilon \| \phi_{\mu, N} \|_{V^2_{+,1}}  \| \psi_{\lambda_1, N_1} \|_{V^2_{ +,M }} \| \varphi_{\lambda_2, N_2} \|_{V^2_{\pm_2,M}}. \label{eqn:B_3-first}
   \end{align}
Note that we get no high frequency gain here (in fact we have a slight loss due to the sum over cubes). On the other hand, by decomposing all three terms into caps of size $\beta$, using null structure, the $L^q_t L^4_x$ Strichartz estimate in Lemma \ref{lem - wave strichartz}, and  Bernstein's inequality followed by Lemma \ref{lem - norm controls Xsb} for $\varphi_{\lambda_2, N_2}$,  we obtain  for any $2<q<2 + \frac{2}{3}$
	\begin{align}
		\Big| \int_{\RR^{1+3}} A_2 dx dt\Big| &\lesa \sum_{\substack{ \kappa, \kappa',\kappa'' \in \mc{C}_\beta \\ |\kappa \mp_2 \kappa'|, |\kappa'' \pm_2 \kappa'| \lesa \beta} }  \beta \| R_{\kappa''} \phi_{\mu, N}\|_{L^\frac{q}{q-2}_t L^\frac{2q}{4-q}_x}  \|  R_{\kappa} \psi_{\lambda_1, N_1} \|_{L^q_t L^4_x} \|R_{\kappa''} \mc{C}^{\pm_2}_d \varphi_{\lambda_2, N_2} \|_{L^q_t L^{\frac{4q}{5q-8}}_x} \notag \\
	    &\lesa \mu^\frac{1}{2} \Big( \frac{d}{\lambda_2}\Big)^{\frac{1}{q} - \frac{1}{4} - \epsilon} \Big( \frac{\lambda_2}{\mu}\Big)^{\frac{5}{q} - \frac{9}{4}}  N_1\| \phi_{\mu, N} \|_{V^2_{+,1}}  \| \psi_{\lambda_1, N_1} \|_{V^2_{ +,M }} \| \varphi_{\lambda_2, N_2} \|_{V^2_{\pm_2,M}}\label{eqn:B_3-ang}
	\end{align}
(schematically, we are putting the product into $L^{\infty-}_tL^{2+}_x\times L^{2+}_t L^{4}_x \times L^{2+}_t L^{4-}_x $). Switching the roles of $\phi_{\mu, N}$ and $\psi_{\lambda_1, N_1}$,  and combining (\ref{eqn:B_3-first}) and (\ref{eqn:B_3-ang}) with $q$ close to 2, and $\epsilon>0$ sufficiently small, we obtain (\ref{eqn - high low II main bound weak I}).

It remains to prove (\ref{eqn - high low II main bound weak II}), thus we need to consider the case where $\varphi$ also has the smallest angular frequency. We begin by again using Lemma \ref{lem - modulation bound} to decompose
	$$A_2 = \sum_{\substack{ \kappa, \kappa',\kappa'' \in \mc{C}_\beta \\ |\kappa \mp_2 \kappa'|, |\kappa'' \pm_2 \kappa'| \lesa \beta} } \sum_{\substack{q,q'' \in Q_{\lambda_2} \\ |q-q''| \lesa \lambda_2}}
	R_{\kappa''} P_{q''} C_{<d} \phi_{\mu, N} \big(   R_{\kappa} P_q \mc{C}^{+}_{\leq d}\psi_{\lambda_1, N_1}\big)^\dagger \gamma^0 R_{\kappa'} \mc{C}^{\pm_2}_{d} \varphi_{\lambda_2, N_2}
	$$
where $\beta = (\frac{d}{\lambda_2})^\frac{1}{2}$. An application of Bernstein's inequality, Lemma \ref{lem - norm controls Xsb}, and  the angular concentration lemma for $\varphi$, together with the null form bound,  and Lemma \ref{lem - wave strichartz}, implies that for any $\epsilon>0$ sufficiently small
	\begin{align*}
		\Big| &\int_{\RR^{1+3}} A_2 dx dt \Big| \\
		&\lesa \sum_{\substack{ \kappa, \kappa',\kappa'' \in \mc{C}_\beta \\ |\kappa \mp_2 \kappa'|, |\kappa'' \pm_2 \kappa'| \lesa \beta} } \sum_{\substack{q,q'' \in Q_{\lambda_2} \\ |q-q''| \lesa \lambda_2}} \beta
	\| R_{\kappa''} P_{q''}  \phi_{\mu, N}\|_{L^{\frac{2a}{a-1}}_t L^{2a}_x} \| R_{\kappa} P_q \psi_{\lambda_1, N_1}\|_{L^{\frac{2a}{a-1}}_t L^{2a}_x} \| R_{\kappa'} \mc{C}^{\pm_2}_{d} \varphi_{\lambda_2, N_2} \|_{L^a_t L^{\frac{a}{a-1}}_x} \\
	&\lesa \beta^{1-\epsilon} \Big( \frac{\mu}{\lambda_2}\Big)^{\epsilon} (\mu \lambda_2)^{1-\frac{1}{a}} ( \beta^2 \lambda_2^3 )^{\frac{1}{a} - \frac{1}{2}} (\beta N_2)^{\delta} \| P_\mu H_N \phi\|_{V^2_{+,1}} \| \psi \|_{F^{\pm_1, M}_{\lambda_1, N_1}}  \| \mc{C}^{\pm_2}_{d} \varphi_{\lambda_2, N_2} \|_{L^a_t L^2_x}\\
	&\lesa \mu^\frac{1}{2} N_2^{\delta} \Big(\frac{\lambda_2}{\mu}\Big)^{\frac{1}{2a} - \frac{1}{4}} \Big( \frac{d}{\lambda_2}\Big)^{\frac{1}{2}(\frac{\delta}{2} - b + \frac{1}{a})}  \|  \phi_{\mu, N} \|_{V^2_{+,1}} \| \psi_{\lambda_1, N_1} \|_{V^2_{+,M}}   \| \varphi \|_{Y^{\pm_2, M}_{\lambda_2, N_2}}
	\end{align*}
which gives  (\ref{eqn - high low II main bound weak II}) since $\frac{1}{2}<\frac{1}{a}<\frac{1}{2} + \frac{\sigma}{1000}$, and  $b-\frac{1}{a} = \frac{2}{a} - 1< \frac{\sigma}{500} < \frac{\delta}{5}$.

\medskip

{\bf High-High: $\mu \approx \lambda_1 \approx \lambda_2$ and $d\ll \mu$.} Our goal is to prove that if $M \g \frac{1}{2}$, then for any $\delta>0$  we have the bound
	\begin{equation}\label{eqn - high high nonresonant}
              \begin{split} \sum_{  d \ll \mu } \Big| \int_{\RR^{1+3}} A_0 dx dt \Big| + \Big| \int_{\RR^{1+3}} A_1 dx dt \Big| + &\Big| \int_{\RR^{1+3}} A_2 dx dt \Big|\\
              &\lesa  \mu^{\frac{1}{2}}  N_{\min}^{\delta} \| \phi_{\mu, N} \|_{V^2_{+,1}}  \| \psi_{\lambda_1, N_1} \|_{V^2_{\pm_1,M }} \| \varphi_{\lambda_2, N_2} \|_{V^2_{\pm_2,M}}\end{split}
    \end{equation}
while if $0<M<\frac{1}{2}$, for every $s, \delta > 0$, we have
	\begin{equation}\label{eqn - high high resonant}
              \begin{split}  \Big| \int_{\RR^{1+3}} \sum_{ d \ll \mu  } A_0 +& A_1 + A_2 dx dt \Big|\\
              &\lesa \mu^\frac{1}{2} N_{\min}^\delta \big( 1 +  \mu^{-\frac{1}{6} + s}  N_{\min}^{\frac{7}{30}}\big) \| \phi_{\mu, N} \|_{V^2_{+,1}}  \| \psi_{\lambda_1, N_1} \|_{V^2_{\pm_1,M }} \| \varphi_{\lambda_2, N_2} \|_{V^2_{\pm_2,M}}.\end{split}
    \end{equation}
The key difference to the previous cases, is that if $0<M\les \frac{1}{2}$, we no longer have the non-resonant bound $d \gtrsim \mu^{-1}$, and thus we also have to estimate the \emph{resonant} interactions $d \ll \mu^{-1}$. This is particularly challenging in light of the fact that in the \emph{strongly resonant regime}, $0<M<\frac{1}{2}$, there is no gain from the null structure when $d \ll \mu^{-1}$. However, we do have \emph{transversality} in the region $d\ll \mu^{-1}$, and consequently, we can apply the key bilinear restriction estimate in Corollary \ref{cor - small scale bilinear estimate}. On the other hand, in the \emph{weakly resonant regime}, $M=\frac{1}{2}$, somewhat surprisingly and in stark contrast to the cases $M\not =\frac{1}{2}$, the null structure gives cancellation for \emph{all} modulation scales.

We start by considering the non-resonant region $ \mu^{-1} \lesa d \lesa \mu$. By decomposing into caps of radius $\beta = (\frac{d}{\mu})^{\frac{1}{2}}$, an application of Lemma \ref{lem - modulation bound} gives the identity
	$$A_0 =  \sum_{\substack{ \kappa, \kappa',\kappa'' \in \mc{C}_\beta \\ |\pm_1 \kappa \mp_2 \kappa'|, |\kappa'' \pm_2 \kappa'| \lesa \beta} } R_{\kappa''} C_d \phi_{\mu, N} \big(   R_{\kappa} \mc{C}^{\pm_1}_{\leq d}\psi_{\lambda_1, N_1}\big)^\dagger \gamma^0 R_{\kappa'} \mc{C}^{\pm_2}_{\leq d} \varphi_{\lambda_2, N_2}.$$
Thus by applying the $L^4_{t,x}$ Strichartz bound, exploiting the null structure as previously (here we need the assumption $d \gtrsim \mu^{-1}$), and using the angular concentration bound in Lemma \ref{lem - angular concentration} on $N_{min}$, we obtain for every $\epsilon>0$
	\begin{align*}
		\Big| \int_{\RR^{1+3}} A_0 dx dx \Big| &\lesa \sum_{\substack{ \kappa, \kappa',\kappa'' \in \mc{C}_\beta \\ |\pm_1 \kappa \mp_2 \kappa'|, |\kappa'' \pm_2 \kappa'| \lesa \beta} } \beta  \|R_{\kappa''} C_d \phi_{\mu, N}\|_{L^2_{t, x}}\|R_{\kappa} \psi_{\lambda_1, N_1}\|_{L^4_{t, x}} \|R_{\kappa'} \varphi_{\lambda_2, N_2}\|_{L^4_{t, x}} \\
		&\lesa \beta^{1-\epsilon} d^{-\frac{1}{2}} \mu  ( \beta N_{\min} )^\delta \| \phi_{\mu, N} \|_{V^2_{+,1}}  \| \psi_{\lambda_1, N_1} \|_{V^2_{\pm_1,M }} \| \varphi_{\lambda_2, N_2} \|_{V^2_{\pm_2,M}}.
	\end{align*}
Taking $\delta>0$ and $\epsilon>0$ sufficiently small, and summing up over the modulation $\mu^{-1} \lesa d \lesa \mu$ then gives (\ref{eqn - high high nonresonant}) and (\ref{eqn - high high resonant}) for $A_0$ in the region $\mu^{-1} \lesa d \lesa \mu$. A similar argument bounds the $A_1$ and $A_2$ terms in (\ref{eqn - high high nonresonant}) and (\ref{eqn - high high resonant}) provided the sum over modulation is restricted to $\mu^{-1} \lesa d \lesa \mu$.

We now consider the case $d \ll \mu^{-1}$. Note that if $M>\frac{1}{2}$, then using Lemma \ref{lem - modulation bound}, we see that $A_0=A_1=A_2 = 0$ and thus (\ref{eqn - high high nonresonant}) is immediate. On the other hand, if we are in the weakly resonant regime $M=\frac{1}{2}$, then another application of Lemma \ref{lem - modulation bound} implies that $\pm_1 = +$, $\pm_2 = -$,  and we have the decomposition
    $$ A_0 =  \sum_{\substack{ \kappa, \kappa',\kappa'' \in \mc{C}_\beta \\ | \kappa + \kappa'|, |\kappa'' - \kappa| \lesa \beta} } \sum_{\substack{ q, q' \in Q_{\mu^2 \beta} \\ |q+q'| \lesa \mu^2 \beta}}  R_{\kappa''} C_d \phi_{\mu, N} \big(   R_{\kappa} P_q \mc{C}^{+}_{\leq d}\psi_{\lambda_1, N_1}\big)^\dagger \gamma^0 R_{\kappa'} P_{q'} \mc{C}^{-}_{\leq d} \varphi_{\lambda_2, N_2}$$
where $\beta = (\frac{d}{\mu})^\frac{1}{2}$. Therefore, using the null form type bound (\ref{eqn:nullsymbolbound}), together with (ii) in Lemma \ref{lem - null form bound} to exploit the null structure, the orthogonality estimate in Lemma \ref{lem - square sum loss}, and an application of Lemma \ref{lem - wave strichartz} gives for every $\epsilon >0$
    \begin{align*}
      \Big| \int_{\RR^{1+3}} A_0 dx dt \Big| &\lesa \sum_{\substack{ \kappa, \kappa',\kappa'' \in \mc{C}_\beta \\ | \kappa - \kappa'|, |\kappa'' - \kappa| \lesa \beta} } \sum_{\substack{ q, q' \in Q_{\mu^2 \beta} \\ |q+q'| \lesa \mu^2 \beta}}  \beta \|R_{\kappa''} C_d \phi_{\mu, N}\|_{L^2_{t, x}} \| R_{\kappa} P_q \psi_{\lambda_1, N_1}\|_{L^4_{t, x}} \| R_{\kappa'} P_{q'} \varphi_{\lambda_2, N_2} \|_{L^4_{t, x}} \\
      &\lesa \beta \times d^{-\frac{1}{2}} \times \mu \times \beta^{-\epsilon} (\mu \beta)^{-\epsilon} \times (\beta N_{min})^\delta \| \phi_{\mu, N} \|_{V^2_{+, 1}} \| \psi_{\lambda_1, N_1} \|_{V^2_{+, M}} \| \varphi \|_{V^2_{-, M}}
          \end{align*}
where we used the angular concentration bound in Lemma \ref{lem - angular concentration} on the term with smallest angular frequency. Choosing $\epsilon>0$ sufficiently small, and summing up over $0<d \ll \mu^{-1}$ then gives (\ref{eqn - high high nonresonant}) for the $A_0$ term. An identical argument bounds the $A_1$ and $A_2$ terms.

It remains to prove (\ref{eqn - high high resonant}) when $0<d \ll \mu^{-1}$. Another application of Lemma \ref{lem - modulation bound}, implies that we must have $\pm_1 = +$ and $\pm_2=-$, as well as the key orthogonality identity
	\begin{align*}
		\sum_{d \ll \mu^{-1}}& A_0 + A_1 + A_2 \\
		&= C_{\ll \mu^{-1}} \phi_{\mu, N} \big(  \mc{C}^{+}_{\ll \mu^{-1} }\psi_{\lambda_1, N_1}\big)^\dagger \gamma^0 \mc{C}^{-}_{\ll \mu^{-1} } \varphi_{\lambda_2, N_2} \\
			&= \sum_{\substack{\kappa, \kappa', \kappa'' \in \mc{C}_{\mu^{-1}} \\ |\kappa + \kappa'|\lesa \mu^{-1}}} \sum_{\substack{q, q'' \in Q_\mu \\ |q - q''| \approx \mu \text{ or } |\kappa - \kappa''| \approx \mu^{-1}}}
		R_{\kappa''} P_{q''}C_{\ll \mu^{-1}} \phi_{\mu, N} \big( R_{\kappa} P_q \mc{C}^{+}_{\ll \mu^{-1} }\psi_{\lambda_1, N_1}\big)^\dagger \gamma^0 R_{\kappa'} \mc{C}^{-}_{\ll \mu^{-1} } \varphi_{\lambda_2, N_2}.
	\end{align*}
Note that the summation is restricted to terms for which $R_{\kappa''} P_{q''}C_{\ll \mu^{-1}} \phi_{\mu, N} $ and  $R_{\kappa} P_q \mc{C}^{+}_{\ll \mu^{-1} }\psi_{\lambda_1, N_1}$  have either angular orthogonality, or radial orthogonality. In either case, we may apply Corollary \ref{cor - small scale bilinear estimate} (via the bound (\ref{eqn:sqrsum bilinear})), the null structure bound in Lemma \ref{lem - null form bound}, and the Klein-Gordon angular Strichartz estimate in Lemma \ref{lem - KG strichartz}, to deduce that for every $\frac{3}{2} < q< \frac{7}{10}$ and $\epsilon>0$ we have
	\begin{align*}
	\Big| \int_{\RR^{1+3}} \sum_{d \ll \mu^{-1}}& A_0 + A_1 + A_2 dx dt \Big| \\
				&\lesa \mu^{-1} \sum_{\substack{\kappa,  \kappa'' \in \mc{C}_{\mu^{-1}} }} \sum_{\substack{q, q' \in Q_\mu \\ |q - q''| \approx \mu \text{ or } |\kappa - \kappa''| \approx \mu^{-1}}}
					\|R_{\kappa''} P_{q''}C_{\ll \mu^{-1}} \phi_{\mu, N} \big( R_{\kappa} P_q \mc{C}^{+}_{\ll \mu^{-1} }\psi_{\lambda_1, N_1}\big)^\dagger\|_{L^q_{t, x}} \| \varphi_{\lambda_2, N_2} \|_{L^{q'}_{t, x}} \\
		&\lesa 	\mu^{\frac{5}{q} - 3 + \epsilon} N_2^{7(\frac{7}{10} - \frac{1}{q}) + \epsilon} 	\| \phi_{\mu, N} \|_{V^2_{+,1}}  \| \psi_{\lambda_1, N_1} \|_{V^2_{+,M }} \| \varphi_{\lambda_2, N_2} \|_{V^2_{-,M}}
	\end{align*}
where for ease of reading we suppressed the $\Pi_\pm(\omega_\kappa)$ matrices used to extract the null form gain of $\mu^{-1}$. Choosing $q$ sufficiently close to $\frac{3}{2}$, and $\epsilon>0$ sufficiently small,  then gives (\ref{eqn - high high resonant}) in the case $N_2 = N_{min}$. To deal with remaining cases, we just reverse the roles of $\phi$, $\psi$, and $\varphi$, again apply Lemma \ref{lem - modulation bound} to deduce the required transversality, and always use the angular Strichartz estimate from Lemma \ref{lem - KG strichartz} on the term with smallest angular frequency. This completes the proof of (\ref{eqn - high high resonant}).

\medskip

{\bf High modulation I: $\mu \lesa \lambda_1 \approx \lambda_2$ and $d \gtrsim \lambda_1$.} We now consider the high modulation case. In this region, the null structure plays no role, and thus the arguments are significantly easier. Our goal is to prove that
	\begin{equation}\label{eqn - high mod, high high main} \begin{split}
		\sum_{d \gtrsim \lambda_1} \Big| \int_{\RR^{1+3}} A_1 dx dt \Big| + \Big| \int_{\RR^{1+3}} A_2 dx dt \Big|
		&\lesa \mu^\frac{1}{2} \Big(\frac{\mu}{\lambda_1} \Big)^\frac{1}{8} \| \phi_{\mu, N} \|_{V^2_{+,1}}  \| \psi_{\lambda_1, N_1} \|_{V^2_{\pm_1,M }} \| \varphi_{\lambda_2, N_2} \|_{V^2_{\pm_2,M}}
		\end{split}
	\end{equation}
and for every $\delta>0$, the weaker bounds
	\begin{equation}\label{eqn - high mod, high high main A0}
		\sum_{d \gtrsim \lambda_1} \Big| \int_{\RR^{1+3}} A_0 dx dt \Big| \lesa \mu^\frac{1}{2} \Big(\frac{\mu}{\lambda_1} \Big)^{\frac{\delta}{8}} (\min\{N_1, N_2\})^\delta \| \phi_{\mu, N} \|_{V^2_{+,1}}  \| \psi_{\lambda_1, N_1} \|_{V^2_{\pm_1,M }} \| \varphi_{\lambda_2, N_2} \|_{V^2_{\pm_2,M}}
	\end{equation}
and
	\begin{equation}\label{eqn - high mod, high high main A0 weak}
		\sum_{d \gtrsim \lambda_1} \Big| \int_{\RR^{1+3}} A_0 dx dt \Big| \lesa \mu^\frac{1}{2} \Big(\frac{\mu}{\lambda_1} \Big)^{\frac{1}{a} - \frac{1}{2}} \| \phi\|_{Y^{+, 1}_{\mu, N}}  \| \psi_{\lambda_1, N_1} \|_{V^2_{\pm_1,M }} \| \varphi_{\lambda_2, N_2} \|_{V^2_{\pm_2,M}}
	\end{equation}
where $a$ is as in the definition of the $Y^{\pm, m}_{\lambda, N}$ norm. We start with the estimates (\ref{eqn - high mod, high high main A0}) and (\ref{eqn - high mod, high high main A0 weak}) for the $A_0$ component. Decomposing $\psi$ and $\varphi$ into cubes of size $\mu$, together with an application of the $L^4_{t, x}$ Strichartz estimate gives for all $\epsilon>0$
	\begin{align}
	    \Big| \int_{\RR^{1+3}} A_0 dx dt \Big| &\lesa  \sum_{ \substack{ q, q' \in Q_\mu \\ |q-q'| \lesa \mu} }  \| C_d \phi_{\mu, N}\|_{L^2_{t, x}}  \| P_{q} \psi_{\lambda_1, N_1}\|_{L^4_{t, x}} \| P_{q'} \varphi_{\lambda_2, N_2} \|_{L^4_{t, x}} \notag \\
	    &\lesa   \mu^\frac{1}{2} \Big( \frac{\lambda_1}{\mu}\Big)^\epsilon \Big( \frac{\lambda_1}{d} \Big)^{\frac{1}{2}} \| \phi_{\mu, N} \|_{V^2_{+,1}}  \| \psi_{\lambda_1, N_1} \|_{V^2_{\pm_1,M }} \| \varphi_{\lambda_2, N_2} \|_{V^2_{\pm_2,M}}. \label{eqn - high mod, high high A0 I}
	\end{align}
As in the proof of (\ref{eqn:B_3-ang}), if we instead  apply the $L^q_t L^4_x$ bound, together with Bernstein's inequality for $\phi$, we obtain for any $2<q<2 + \frac{2}{11}$
	\begin{align}
	    \Big| \int_{\RR^{1+3}} A_0 dx dt \Big| &\lesa   \| C_d \phi_{\mu, N}\|_{L^q_t L^{\frac{4q}{5q-8}}_x}  \|  \psi_{\lambda_1, N_1}\|_{L^q_t L^4_x} \|  \varphi_{\lambda_2, N_2} \|_{L^\frac{q}{q-2}_t L^\frac{2q}{4-q}_x} \notag \\
	    &\lesa \mu^\frac{1}{2} \Big( \frac{\lambda_1}{d} \Big)^\frac{1}{q} \Big( \frac{\mu}{\lambda_1} \Big)^{\frac{6}{q} - \frac{11}{4}} N_1\| \phi_{\mu, N} \|_{V^2_{+,1}}  \| \psi_{\lambda_1, N_1} \|_{V^2_{\pm_1,M }} \| \varphi_{\lambda_2, N_2} \|_{V^2_{\pm_2,M}} \label{eqn - high mod, high high A0 ang}
	\end{align}
(schematically, we are putting the product into $L^{2+}_t L^{4-}_x \times L^{2+}_t L^4_x \times L^{\infty-}_tL^{2+}_x$). Switching the roles of $\psi_{\lambda_1, N_1}$ and $\varphi_{\lambda_2, N_2}$,  and combining (\ref{eqn - high mod, high high A0 I}) and (\ref{eqn - high mod, high high A0 ang}) with $q$ sufficiently close to $2$ and $\epsilon>0$ sufficiently small,  followed by summing up over $d \gtrsim \lambda_1$,  we obtain (\ref{eqn - high mod, high high main A0}). On the other hand, to obtain (\ref{eqn - high mod, high high main A0 weak}), we again use Lemma \ref{lem - wave strichartz} to deduce that
	\begin{align*}
		\Big| \int_{\RR^{1+3}} A_0 dx dt \Big|
		&\lesa \sum_{\substack{q,q' \in Q_{\mu} \\ |q-q''| \lesa \mu}}
	\|  C_d \phi_{\mu, N}\|_{L^a_t L^{\frac{a}{a-1}}_x} \| P_q \psi_{\lambda_1, N_1}\|_{L^{\frac{2a}{a-1}}_t L^{2a}_x} \|  P_{q'} \varphi_{\lambda_2, N_2} \|_{L^{\frac{2a}{a-1}}_t L^{2a}_x} \\
	&\lesa \mu^\frac{1}{2} \Big( \frac{\lambda_1}{d} \Big)^b \Big( \frac{\mu}{\lambda} \Big)^{b+\frac{1}{a} -1 - \epsilon}   \| \phi \|_{Y^{+, 1}_{\mu, N}} \| \psi_{\lambda_1, N_1} \|_{V^2_{\pm_1,M }} \| \varphi_{\lambda_2, N_2} \|_{V^2_{\pm_2,M}}
	\end{align*}
which then gives (\ref{eqn - high mod, high high main A0 weak}) if we choose $\epsilon$ sufficiently small as $\frac{1}{a} > \frac{1}{2}$ and $b+\frac{1}{a} - 1 = 4(\frac{1}{a} - \frac{1}{2})$ (here $a, b$ are as in the definition of the $Y^{\pm, m}_\lambda$ norm).

We now turn to the estimates for $A_1$ and $A_2$. By symmetry, it is enough to consider the $A_1$ term. After decomposing into cubes of size $\mu$ and applying the $L^4_{t, x}$ Strichartz estimate, we obtain
	\begin{align*}
	\Big| \int_{\RR^{1+3}} A_1 dt dx \Big| &\lesa  \sum_{\substack{q, q' \in Q_\mu \\ |q - q'| \lesa \mu}}
	\| \phi_{\mu, N} \|_{L^4_{t, x}}
	\| \mc{C}^{\pm_1}_d P_q \psi_{\lambda_1, N_1} \|_{L^2_{t, x}}
	\| P_{q'} \varphi_{\lambda_2, N_2} \|_{L^4_{t, x}} \\
		&\lesa \mu^\frac{1}{2} \Big( \frac{\mu}{\lambda_1} \Big)^{\frac{1}{4}-\epsilon} \Big( \frac{\lambda_1}{d} \Big)^\frac{1}{2} \| \phi_{\mu, N} \|_{V^2_{+,1}}  \| \psi_{\lambda_1, N_1} \|_{V^2_{\pm_1,M }} \| \varphi_{\lambda_2, N_2} \|_{V^2_{\pm_2,M}}.
	\end{align*}
Summing up over $d \gtrsim \lambda_1$ and choosing $\epsilon$ sufficiently small, then gives (\ref{eqn - high mod, high high main}).

\medskip

{\bf High modulation II: $\mu \gg \min\{ \lambda_1, \lambda_2\}$ and $d \gtrsim \mu$.} Our goal is to prove the bound
  \begin{equation}\label{eqn - high mod high low main}
	\begin{split} \sum_{d \gtrsim  \mu } \Big| \int_{\RR^{1+3}} A_0 dx dt \Big| + &\Big| \int_{\RR^{1+3}} A_1 dx dt \Big| + \Big| \int_{\RR^{1+3}} A_2 dx dt \Big|\\
              &\lesa  \mu^{\frac{1}{2}}  \Big( \frac{\min\{\lambda_1, \lambda_2\}}{\mu} \Big)^\frac{1}{4} \| \phi_{\mu, N} \|_{V^2_{+,1}}  \| \psi_{\lambda_1, N_1} \|_{V^2_{\pm_1,M }} \| \varphi_{\lambda_2, N_2} \|_{V^2_{\pm_2,M}}.\end{split}
    \end{equation}
As the estimate is essentially symmetric in $\lambda_1$ and $\lambda_2$, we may assume that $\lambda_1 \g \lambda_2$. The bound for $A_0$ follows by decomposing into cubes of size $\lambda_2$ and applying the standard $L^4_{t, x}$ Strichartz estimate to obtain
	\begin{align*}
		\Big| \int_{\RR^{3+1}} A_0 dt dx \Big| &\lesa \sum_{\substack{q, q'' \in Q_{\lambda_2} \\ |q - q''| \lesa \lambda_2}}
			\|C_d P_{q''} \phi_{\mu, N} \|_{L^2_{t, x}}
	\|P_q \psi_{\lambda_1, N_1} \|_{L^4_{t, x}}
	\| \varphi_{\lambda_2, N_2} \|_{L^4_{t, x}}  \\
	&\lesa \mu^\frac{1}{2} \Big( \frac{\lambda_2}{\mu} \Big)^{\frac{3}{4}-\epsilon} \Big( \frac{\mu}{d} \Big)^\frac{1}{2} \| \phi_{\mu, N} \|_{V^2_{+,1}}  \| \psi_{\lambda_1, N_1} \|_{V^2_{\pm_1,M }} \| \varphi_{\lambda_2, N_2} \|_{V^2_{\pm_2,M}}
	\end{align*}
which easily gives (\ref{eqn - high mod high low main}) for the $A_0$ term, provided we choose $\epsilon$ sufficiently small. The proof for the $A_1$ term is identical (as we do not exploit any null structure here). On the other hand, to estimate the $A_2$ term, we again decompose into cubes of size $\lambda_2$ and apply the $L^4_{t, x}$ Strichartz estimate to deduce that
	\begin{align*}
		\Big| \int_{\RR^{3+1}} A_2 dt dx \Big| &\lesa  \sum_{\substack{q, q'' \in Q_{\lambda_2} \\ |q - q''| \lesa \lambda_2}}
			\| P_{q''} \phi_{\mu, N} \|_{L^4_{t, x}}
	\|P_q \psi_{\lambda_1, N_1} \|_{L^4_{t, x}}
	\|\mc{C}^{\pm_2}_d \varphi_{\lambda_2, N_2} \|_{L^2_{t, x}}  \\
	&\lesa \mu^\frac{1}{2} \Big( \frac{\lambda_2}{\mu} \Big)^{\frac{1}{2}-\epsilon} \Big( \frac{\mu}{d} \Big)^\frac{1}{2}\| \phi_{\mu, N} \|_{V^2_{+,1}}  \| \psi_{\lambda_1, N_1} \|_{V^2_{\pm_1,M }} \| \varphi_{\lambda_2, N_2} \|_{V^2_{\pm_2,M}}.
	\end{align*}
Therefore (\ref{eqn - high mod high low main}) follows. This completes the proof of Theorem \ref{thm - trilinear freq loc integral}.
\end{proof}


\subsection{Proof of Theorem \ref{thm - bilinear F freq loc endpoint}}\label{subsect:proof2}

We begin with the proof of (\ref{eqn - thm bilinear F freq loc endpoint - main ineq I}). An application of the energy inequality in Lemma \ref{lem - energy ineq for F norms} gives
    \begin{align*} \big\| P_{\lambda_1} H_{N_1} \Pi_{\pm_1} \mc{I}^{\pm_1}_M\big[ \phi_{\mu, N} \gamma^0 \Pi_{\pm_2} &\varphi_{\lambda_2, N_2}\big] \big\|_{V^2_{\pm_1,M}}\\
        &\lesa \sup_{\| \psi_{\lambda_1, N_1} \|_{V^2_{\pm_1,M}} \lesa 1} \Big| \int_{\RR^{1+3}} \phi_{\mu, N} ( \Pi_{\pm_1} \psi_{\lambda_1, N_1})^\dagger \gamma^0 \Pi_{\pm_2} \varphi_{\lambda_2, N_2} dx dt \Big|.\end{align*}
 Therefore an application of (\ref{eqn - trilinear freq loc integral I})  in Theorem \ref{thm - trilinear freq loc integral} implies that
     \begin{equation}\label{eqn - thm bilinear F freq loc endpoint - F bound}
     \begin{split}
        \big\| P_{\lambda_1} H_{N_1} \Pi_{\pm_1} \mc{I}^{\pm_1}_M\big[ \phi_{\mu, N} \gamma^0 \Pi_{\pm_2} &\varphi_{\lambda_2, N_2}\big] \big\|_{V^2_{\pm_1,M}} \\
        &\lesa \mu^\frac{1}{2} (\min\{ N, N_2\})^\frac{\sigma}{4} \mb{B}_{\min\{ \frac{\sigma}{32}, \frac{1}{2a} - \frac{1}{4}\}} \| \phi \|_{F^{+, 1}_{\mu, N}} \| \varphi \|_{F^{\pm_2, M}_{\lambda_2, N_2}}
        \end{split}
     \end{equation}
 which gives the required bound (\ref{eqn - thm bilinear F freq loc endpoint - main ineq I}) for the $F^{\pm_1, M}_{\lambda_1, N_1}$ component of the norm. To complete the proof of (\ref{eqn - thm bilinear F freq loc endpoint - main ineq I}), it remains show that there exists $\epsilon>0$ such that
 	\begin{equation}\label{eqn - thm bilinear F freq loc endpoint - Y bound}
 		\big\| \Pi_{\pm_1} \mc{I}^{\pm_1}_M\big[ \phi_{\mu, N} \gamma^0 \Pi_{\pm_2} \varphi_{\lambda_2, N_2}\big] \big\|_{Y^{\pm_1, M}_{\lambda_1, N_1}} \lesa \mu^\frac{1}{2} (\min\{ N, N_2\})^\frac{\sigma}{2} \mb{B}_{\epsilon} \| \phi \|_{F^{+, 1}_{\mu, N}} \| \varphi \|_{F^{\pm_2, M}_{\lambda_2, N_2}}.
 	\end{equation}
  To this end, we consider separately the cases $\lambda_1 \ll \lambda_2$ and $\lambda_1 \gtrsim \lambda_2$. In the former region, note that an application of (\ref{eqn - trilinear freq loc integral III}) in Theorem \ref{thm - trilinear freq loc integral} together with the energy inequality Lemma \ref{lem - energy ineq for F norms}, and the $L^2_{t, x}$ bound in Lemma \ref{lem - norm controls Xsb}, gives
   \begin{align}  \big\| P_{\lambda_1} H_{N_1} \mc{C}^{\pm_1}_d \mc{I}^{\pm_1}_M\big[ &\phi_{\mu, N} \gamma^0 \Pi_{\pm_2} \varphi_{\lambda_2, N_2}\big] \big\|_{L^2_{t, x}} \notag \\
    &\lesa  d^{-\frac{1}{2}}\big\| P_{\lambda_1} H_{N_1} \Pi_{\pm_1} \mc{I}^{\pm_1}_M\big[ \phi_{\mu, N} \gamma^0 \Pi_{\pm_2} \varphi_{\lambda_2, N_2}\big] \big\|_{V^2_{\pm_1,M}} \notag \\
    &\lesa  d^{-\frac{1}{2}} \mu^\frac{1}{2} (\min\{ N, N_2\})^\frac{\sigma}{4} \Big(\frac{\lambda_1}{\lambda_2}\Big)^{\frac{\sigma}{32}} \|\phi_{\mu, N} \|_{V^2_{+,1}} \| \varphi_{\lambda_2, N_2} \|_{V^2_{\pm_2,M}}. \label{eqn - thm bilinear F freq loc endpoint - L2 bound}
    \end{align}
 On the other hand, since we are localised away from the hyperboloid we have by (\ref{eqn - Cd mult and duhamel int}) together with Lemma \ref{lem - wave strichartz}
    \begin{align} \big\| P_{\lambda_1} H_{N_1} \mc{C}^{\pm_1}_d \mc{I}^{\pm_1}_M\big[ \phi_{\mu, N} & \gamma^0 \Pi_{\pm_2} \varphi_{\lambda_2, N_2}\big] \big\|_{L^\frac{3}{2}_t L^2_x} \notag\\
     &\lesa  d^{-1} \| P_{\lambda_1}(\phi_{\mu, N} \gamma^0 \Pi_{\pm_2} \varphi_{\lambda_2, N_2}) \|_{L^\frac{3}{2}_t L^2_x} \notag \\
     &\lesa d^{-1}  \| \phi_{\mu, N}\|_{L^4_{t, x}}\|  \varphi_{\lambda_2, N_2}\|_{L^{\frac{12}{5}}_t L^4_x} \notag\\
     &\lesa d^{-1} \mu^\frac{1}{2} \lambda_2^{\frac{1}{3}} N_2  \| \phi_{\mu, N} \|_{V^2_{+,1}} \| \varphi_{\lambda_2, N_2}\|_{V^2_{\pm_2,M}}.  \label{eqn - thm bilinear freq loc endpoint - Y intermediate bound}
    \end{align}
Repeating this argument but instead putting $\phi \in L^{\frac{12}{5}}_t L^4_x$ and $\varphi \in L^4_{t, x}$ we deduce that, since $\lambda_1 \ll \lambda_2\approx \mu$,
    \begin{equation}\label{eqn - thm bilinear F freq loc endpoint - weak Y bound} \begin{split}
           d \lambda_1^{-\frac{1}{3}} \big\| P_{\lambda_1} H_{N_1} \mc{C}^{\pm_1}_d \mc{I}^{\pm_1}_M\big[ \phi_{\mu, N} &\gamma^0 \Pi_{\pm_2} \varphi_{\lambda_2, N_2}\big] \big\|_{L^\frac{3}{2}_t L^2_x} \\
             &\lesa \mu^\frac{1}{2} \min\{ N, N_2\} \Big( \frac{ \lambda_2 }{\lambda_1}\Big)^{ \frac{1}{3} } \|\phi_{\mu, N} \|_{V^2_{+,1}} \| \varphi_{\lambda_2, N_2} \|_{V^2_{\pm_2,M}}. \end{split}
    \end{equation}
 Note that this bound is far to weak to be useful on its own, as we have $\lambda_1 \ll \lambda_2$. On the other hand, if we combine  (\ref{eqn - thm bilinear F freq loc endpoint - L2 bound}) and (\ref{eqn - thm bilinear F freq loc endpoint - weak Y bound}), and use the convexity of the $L^p_t$ spaces, we deduce that if we let $0<\theta<1$ be given by $\frac{1}{a} = \frac{ 2 \theta}{3} + \frac{1-\theta}{2}$, then, as this forces $b=\frac{1+\theta}{2}$, we deduce that
    \begin{align*}
    \lambda_1^{\frac{1}{a} - b} d^{b}  \big\| P_{\lambda_1}& H_{N_1} \mc{C}^{\pm_1}_d \mc{I}^{\pm_1}_M\big[ \phi_{\mu, N} \gamma^0 \Pi_{\pm_2} \varphi_{\lambda_2, N_2}\big] \big\|_{L^{a}_t L^2_x} \\
    	&\lesa \Big( d \lambda_1^{-\frac{1}{3}} \big\| P_{\lambda_1} H_{N_1} \mc{C}^{\pm_1}_d \mc{I}^{\pm_1}_M\big[ \phi_{\mu, N} \gamma^0 \Pi_{\pm_2} \varphi_{\lambda_2, N_2}\big] \big\|_{L^\frac{3}{2}_t L^2_x} \Big)^\theta \\
    	&\qquad \qquad \qquad \qquad \times \Big( d^{\frac{1}{2}} \big\| P_{\lambda_1} H_{N_1} \mc{C}^{\pm_1}_d \mc{I}^{\pm_1}_M\big[ \phi_{\mu, N} \gamma^0 \Pi_{\pm_2} \varphi_{\lambda_2, N_2}\big] \big\|_{L^2_{t, x}}\Big)^{1-\theta}\\
        &\lesa  \mu^\frac{1}{2}  \big(\min\{ N, N_2\}\big)^{\theta + \frac{\sigma }{4}(1-\theta)} \Big( \frac{ \lambda_1 }{\lambda_2}\Big)^{ \frac{\sigma}{32}(1-\theta) - \frac{1}{3}\theta}  \|\phi_{\mu, N} \|_{V^2_{+,1}} \| \varphi_{\lambda_2, N_2} \|_{V^2_{\pm_2,M}}.
    \end{align*}
Since $\frac{1}{2} < \frac{1}{a} < \frac{1}{2} + \frac{\sigma}{1000}$, it is easy enough to check that $\frac{\sigma}{32}(1-\theta) - \frac{1}{3}\theta>0$, and hence (\ref{eqn - thm bilinear F freq loc endpoint - Y bound}) holds when $\lambda_1 \ll \lambda_2$. We now consider the case $\lambda_1 \gtrsim \lambda_2$. The proof is similar to the previous case, the main difference is that we need a more refined version of the bound (\ref{eqn - thm bilinear F freq loc endpoint - weak Y bound}). To this end, by decomposing $\varphi$ into cubes of size $\min\{\mu, \lambda_2\}$, we deduce that by Lemma \ref{lem - wave strichartz} and Lemma \ref{lem - square sum loss}, for every $\epsilon'>0$
	\begin{align*}
	  \big\| P_{\lambda_1} H_{N_1} \mc{C}^{\pm_1}_d \mc{I}^{\pm_1}_M\big[ \phi_{\mu, N} & \gamma^0 \Pi_{\pm_2} \varphi_{\lambda_2, N_2}\big] \big\|_{L^\frac{3}{2}_t L^2_x} \\
     &\lesa  d^{-1} \Big\|  \| \phi_{\mu, N} \gamma^0 \Pi_{\pm_2} P_q \varphi_{\lambda_2, N_2}\|_{L^2_x}^2 \Big)^{\frac{1}{2}}  \Big\|_{L^\frac{3}{2}_t}  \\
     &\lesa d^{-1}  \|\phi_{\mu, N} \|_{L^{\frac{12}{5}}_t L^4_x} \Big( \sum_{q \in Q_{\min\{\mu, \lambda_2\}}}\| P_q \varphi_{\lambda_2, N_2} \|_{L^4_{t, x}}^2 \Big)^\frac{1}{2}\\
     &\lesa d^{-1} \mu^{\frac{1}{3}} N (\min\{ \mu, \lambda_2\})^{\frac{1}{4}-\epsilon'} \lambda_2^{\frac{1}{4} + \epsilon'} \|\phi_{\mu, N} \|_{V^2_{+,1}} \| \varphi_{\lambda_2, N_2} \|_{V^2_{\pm_2,M}}.
	\end{align*}
Since $(\min\{\mu, \lambda_2\})^{\frac{1}{4}-\epsilon'} \les \mu^\frac{1}{6} \lambda_2^{\frac{1}{4} - \frac{1}{6}+\epsilon'}$ (for $\epsilon'$ sufficiently small) and $\lambda_2 \lesa \lambda_1$, by using the bound (\ref{eqn - thm bilinear freq loc endpoint - Y intermediate bound}), we deduce that
	 \begin{equation}\label{eqn - thm bilinear F freq loc endpoint - weak Y bound I} \begin{split}
           d \lambda_1^{-\frac{1}{3}} \big\| P_{\lambda_1} H_{N_1} \mc{C}^{\pm_1}_d \mc{I}^{\pm_1}_M\big[ \phi_{\mu, N} &\gamma^0 \Pi_{\pm_2} \varphi_{\lambda_2, N_2}\big] \big\|_{L^\frac{3}{2}_t L^2_x} \\
             &\lesa \mu^\frac{1}{2} \min\{ N, N_2\} \|\phi_{\mu, N} \|_{V^2_{+,1}} \| \varphi_{\lambda_2, N_2} \|_{V^2_{\pm_2,M}}. \end{split}
     \end{equation}
Note that, unlike the bound (\ref{eqn - thm bilinear F freq loc endpoint - weak Y bound I}), we have no high frequency loss here. As in the case $\lambda_1 \ll \lambda_2$, we now combine the bound (\ref{eqn - thm bilinear F freq loc endpoint - F bound}) with (\ref{eqn - thm bilinear F freq loc endpoint - weak Y bound I}), and deduce by the convexity of the $L^p_t$ norm and Lemma \ref{lem - norm controls Xsb}, that
	\begin{align*}
    \lambda_1^{\frac{1}{a} - b} d^{b}  \big\| P_{\lambda_1}& H_{N_1} \mc{C}^{\pm_1}_d \mc{I}^{\pm_1}_M\big[ \phi_{\mu, N} \gamma^0 \Pi_{\pm_2} \varphi_{\lambda_2, N_2}\big] \big\|_{L^{a}_t L^2_x} \\
    	&\lesa \Big( d \lambda_1^{-\frac{1}{3}} \big\| P_{\lambda_1} H_{N_1} \mc{C}^{\pm_1}_d \mc{I}^{\pm_1}_M\big[ \phi_{\mu, N} \gamma^0 \Pi_{\pm_2} \varphi_{\lambda_2, N_2}\big] \big\|_{L^\frac{3}{2}_t L^2_x} \Big)^\theta \\
    	&\qquad \qquad \qquad \qquad \times \Big( d^{\frac{1}{2}} \big\| P_{\lambda_1} H_{N_1} \mc{C}^{\pm_1}_d \mc{I}^{\pm_1}_M\big[ \phi_{\mu, N} \gamma^0 \Pi_{\pm_2} \varphi_{\lambda_2, N_2}\big] \big\|_{L^2_{t, x}}\Big)^{1-\theta}\\
        &\lesa  \mu^\frac{1}{2}  \big(\min\{ N, N_2\}\big)^{\theta + \frac{\sigma }{4}(1-\theta)} \mb{B}_{\min\{\frac{\sigma}{32}, \frac{1}{2a} - \frac{1}{4}\}}^{1-\theta}  \| \phi \|_{F^{+, 1}_{\mu, N}} \| \varphi \|_{F^{\pm_2, M}_{\lambda_2, N_2}}.
    \end{align*}
Since $0<\theta \ll \sigma$, we obtain (\ref{eqn - thm bilinear F freq loc endpoint - Y bound}). Therefore, the bound (\ref{eqn - thm bilinear F freq loc endpoint - main ineq I}) follows.

We now turn to the proof of the second inequality (\ref{eqn - thm bilinear F freq loc endpoint - main ineq II}). The argument is similar to the proof of (\ref{eqn - thm bilinear F freq loc endpoint - main ineq I})
so we will be brief. An application of the energy inequality in Lemma \ref{lem - energy ineq for F norms} together with (\ref{eqn - trilinear freq loc integral II}) in Theorem \ref{thm - trilinear freq loc integral} implies that
     \begin{equation}\label{eqn - thm bilinear F freq loc endpoint - F bound wave} \big\| P_\mu H_N \mc{I}^{+}_1\big[  (\Pi_{\pm_1} \psi_{\lambda_1, N_1})^\dagger \gamma^0 \Pi_{\pm_2} \varphi_{\lambda_2, N_2}\big] \big\|_{V^2_{+,1}}\lesa \mu^\frac{1}{2} (\min\{ N_1, N_2\})^\frac{\sigma}{4} \mb{B}_{\min\{ \frac{\sigma}{32}, \frac{1}{2a} - \frac{1}{4}\}} \| \psi \|_{F^{\pm_1, M}_{\lambda_1, N_1}} \| \varphi \|_{F^{\pm_2, M}_{\lambda_2, N_2}}. \end{equation}
Therefore it only remains to prove that there exists  $\epsilon>0$ such that
 	\begin{equation}\label{eqn - thm bilinear F freq loc endpoint - Y bound wave}
 		\big\|  \mc{I}^{+}_1\big[  (\Pi_{\pm_1} \psi_{\lambda_1, N_1})^\dagger \gamma^0 \Pi_{\pm_2} \varphi_{\lambda_2, N_2}\big] \big\|_{Y^{+, 1}_{\mu, N}} \lesa \mu^\frac{1}{2} (\min\{ N, N_2\})^\frac{\sigma}{2} \mb{B}_{\epsilon} \| \psi \|_{F^{\pm_1, M}_{\lambda_1, N_1}} \| \varphi \|_{F^{\pm_2, M}_{\lambda_2, N_2}}.
 	\end{equation}
Similar to the proof of (\ref{eqn - thm bilinear F freq loc endpoint - Y bound}), we consider separately the cases $\mu \ll \lambda_1$ and $\mu \gtrsim \lambda_1$. In the former case, as in (\ref{eqn - thm bilinear F freq loc endpoint - weak Y bound I}),  since we are localised away from the hyperboloid we have by (\ref{eqn - Cd mult and duhamel int}) together with Lemma \ref{lem - wave strichartz}
    \begin{align} \big\| P_{\mu} H_{N} \mc{C}^{\pm_1}_d &\mc{I}^{+}_1\big[  (\Pi_{\pm_1} \psi_{\lambda_1, N_1})^\dagger \gamma^0 \Pi_{\pm_2} \varphi_{\lambda_2, N_2}\big]\big\|_{L^\frac{3}{2}_t L^2_x}\notag \\
      &\lesa  d^{-1} \| P_{\mu}((\Pi_{\pm_1} \psi_{\lambda_1, N_1})^\dagger \gamma^0 \Pi_{\pm_2} \varphi_{\lambda_2, N_2}) \|_{L^\frac{3}{2}_t L^2_x} \notag \\
     &\lesa d^{-1} (\min\{\lambda_1, \lambda_2\})^{\frac{1}{3}} (\max\{ \lambda_1, \lambda_2\})^\frac{1}{2} \min\{N_1, N_2\}  \| \psi_{\lambda_1, N_1} \|_{V^2_{\pm_1,M}} \| \varphi_{\lambda_2, N_2}\|_{V^2_{\pm_2,M}}.  \label{eqn - thm bilinear freq loc endpoint - Y intermediate bound wave}
    \end{align}
Since $\lambda_1 \approx \lambda_2$, we can replace the $\max$ and $\min$ in (\ref{eqn - thm bilinear freq loc endpoint - Y intermediate bound wave}) with $\lambda_1^{\frac{1}{3} + \frac{1}{2}}$. If we now combine
(\ref{eqn - thm bilinear freq loc endpoint - Y intermediate bound wave}) with the energy inequality in Lemma \ref{lem - energy ineq for F norms},  the bound (\ref{eqn - trilinear freq loc integral IV}) in Theorem \ref{thm - trilinear freq loc integral}, and Lemma \ref{lem - norm controls Xsb}, we deduce that by the convexity of the $L^p_t$ spaces that
    \begin{align*}
    \mu^{\frac{1}{a} - b} d^{b}  \big\| P_{\mu}& H_{N} C^+_d\mc{I}^{+}_1\big[  (\Pi_{\pm_1} \psi_{\lambda_1, N_1})^\dagger \gamma^0 \Pi_{\pm_2} \varphi_{\lambda_2, N_2}\big]\big\|_{L^{a}_t L^2_x} \\
    	&\lesa \Big( d \mu^{-\frac{1}{3}} \big\| P_{\mu} C^+_d \mc{I}^{+}_1\big[  (\Pi_{\pm_1} \psi_{\lambda_1, N_1})^\dagger \gamma^0 \Pi_{\pm_2} \varphi_{\lambda_2, N_2} \big]\big\|_{L^\frac{3}{2}_t L^2_x} \Big)^\theta \\
    	&\qquad \qquad \qquad \qquad \times \Big( d^{\frac{1}{2}} \big\| P_{\mu} C^+_d \mc{I}^{+}_1\big[  (\Pi_{\pm_1} \psi_{\lambda_1, N_1})^\dagger \gamma^0 \Pi_{\pm_2} \varphi_{\lambda_2, N_2} \big] \big\|_{L^2_{t, x}}\Big)^{1-\theta}\\
        &\lesa  \mu^\frac{1}{2}  \big(\min\{ N_1, N_2\}\big)^{\theta + \frac{\sigma }{4}(1-\theta)} \Big( \frac{ \lambda_1 }{\mu}\Big)^{ \frac{\sigma}{32}(1-\theta) - \frac{5}{6}\theta}  \| \psi_{\lambda_1, N_1} \|_{V^2_{\pm_1,M}} \| \varphi_{\lambda_2, N_2}\|_{V^2_{\pm_2,M}}
    \end{align*}
where as previously, we have $\frac{1}{a} = \frac{2 \theta}{3} + \frac{1-\theta}{2}$ (which implies that $b=\frac{1+\theta}{2}$). Since $\frac{1}{2} < \frac{1}{a} < \frac{1}{2} + \frac{\sigma}{1000}$, it is easy enough to check that $\frac{\sigma}{32}(1-\theta) - \frac{5}{6}\theta>0$, and hence (\ref{eqn - thm bilinear F freq loc endpoint - Y bound wave}) holds when $\mu \ll \lambda_1$. We now consider the case $\mu \gtrsim \lambda_1$. Since we now have $(\min\{\lambda_1, \lambda_2\})^{\frac{1}{3}} (\max\{ \lambda_1, \lambda_2\})^\frac{1}{2} \lesa \mu^{\frac{1}{3} + \frac{1}{2}}$, an application of (\ref{eqn - thm bilinear freq loc endpoint - Y intermediate bound wave}), together with (\ref{eqn - thm bilinear F freq loc endpoint - F bound wave}), Lemma \ref{lem - norm controls Xsb} gives
   \begin{align*}
    \mu^{\frac{1}{a} - b} d^{b}  \big\| P_{\mu}& H_{N} C^+_d\mc{I}^{+}_1\big[  (\Pi_{\pm_1} \psi_{\lambda_1, N_1})^\dagger \gamma^0 \Pi_{\pm_2} \varphi_{\lambda_2, N_2}\big]\big\|_{L^{a}_t L^2_x} \\
    	&\lesa \Big( d \mu^{-\frac{1}{3}} \big\| P_{\mu} C^+_d \mc{I}^{+}_1\big[  (\Pi_{\pm_1} \psi_{\lambda_1, N_1})^\dagger \gamma^0 \Pi_{\pm_2} \varphi_{\lambda_2, N_2} \big]\big\|_{L^\frac{3}{2}_t L^2_x} \Big)^\theta \\
    	&\qquad \qquad \qquad \qquad \times \Big( d^{\frac{1}{2}} \big\| P_{\mu} C^+_d \mc{I}^{+}_1\big[  (\Pi_{\pm_1} \psi_{\lambda_1, N_1})^\dagger \gamma^0 \Pi_{\pm_2} \varphi_{\lambda_2, N_2} \big] \big\|_{L^2_{t, x}}\Big)^{1-\theta}\\
        &\lesa  \mu^\frac{1}{2}  \big(\min\{ N_1, N_2\}\big)^{\theta + \frac{\sigma }{4}(1-\theta)} \mb{B}^{1-\theta}_{\min\{\frac{\sigma}{32}, \frac{1}{2a} - \frac{1}{4}\}}  \| \psi_{\lambda_1, N_1} \|_{V^2_{\pm_1,M}} \| \varphi_{\lambda_2, N_2}\|_{V^2_{\pm_2,M}}
    \end{align*}
Since $0<\theta \ll \sigma$ and $\frac{1}{a}>\frac{1}{2}$, we obtain (\ref{eqn - thm bilinear F freq loc endpoint - Y bound wave}). Therefore, the bound (\ref{eqn - thm bilinear F freq loc endpoint - main ineq I}) follows. This completes the proof of Theorem \ref{thm - bilinear F freq loc endpoint}.

    \bibliographystyle{amsplain} \bibliography{bilinear}
  \end{document}